\documentclass[a4paper,12pt]{amsart}

\usepackage{fullpage}
\usepackage{amsmath,amssymb,graphicx,psfrag}
\usepackage[T1]{fontenc}
\usepackage{amsmath,amssymb,ifthen,moreverb}
\usepackage{color}

\newcommand{\A}{\mathbb{A}}

\newcommand{\N}{\mathbb{N}}
\newcommand{\R}{\mathbb{R}}
\newcommand{\OO}{\mathcal{O}}
\newcommand{\RR}{\mathcal{R}}
\newcommand{\TT}{\mathcal{T}}

\newcommand{\XX}{\mathcal{X}}

\newcommand{\dual}[3][]{#1\langle#2\,,\,#3#1\rangle}

\newcommand{\norm}[3][]{#1\|#2#1\|_{#3}}
\newcommand{\diam}{{\rm diam}}
\newcommand{\set}[3][\big]{#1\{#2\,:\,#3#1\}}

\def\K{\mathbb K}
\def\C{\mathbb C}
\def\HH{\mathcal H}
\def\Re{\operatorname{Re}}
\def\product#1#2{(#1\,,\,#2)}
\def\reff#1#2{\!\stackrel{\eqref{#1}}{#2}\!}

\def\UU{\mathcal U}
\def\MM{\mathcal M}
\def\refine{\operatorname{refine}}
\def\Cmark{C_{\rm mark}}
\def\Cstab{C_{\rm stb}}
\def\Cred{C_{\rm red}}
\def\Crel{C_{\rm rel}}
\def\Crelexact{C_{\rm rel}^\star}
\def\qred{q_{\rm red}}
\def\T{\mathbb{T}}
\def\Cest{C_{\rm est}}
\def\qest{q_{\rm est}}

\def\qlin{q_{\rm lin}}

\def\Cdrel{C_{\rm drel}^\star}
\def\Cpic{C_{\rm pic}}

\def\Cmon{C_{\rm mon}}
\def\Copt{C_{\rm opt}}
\def\Cmesh{C_{\rm mesh}}
\def\Clin{C_{\rm lin}}
\def\Cson{C_{\rm son}}

\def\osc{{\rm osc}}

\def\plus{\bullet}

\newcommand{\inexact}{{}}
\newcommand{\exact}{\star}

\def\pic{{\rm Pic}}

\def\Cwork{C_{\rm work}}
\def\tildeCwork{\widetilde{C}_{\rm work}}

\def\diver{\operatorname{div}}
\def\GG{\mathbf{\mu}}

\def\IntSet#1{\int_{#1}}
\def\Int#1#2{\int_{#1}^{#2}}
\def\d#1{\mathop{\mathrm{d}#1}}
\def\zzeta{\mathbf{\zeta}}
\def\kkappa{\mathbf{\kappa}}

\newtheorem{lemma}{Lemma}
\newtheorem{theorem}[lemma]{Theorem}
\newtheorem{proposition}[lemma]{Proposition}

\newtheorem{remark}[lemma]{Remark}
\newtheorem{algorithm}[lemma]{Algorithm}

\renewcommand{\subsection}[1]{\refstepcounter{subsection}\medskip{\bf\thesubsection.~#1.}}

\usepackage{fancyhdr}
\advance\footskip0.4cm
\advance\textheight-0.4cm
\calclayout
\title{Rate optimal adaptive FEM with inexact solver\\for nonlinear operators}

\author{Gregor Gantner}
\author{Alexander Haberl}
\author{Dirk Praetorius}
\author{Bernhard Stiftner}

\address{TU Wien, Institute for Analysis and Scientific Computing, Wiedner Hauptstr.~8--10/E101/4, 1040 Wien, Austria}
\email{\{gregor.gantner\,,\,dirk.praetorius\,,\,bernhard.stiftner\}@asc.tuwien.ac.at}
\email{alexander.haberl@asc.tuwien.ac.at\quad\rm(corresponding author)}

\date{\today}
\begin{document}
\begin{abstract}
We prove convergence with optimal algebraic rates for an adaptive finite element method for nonlinear equations with strongly monotone operator. 
Unlike prior works, our analysis also includes the iterative and inexact solution of the arising nonlinear systems by means of the Picard iteration. 
Using nested iteration, we prove, in particular, that the number of of Picard iterations is uniformly bounded 
in generic cases, and the overall computational cost is (almost) optimal. Numerical experiments confirm the theoretical results.
\end{abstract}
\maketitle

\section{Introduction}
\label{section:introduction}

In recent years, the analysis of convergence and optimal convergence behaviour of adaptive finite element methods has matured. 
We refer to the seminal works~\cite{doerfler,mns,bdd,stevenson07,ckns,ffp} for some milestones for linear elliptic equations,~\cite{veeser,dk08,bdk,gmz} for non-linear problems, and~\cite{axioms} for some general abstract framework. 
While the interplay of adaptive mesh-refinement, optimal convergence rates, and inexact solvers has already been addressed and analyzed, e.g., in~\cite{stevenson07,MR3068564,MR3095916,axioms} for linear PDEs and in~\cite{cg13} for eigenvalue problems, the influence of inexact solvers for nonlinear equations has not been analyzed yet. 
The work~\cite{gmz11} considers adaptive mesh-refinement in combination with a {K}a\v canov-type iterative solver for strongly monotone operators. 
In the spirit of~\cite{msv,siebert}, the focus is on a plain convergence result of the overall strategy, while the proof of optimal convergence rates remains open.

On the other hand, there is a rich body on {\sl a~posteriori} error estimation which also includes the iterative and inexact solution for nonlinear problems; see, e.g., \cite{ev13}.
The present work aims to close the gap between numerical analysis (e.g., \cite{axioms}) and empirical evidence of optimal convergence rates (e.g., \cite{gmz11,ev13})
by analyzing an adaptive algorithm from~\cite{cw2015}.

We consider a nonlinear elliptic equation in the variational formulation
\begin{align}\label{eq:strongform}
 \dual{Au^\exact}{v} = \dual{F}{v} 
 \quad\text{for all }v\in\HH,
\end{align}
where $\HH$ is a Hilbert space over $\K\in\{\R,\C\}$ with dual space $\HH^*$ and corresponding duality brackets $\dual\cdot\cdot$ and $F\in\HH^*$. 
We treat the variational formulation in an abstract framework.
We suppose that the operator $A:\HH\to\HH^*$ satisfies the following conditions:
\begin{enumerate}
 \renewcommand{\theenumi}{O\arabic{enumi}}
 \bf
 \item\label{axiom:monotone}
 \rm
 \textbf{$\boldsymbol A$ is strongly monotone:} There exists $\alpha>0$ such that
  \begin{align*}
  \alpha\,\norm{w-v}\HH^2 \le \Re\,\dual{Aw-Av}{w-v}\quad\text{for all $v,w\in\HH$}.
  \end{align*}
  \bf
  \item\label{axiom:lipschitz}
  \rm
  \textbf{$\boldsymbol A$ is Lipschitz continuous:} There exists $L>0$ such that
  \begin{align*}
  	\norm{Aw-Av}{\HH^*} \le L\,\norm{w-v}\HH\quad\text{for all $v,w\in\HH$}.
  \end{align*}
  \bf
  \item\label{axiom:potential}
  \rm
  \textbf{$\boldsymbol{A}$ has a potential:} There exists a G\^ateaux differentiable
  function $P:\HH\to\mathbb{K}$ such that its derivative ${\rm d}P:\HH\to\HH^*$ coincides with $A$, i.e., for all $v,w\in\HH$, it holds that
  \begin{align}\label{eq:gateaux}
  \dual{Aw}{v}=\dual{{\rm d} P(w)}{v}=\lim_{\substack{r\to 0\\r\in\R}}\frac{P(w+rv)-P(w)}{r}.
  \end{align}
\end{enumerate} 
We note that~\eqref{axiom:monotone}--\eqref{axiom:lipschitz} guarantee that there exists a unique solution $u^\exact \in\HH$ to~\eqref{eq:strongform}. The latter is the limit of any sequence of Picard iterates $u^{n+1} = \Phi(u^n)$ for all $n\in\N_0$ with arbitrary initial guess $u^0\in\HH$, where the nonlinear mapping $\Phi:\HH\to\HH$ is a contraction (see Section~\ref{section:banach} for details). The additional assumption~\eqref{axiom:potential} implies that the nonlinear problem~\eqref{eq:strongform} as well as its later discretization are equivalently stated as energy minimization problems (Lemma~\ref{lemma:energy eq}).
In view of applications, we admit that~\eqref{axiom:monotone}--\eqref{axiom:lipschitz} exclude the $p$-Laplacian~\cite{veeser,dk08,bdk}, 
but cover the same problem class as, e.g.,~\cite{cw2015,gmz11,gmz}; see also~\cite{multiscale} for strongly monotone nonlinearities arising in magnetostatics.

Based on adaptive mesh-refinement of an initial triangulation $\TT_0$, our adaptive algorithm generates a sequence of conforming nested subspaces $\XX_\ell\subseteq\XX_{\ell+1}\subset\HH$, corresponding discrete solutions $u_\ell^\inexact\in\XX_\ell$, and {\sl a~posteriori} error estimators $\eta_\ell(u_\ell^\inexact)$ such that $\norm{u^\exact-u_\ell^\inexact}\HH \le \Crelexact \,\eta_\ell(u_\ell^\inexact)\to0$ as $\ell\to\infty$ at optimal algebraic rate in the sense of certain approximation classes~\cite{ckns,ffp,axioms}. 
While the plain convergence result from~\cite{gmz11} applies to various marking strategies, our convergence analysis follows the concepts from~\cite{ckns,ffp,axioms} and is hence tailored to the D\"orfler marking strategy.

Unlike~\cite{gmz,bdk}, we note that the computed discrete solutions $u_\ell^\inexact \neq u_\ell^\exact$ in general, where $u_\ell^\exact\in\XX_\ell$ is the Galerkin approximation to~\eqref{eq:strongform}, i.e.,
\begin{align}\label{eq':strongform}
 \dual{Au^\exact_\ell}{v_\ell} = \dual{F}{v_\ell}
 \quad\text{for all }v_\ell\in\XX_\ell,
\end{align}
since this discrete nonlinear system cannot be solved exactly in practice. Instead, $u_\ell^\inexact = u_\ell^n := \Phi_\ell(u_\ell^{n-1})$ is a Picard approximate to $u_\ell^\exact$ (see Section~\ref{section:discretization} for the definition of $\Phi_\ell$), and each Picard iterate can be computed by solving one linear system. Our adaptive algorithm steers both, the local mesh-refinement as well as the number of Picard iterations, where we employ \emph{nested iteration} $u_{\ell+1}^0 := u_\ell^\inexact$ to lower the number of Picard steps. 
To shorten the presentation, we did not include the iterative (and inexact) solution of the arising linear systems into the convergence analysis, but assume that these are solved exactly.
We note however that the extended analysis can be done along the lines of~\cite{stevenson07,MR3068564,MR3095916}. Details will appear in~\cite{haberlphd}.

\medskip

{\bf Outline of work.}\quad Section~\ref{section:banach} recalls the well-known proof that~\eqref{eq:strongform} admits a unique solution, since our numerical scheme relies on the Picard mapping $\Phi$ which is at the core of the mathematical argument. Section~\ref{section:discretization} comments on the discrete problem~\eqref{eq':strongform} and introduces the discrete Picard mapping $\Phi_\ell$ which is used for our iterative solver. Section~\ref{section:adaptivealgorithm} states the adaptive strategy (Algorithm~\ref{algorithm}) as well as some abstract properties~\eqref{axiom:stability}--\eqref{axiom:discrete_reliability} of the error estimator in the spirit of~\cite{axioms} which are exploited for the analysis.
We prove that nested iteration essentially yields a bounded number of discrete Picard iterations (Proposition~\ref{proposition:nested_iteration}). Linear convergence of the proposed algorithm in the sense of
\begin{align}
 \exists \Clin>0 \, \exists 0<\qlin<1 \, \forall \ell,k \in\N_0 \quad
 \eta_{\ell+k}(u_{\ell+k}^\inexact) \le \Clin\,\qlin^k\,\eta_\ell(u_\ell^\inexact)
\end{align}
is proved in Section~\ref{section:linear} (Theorem~\ref{theorem:linear_convergence}), where we exploit the property~\eqref{axiom:potential}. Optimal algebraic convergence behavior
in the sense of
\begin{align}
 \forall s>0 \, \exists \Copt>0 \, \forall \ell \in \N_0 \quad
 \eta_\ell(u_\ell^\inexact) \le \Copt\,\norm{u^\exact}{\A_s}\,\big(\#\TT_\ell-\#\TT_0+1)^{-s}
\end{align}
is proved in Section~\ref{section:optimal_rates} (Theorem~\ref{theorem:optimal_rate}), where $\norm{u^\exact}{\A_s}<\infty$ if rate $s$ is possible for the optimal meshes (see Section~\ref{section:approximation_class} for the precise definition of $\norm{u^\exact}{\A_s}$).
As a consequence of the preceding results, we also obtain that the overall computational effort of the adaptive
strategy is (almost) optimal (Theorem~\ref{theorem:optimal_comp}).
 Whereas Algorithm~\ref{algorithm} is indexed by the adaptively generated meshes, Section~\ref{section:full} gives an equivalent formulation (Algorithm~\ref{algorithm:new}), where the Picard iterations are taken into account. Throughout, the analysis of Section~\ref{section:banach}--\ref{section:full} is given in an abstract frame (in the spirit of~\cite{axioms}). 
The final Section~\ref{section:numerics} illustrates our theoretical findings with some numerical experiments, for which the estimator properties~\eqref{axiom:stability}--\eqref{axiom:discrete_reliability} are satisfied.

\bigskip

{\bf General notation.}\quad 
Throughout all statements, all constants as well as their dependencies are explicitly given. In proofs, we may abbreviate the notation by use of the symbol $\lesssim$ which indicates
$\leq$ up to some multiplicative constant which is clear from the context. Moreover, the symbol $\simeq$ states that both estimates $\lesssim$ and $\gtrsim$ hold. 
If (discrete) quantities are related to some triangulation, this is explicitly stated by use of appropriate indices, e.g., $u_\bullet$ is the discrete solution for the triangulation $\TT_\bullet$, $v_\circ$ is a generic discrete function in the discrete space $\XX_\circ$, and $\eta_\ell(\cdot)$ is the error estimator with respect to the triangulation $\TT_\ell$. 

\section{Banach fixpoint theorem}
\label{section:banach}

In this section, we prove that the model problem~\eqref{eq:strongform} admits a unique solution $u\in\HH$.
The proof follows from the Banach fixpoint theorem and relies only on
\eqref{axiom:monotone}--\eqref{axiom:lipschitz}.
Let $\product\cdot\cdot_\HH$ denote the $\HH$-scalar product.
Recall that the Riesz mapping $I_\HH:\HH\to\HH^*$, $I_\HH w:=\product{\cdot}{w}_\HH$ is (up to complex conjugation) an isometric isomorphism \cite[Chapter~III.6]{yoshi}.
Define
\begin{align}\label{eq:picard}
 \Phi v:=v-(\alpha/L^2)\,I_\HH^{-1}(Av-F)
 \quad\text{and}\quad
 q:=(1-\alpha^2/L^2)^{1/2} < 1,
\end{align}
where we note that~\eqref{axiom:monotone}--\eqref{axiom:lipschitz} imply, in particular, $\alpha\le L$.
Then,
\begin{align*}
 \norm{\Phi v -\Phi w}\HH^2
 = \norm{v-w}\HH^2 - 2\,\frac{\alpha}{L^2}\,\Re\,\product{v-w}{I_\HH^{-1}(Av-Aw)}_\HH + \frac{\alpha^2}{L^4}\,\norm{I_\HH^{-1}(Av-Aw)}\HH^2.
\end{align*}
First, note that
\begin{align*}
 \product{v-w}{I_\HH^{-1}(Av-Aw)}_\HH = \dual{Av-Aw}{v-w}.
\end{align*}
Second, note that 
\begin{align*}
 \norm{I_\HH^{-1}(Av-Aw)}\HH^2 = \norm{Av-Aw}{\HH^*}^2 \stackrel{\eqref{axiom:lipschitz}}{\le} L^2\, \norm{v-w}\HH^2.
\end{align*}
Combining these observations, we see that
\begin{align}\label{eq:picard_contraction}
 \begin{split}
 \norm{\Phi v -\Phi w}\HH^2 
 &\le \Big(1+\frac{\alpha^2}{L^2}\Big)\,\norm{v-w}\HH^2 - 2\,\frac{\alpha}{L^2}\Re\,\dual{Av-Aw}{v-w}
 \\ 
 &\reff{axiom:monotone}\le \Big(1-\frac{\alpha^2}{L^2}\Big)\,\norm{v-w}\HH^2 
 \stackrel{\eqref{eq:picard}}{=} q^2 \, \norm{v - w}\HH^2.
 \end{split}
\end{align}
Hence, $\Phi:\HH\to\HH$ is a contraction with Lipschitz constant $0<q<1$. According to the Banach fixpoint theorem, $\Phi$ has a unique fixpoint $u^\exact\in\HH$, i.e., $u^\exact = \Phi u^\exact$. By definition of $\Phi$, the strong form~\eqref{eq:strongform} is equivalent to $u^\exact = \Phi u^\exact$. 
Overall,~\eqref{eq:strongform} admits a unique solution.

Moreover, the Banach fixpoint theorem guarantees that for each initial guess $u^0\in\HH$, the Picard iteration $u^{n}:=\Phi u^{n-1}$ converges to $u^\exact$ as $n\to\infty$.  Note that, for all $n\in\N$, 
\begin{align*}
 \norm{u^\exact-u^n}\HH = \norm{\Phi u^\exact - \Phi u^{n-1}}\HH
 \stackrel{\eqref{eq:picard_contraction}}{\le} q\,\norm{u^\exact-u^{n-1}}\HH
 \le q\,\norm{u^\exact-u^n}\HH + q\,\norm{u^n-u^{n-1}}\HH.
\end{align*}
Rearranging this estimate and induction on $n$ with~\eqref{eq:picard_contraction}, we derive the following well-known {\sl a~posteriori} and {\sl a~priori} estimate for the Picard iterates,
\begin{subequations}\label{eq:estimate}
\begin{align}
 \norm{u^\exact-u^n}\HH \le \frac{q}{1-q}\,\norm{u^n-u^{n-1}}\HH
 \stackrel{\eqref{eq:picard_contraction}}{\le} \frac{q^n}{1-q}\,\norm{u^1-u^0}\HH.
\end{align}
Moreover, it holds that
\begin{align}
\norm{u^n-u^{n-1}}\HH \leq \norm{u^\exact-u^n}\HH  + \norm{u^\exact-u^{n-1}}\HH  
 	\stackrel{\eqref{eq:picard_contraction}}{\leq} (1+q) \, \norm{u^\exact-u^{n-1}}\HH.
\end{align}
Thus, the {\sl a~posteriori} computable term $\norm{u^n-u^{n-1}}\HH$ provides an upper bound
for $\norm{u^\exact-u^n}\HH$ as well as a lower bound for $\norm{u^\exact-u^{n-1}}\HH$.
\end{subequations}

\section{Discretization and {a~priori} error estimation}
\label{section:discretization}

\subsection{Nonlinear discrete problem}
Suppose that $\XX_\plus \subset \HH$ is a conforming discrete subspace of $\HH$. 
If~\eqref{axiom:monotone}--\eqref{axiom:lipschitz} are satisfied, the restriction 
$A_\plus:\XX_\plus\to\XX_\plus^*$ of $A$ is
strongly monotone and Lipschitz continuous,
even with the same constants $\alpha,L>0$ for~\eqref{axiom:monotone}--\eqref{axiom:lipschitz} as in the continuous case. In particular, there exists a unique  solution $u_\plus^\exact \in\XX_\plus$ to
\begin{align}\label{eq:discrete}
 \dual{Au_\plus^\exact}{v_\plus} = \dual{F}{v_\plus} 
 \quad\text{for all }v_\plus\in\XX_\plus.
\end{align}
First, recall the following well-known C\'ea-type estimate for strongly monotone operators. 
We include its proof for the sake of completeness.

\begin{lemma}\label{lemma:cea}
Suppose that the 
operator $A$ satisfies~\eqref{axiom:monotone}--\eqref{axiom:lipschitz}. Then, it holds that

\begin{align}
\norm{u^\exact-u_\plus^\exact}\HH \le \frac{L}{\alpha}\,\min_{w_\plus\in\XX_\plus}\norm{u^\exact-w_\plus}\HH.
\end{align}	
\end{lemma}

\begin{proof}
Note the Galerkin orthogonality $\dual{Au^\exact-Au_\plus^\exact}{v_\plus} = 0$ for all $v_\plus\in\XX_\plus$.
For $w_\plus\in\XX_\plus$ and $u^\star\neq u^\star_\bullet$, this results in
\begin{align*}
 \alpha\,\norm{u^\exact-u_\plus^\exact}\HH
 \reff{axiom:monotone}\le \frac{\Re \, \dual{Au^\exact-Au_\plus^\exact}{u^\exact-u_\plus^\exact}}{\norm{u^\exact-u_\plus^\exact}\HH}
 = \frac{\Re \, \dual{Au^\exact-Au_\plus^\exact}{u^\exact-w_\plus}}{\norm{u^\exact-u_\plus^\exact}\HH}
 &\reff{axiom:lipschitz}\le L\,\norm{u^\exact-w_\plus}\HH.
\end{align*}%
Finite dimension concludes that the infimum over all $w_\plus \in \XX_\plus$ is, in fact, attained.
\end{proof}

\subsection{Linearized discrete problem}
\label{section:discreteproblem}
Note that the nonlinear system~\eqref{eq:discrete} can hardly be solved exactly. 
With the discrete Riesz mapping $I_\plus:\XX_\plus\to\XX_\plus^*$ and the restriction $F_\plus\in\XX_\plus^*$ of $F$ to $\XX_\plus$, define $\Phi_\plus:\XX_\plus\to\XX_\plus$ by $\Phi_\plus v_\plus:=v_\plus-(\alpha/L^2)I_\plus^{-1}(A_\plus v_\plus-F_\plus)$.
Given $u_\plus^{n} \in \XX_\plus$, we thus compute the discrete Picard iterate $u_\plus^{n} := \Phi_\plus u^{n-1}_\plus$ as follows:
\begin{itemize}
	\item[(i)] Solve the linear system $(v_\plus, w_\plus)_\HH = \dual{A u_\plus^{n-1} - F}{v_\plus}$ for all $v_\plus \in \XX_\plus$.
	\item[(ii)] Define $u_\plus^{n} := u_\plus^{n-1} - \frac{\alpha}{L^2} w_\plus$.
\end{itemize}
Then, $\Phi_\plus$ is a contraction (cf.\ Section~\ref{section:banach}), and for each initial guess $u^0_\plus\in\XX_\plus$, the Picard iteration $u^{n+1}_\plus = \Phi_\plus u^n_\plus$ converges to $u_\plus^\exact$ as $n\to\infty$. 
Moreover, the error estimates~\eqref{eq:estimate} also hold for the discrete Picard iteration, i.e., for all $n\in\N$, it holds that
\begin{align}\label{eq:estimate:discrete}
\begin{split}
 \norm{u_\plus^\exact-u_\plus^n}\HH &\le \frac{q}{1\!-\!q}\,\norm{u_\plus^n-u_\plus^{n-1}}\HH 
 \le \min\Big\{\frac{q^n}{1\!-\!q}\norm{u_\plus^1-u_\plus^0}\HH,\,\frac{q(1\!+\!q)}{1\!-\!q}\norm{u_\plus^\exact-u_\plus^{n-1}}\HH\Big\}.\hspace*{-2mm}
\end{split}
\end{align}
Finally, we recall the following {\sl a~priori} estimate for the discrete Picard iteration from \cite[Proposition~2.1]{cw2015} and also include its simple proof for the sake of completeness:

\begin{lemma}\label{lemma:apriori}
Suppose that the 
operator $A$ satisfies~\eqref{axiom:monotone}--\eqref{axiom:lipschitz}. Then, it holds that
\begin{align}
\norm{u^\exact-u_\plus^n}\HH \le \frac{L}{\alpha}\,\min_{w_\plus\in\XX_\plus}\norm{u^\exact-w_\plus}\HH 
+ \frac{q^n}{1-q}\,\norm{u_\plus^1-u_\plus^0}\HH
\quad\text{for all }n\in\N.
\end{align}
\end{lemma}
\begin{proof}
With
$\displaystyle \norm{u^\exact-u_\plus^n}\HH
 \le \norm{u^\exact-u_\plus^\exact}\HH + \norm{u_\plus^\exact-u_\plus^n}\HH
 \reff{eq:estimate:discrete}\le \norm{u^\exact-u_\plus^\exact}\HH + \frac{q^n}{1-q}\,\norm{u_\plus^1-u_\plus^0}\HH,$
the proof follows from the C\'ea-type estimate of Lemma~\ref{lemma:cea}.
\end{proof}

\begin{remark}\label{remark1}
Note that $u_\plus^1 = \Phi_\plus u_\plus^0$ implies that
\begin{align*}
 \product{u_\plus^1}{v_\plus}_\HH = \product{u_\plus^0}{v_\plus}_\HH
 - \frac{\alpha}{L^2}\,\dual{A u_\plus^0-F}{v_\plus}
 \quad\text{for all }v_\plus\in\XX_\plus.
\end{align*}
For $v_\plus = u_\plus^1-u_\plus^0$, this reveals that
\begin{align*}
 \norm{u_\plus^1-u_\plus^0}{\HH}^2 
 = - \frac{\alpha}{L^2}\,\dual{A u_\plus^0-F}{u_\plus^1-u_\plus^0} 
 &\le \frac{\alpha}{L^2}\,\norm{Au_\plus^0-F}{\HH^*}\,\norm{u_\plus^1-u_\plus^0}{\HH}.
\end{align*}
Consequently, we get
\begin{align}\label{eq:remark_normbounded}
 \norm{u_\plus^1-u_\plus^0}\HH 
 \le \frac{\alpha}{L^2}\,\norm{Au_\plus^0-F}{\HH^*}
  \stackrel{\eqref{eq:strongform}}{=} \frac{\alpha}{L^2} \, \norm{A u_\plus^0 - Au^\exact}{\HH^*} 
 \stackrel{\eqref{axiom:lipschitz}}{\leq} \frac{\alpha}{L}\,\norm{u_\plus^0 - u^\exact}{\HH}.
\end{align}
\noindent
Therefore, boundedness of $\norm{u_\plus^1-u_\plus^0}\HH$ in the {\sl a~priori} estimate of Lemma~\ref{lemma:apriori} can be guaranteed independently of the space $\XX_\plus \subset \HH$ by choosing, e.g., $u_\plus^0:=0$.
If $\min_{w_\plus\in\XX_\plus}\norm{u^\exact-w_\plus}\HH = \OO(N^{-s})$ for some $s>0$ and with $N>0$ being the degrees of freedom associated with $\XX_\plus$, this suggests the choice $n = \OO(\log N)$ in Lemma~\ref{lemma:apriori}; see the discussion in~\cite[Remark~3.7]{cw2015}. Moreover, we shall see below that the choice of $u_\plus^0$ by nested iteration generically leads to $n = \OO(1)$; see Proposition~\ref{proposition:nested_iteration} below.
\qed
\end{remark}

\section{Adaptive algorithm}
\label{section:adaptivealgorithm}

\subsection{Basic properties of mesh-refinement}
Suppose that all considered discrete spaces $\XX_\plus\subset\HH$ are associated with a triangulation $\TT_\plus$ of a fixed bounded Lipschitz domain $\Omega\subset\R^d$, $d\ge2$. 
Suppose that $\refine(\cdot)$ is a fixed mesh-refinement strategy. 

Given a triangulation $\TT_\bullet$ and $\MM_\bullet \subseteq \TT_\bullet$, let $\TT_\circ := \refine(\TT_\bullet,\MM_\bullet)$ be the coarsest triangulation such that all 
marked elements $T \in \TT_\bullet$ have been refined, i.e., $\MM_\bullet \subseteq \TT_\bullet \setminus \TT_\circ$.

We write $\TT_\circ\in\refine(\TT_\plus)$ if $\TT_\circ$ is obtained by a finite number of refinement steps, i.e., there exists $n\in\N_0$ as well as a finite sequence $\TT_{(0)},\dots,\TT_{(n)}$ of triangulations and corresponding sets $\MM_{(j)}\subseteq\TT_{(j)}$ such that
\begin{itemize}
	\item $\TT_\plus = \TT_{(0)}$,
	\item $\TT_{(j+1)} = \refine(\TT_{(j)},\MM_{(j)})$ for all $j=0,\dots,n-1$,
	\item $\TT_\circ = \TT_{(n)}$.
\end{itemize}
In particular, $\TT_\plus\in\refine(\TT_\plus)$. Further, we suppose that refinement $\TT_\circ\in \refine(\TT_\plus)$ yields nestedness $\XX_\plus\subseteq\XX_\circ$ of the corresponding discrete spaces.

In view of the adaptive algorithm (Algorithm \ref{algorithm} below), let $\TT_0$ be a fixed initial triangulation. 
 To ease notation, let $\T:=\refine(\TT_0)$ be the set of all possible triangulations which can be obtained by successively refining $\TT_0$.
 
\subsection{\textsl{A posteriori} error estimator}
\label{subsection:errorestimator}
Suppose that for each $T\in\TT_\plus\in\T$ and each discrete function $v_\plus\in\XX_\plus$, one can compute an associated \emph{refinement indicator} $\eta_\plus(T,v_\plus)\ge0$. To abbreviate notation, let
\begin{align}\label{eq:formalestiamtor}
 \eta_\plus(v_\plus) := \eta_\plus(\TT_\plus,v_\plus),
 \quad\text{where}\quad
 \eta_\plus(\UU_\plus,v_\plus) := \Big(\sum_{T\in\UU_\plus}\eta_\plus(T,v_\plus)^2\Big)^{1/2}
 \quad\text{for all }\UU_\plus\subseteq\TT_\plus.
\end{align}
We suppose the following properties with fixed constants $\Cstab,\Cred,\Crelexact,\Cdrel \ge 1$ and $0<\qred<1$ which slightly generalize those {\em axioms of adaptivity} of~\cite{axioms}:
\begin{enumerate}
\renewcommand{\theenumi}{A\arabic{enumi}}
\bf
\item\label{axiom:stability}
\rm
\textbf{stability on non-refined element domains:} For all triangulations $\TT_\plus\in\T$ and $\TT_\circ\in\refine(\TT_\plus)$, arbitrary discrete functions $v_\plus\in\XX_\plus$ and $v_\circ\in\XX_\circ$, and an arbitrary set $\UU_\plus\subseteq\TT_\plus\cap\TT_\circ$ of non-refined elements, it holds that
$$|\eta_\circ(\UU_\plus,v_\circ)-\eta_\plus(\UU_\plus,v_\plus)| \le \Cstab\,\norm{v_\plus-v_\circ}\HH.$$
\bf
\item\label{axiom:reduction}
\rm
\textbf{reduction on refined element domains:} For all triangulations $\TT_\plus\in\T$ and $\TT_\circ\in\refine(\TT_\plus)$, and arbitrary 
$v_\plus\in\XX_\plus$ and $v_\circ\in\XX_\circ$, it holds that
$$\eta_\circ(\TT_\circ\backslash\TT_\plus,v_\circ)^2 \le\qred\,\eta_\plus(\TT_\plus\backslash\TT_\circ,v_\plus)^2 + \Cred\,\norm{v_\circ-v_\plus}\HH^2.$$
\bf
\item\label{axiom:reliability}
\rm
\textbf{reliability:} For all triangulations $\TT_\plus\in\T$, the error of the discrete solution $u_\plus^\exact\in\XX_\plus$ to~\eqref{eq:discrete} is controlled by $$\norm{u^\exact-u_\plus^\exact}\HH\le\Crelexact\,\eta_\plus(u_\plus^\exact).$$
\bf
\item\label{axiom:discrete_reliability}
\rm
\textbf{discrete reliability:}
For all $\TT_\plus\in\T$ and all $\TT_\circ\in\refine(\TT_\plus)$, there exists a set $\RR_{\plus,\circ}\subseteq\TT_\plus$ 
with $\TT_\plus\backslash\TT_\circ \subseteq \RR_{\plus,\circ}$ as well as $\#\RR_{\plus,\circ} \le \Cdrel\,\#(\TT_\plus\backslash\TT_\circ)$ such that the difference of the discrete solutions $u_\plus^\exact \in \XX_\plus$ and $u_\circ^\exact \in \XX_\circ$ is controlled by
\begin{align*}
	\norm{u_\circ^\exact-u_\plus^\exact}{\HH} \le \Cdrel\,\eta_\plus(\RR_{\plus,\circ},u_\plus^\exact).
\end{align*}
\end{enumerate}

\begin{remark}
	Suppose the following approximation property of $u^\exact \in \HH$: For all $\TT_\bullet \in \T$ and all $\varepsilon >0$,
	there exists a refinement $\TT_\circ \in \refine(\TT_\bullet)$  such that $\norm{u^\exact - u_\circ^\exact}{\HH} \leq \varepsilon$.
	Then, discrete reliability~\eqref{axiom:discrete_reliability} already implies reliability~\eqref{axiom:reliability}; 
	see \cite[Lemma 3.4]{axioms}.
\end{remark}

We note that~\eqref{axiom:reliability}--\eqref{axiom:discrete_reliability} are formulated for the non-computable Galerkin solutions $u_\plus^\exact\in\XX_\plus$ to~\eqref{eq:discrete}, while Algorithm~\ref{algorithm} below generates approximations $u_\plus^{n}\approx u_\plus^\exact\in\XX_\plus$. The following lemma proves that reliability~\eqref{axiom:reliability} transfers to certain Picard iterates.

\begin{lemma}\label{lemma:reliability}
Suppose~\eqref{axiom:monotone}--\eqref{axiom:lipschitz} for the 
operator $A$ as well as stability~\eqref{axiom:stability} and reliability~\eqref{axiom:reliability} for the {\sl a~posteriori} error estimator. Let $\lambda >0$, $n \in \N$, and $u_\bullet^0\in\XX_\bullet$ and suppose that the Picard iterate $u_\plus^n\in\XX_\plus$ satisfies $\norm{u_\plus^{n} - u_\plus^{n-1}}{\HH} \le \lambda \eta_\plus(u_\plus^n)$. Then, it holds
that
\begin{align}\label{eq:reliability}
 \norm{u^\exact-u_\plus^n}\HH 
 \le \Crel^\inexact\,\eta_\plus(u_\plus^n)
 \quad\text{with}\quad
 \Crel^\inexact := \Crelexact + \lambda\,(1+\Crelexact\Cstab)\,\frac{q}{1-q}.
\end{align}
\end{lemma}

\begin{proof}
Reliability~\eqref{axiom:reliability} and stability~\eqref{axiom:stability} prove that
\begin{align*}
 \norm{u^\exact-u_\plus^n}\HH 
 \le \norm{u^\exact-u_\plus^\exact}\HH + \norm{u_\plus^\exact-u_\plus^n}\HH
 &\reff{axiom:reliability}\le \Crelexact\,\eta_\plus(u_\plus^\exact) + \norm{u_\plus^\exact-u_\plus^n}\HH
 \\&
 \reff{axiom:stability}\le \Crelexact\,\eta_\plus(u_\plus^n) + (1+\Crelexact\Cstab)\,\norm{u_\plus^\exact-u_\plus^n}\HH.
\end{align*}
The {\sl a~posteriori} estimate~\eqref{eq:estimate:discrete} together with the assumption on $u_\plus^n$ yields that
\begin{align*}
 \norm{u_\plus^\exact-u_\plus^n}\HH 
 \reff{eq:estimate:discrete}\le \frac{q}{1-q}\,\norm{u_\plus^n-u_\plus^{n-1}}\HH \le \lambda\,\frac{q}{1-q}\,\eta_\plus(u_\plus^n).
\end{align*}
Combining these estimates, we conclude the proof.
\end{proof}

\subsection{Adaptive algorithm}
In the sequel, we analyze the following adaptive algorithm which ---up to a different {\sl a~posteriori} error estimation based on \emph{elliptic reconstruction}--- is also considered in~\cite{cw2015}.

\begin{algorithm}\label{algorithm}
\textsc{Input:} Initial triangulation $\TT_0$, adaptivity parameters $0<\theta\le 1$, $\lambda>0$ and $\Cmark\ge1$, arbitrary initial guess $u_0^0\in\XX_0$, e.g., $u_0^0:=0$.\\
\textsc{Adaptive loop:} For all $\ell=0,1,2,\dots$, iterate the following steps~{\rm(i)--(iii)}.
\begin{enumerate}
\item[\rm(i)]Repeat the following steps~{\rm(a)--(b)} for all $n=1,2,\dots$, until 
$\norm{u_\ell^n-u_\ell^{n-1}}\HH\le \lambda\,\eta_\ell(u_\ell^n)$.
\begin{enumerate}
\item[\rm (a)] Compute discrete Picard iterate $u_\ell^n = \Phi_\ell u_\ell^{n-1}\in\XX_\ell$.
\item[\rm (b)] Compute refinement indicators $\eta_\ell(T,u_\ell^n)$ for all $T\in\TT_\ell$.
\end{enumerate}
\item[\rm(ii)] Define $u_\ell^\inexact := u_\ell^n \in \XX_\ell$ and determine a set $\MM_\ell\subseteq\TT_\ell$ of marked elements which has minimal cardinality up to the multiplicative constant $\Cmark$ and which satisfies the D\"orfler marking criterion $\theta\,\eta_\ell(u_\ell^\inexact) \le \eta_\ell(\MM_\ell,u_\ell^\inexact)$.
\item[\rm(iii)] Generate the new triangulation $\TT_{\ell+1} := \refine(\TT_\ell,\MM_\ell)$ by refinement of (at least) all marked elements $T\in\MM_\ell$ and define $u_{\ell+1}^0 := u_\ell^\inexact \in\XX_\ell\subseteq\XX_{\ell+1}$.
\end{enumerate}
\textsc{Output:} Sequence of discrete solutions $u_\ell^\inexact\in\XX_\ell$ and corresponding estimators $\eta_\ell(u_\ell^\inexact)$.\qed
\end{algorithm}

The following two results analyze the (lucky) breakdown of Algorithm~\ref{algorithm}. 
The first proposition shows that, if the repeat loop of step (i) does not terminate after finitely many steps, 
then the exact solution $u^\exact=u_\ell^\exact$ belongs to the discrete space $\XX_\ell$.

\begin{proposition}\label{proposition:repeat}
	Suppose~\eqref{axiom:monotone}--\eqref{axiom:lipschitz} for the nonlinear operator $A$ 
	as well as stability~\eqref{axiom:stability} and reliability~\eqref{axiom:reliability} for the {\sl a~posteriori} error estimator. 
	Let $\lambda >0$ and assume that step~{\rm(i)} in Algorithm~\ref{algorithm} does not terminate for some fixed $\ell \in \N_0$, i.e., 
	$\norm{u_\ell^n - u_\ell^{n-1}}{\HH} > \lambda \eta_\ell(u_\ell^n)$ for all $n \in \N$. Define $u_{-1}^\inexact:=0$ if $\ell=0$. Then, there holds
	$u^\exact = u_\ell^\exact \in \XX_\ell$ and
	\begin{align}\label{eq:geometric}
		\eta_\ell(u_\ell^n) < \lambda^{-1} q^{n-1} \frac{\alpha}{L} \,  \norm{u^\exact - u_{\ell-1}^\inexact}{\HH} \xrightarrow{n\to\infty} 0 = \eta_\ell(u^\exact).
	\end{align}  
\end{proposition}
\begin{proof}
	According to \eqref{axiom:stability}, $\eta_\ell(v_\ell)$ depends Lipschitz continuously on $v_\ell \in \XX_\ell$. 
	Moreover, it holds that $\norm{u_\ell^\exact - u_\ell^n}{\HH} \rightarrow 0$ and hence  $\norm{u_\ell^n - u_\ell^{n-1}}{\HH} \rightarrow 0$
	as $n \to \infty$. This proves that
	\begin{align*}
		\norm{u^\exact - u_\ell^\exact}{\HH} \reff{axiom:reliability}\lesssim \eta_\ell(u_\ell^\exact)
		= \lim_{n \to \infty} \eta_\ell(u_\ell^{n}) \leq \lambda^{-1} \lim_{n \to \infty} \norm{u_\ell^n - u_\ell^{n-1}}{\HH} = 0.
	\end{align*}
	Hence, $u^\exact= u_\ell^\exact \in \XX_\ell$ and $\eta_\ell(u^\exact) = \eta_\ell(u_\ell^\exact) =0$.
	With $\lambda \, \eta_\ell(u_\ell^{n}) < \norm{u_\ell^n - u_\ell^{n-1}}{\HH}$, we see that
	\begin{align*}
	 \eta_\ell(u_\ell^{n}) < \lambda^{-1} \norm{u_\ell^n - u_\ell^{n-1}}{\HH} \stackrel{\eqref{eq:picard_contraction}}{\leq} 
	\lambda^{-1} q^{n-1} \, \norm{u_\ell^1 - u_\ell^0}\HH
	\reff{eq:remark_normbounded}\le \lambda^{-1}\,q^{n-1} \frac{\alpha}{L} \, \norm{u^\exact - u_{\ell-1}^\inexact}\HH.
	\end{align*}
	This concludes \eqref{eq:geometric}.
\end{proof}

The second proposition shows that, if the repeat loop of step~(i) does terminate with $\eta_\ell(u_\ell)=0$, then $u_\ell = u^\exact$ as well as $\eta_k(u_k)=0$ and $u_k=u_k^1=u^\star$ for all $k>\ell$.

\begin{proposition}\label{prop:neww}
Suppose~\eqref{axiom:monotone}--\eqref{axiom:lipschitz} for the 
operator $A$ as well as stability~\eqref{axiom:stability} and reliability~\eqref{axiom:reliability} for the 
error estimator.
If step~{\rm(i)} in Algorithm~\ref{algorithm} terminates with $\eta_\ell(u_\ell^\inexact)=0$ 
(or equivalently with $\MM_\ell = \emptyset$ in step~{\rm(ii)}) for some $\ell\in\N_0$, then 
$u_k^\inexact=u^\exact$ as well as $\MM_k = \emptyset$ for all $k \geq \ell$. Moreover, for all $k \geq \ell+1$, step~{\rm(i)} terminates after one iteration. 
\end{proposition}

\begin{proof}
Clearly, $\MM_\ell = \emptyset$ implies $\eta_\ell(u_\ell^\inexact)=0$ . Conversely, $\eta_\ell(u_\ell^\inexact)=0$ 
also implies $\MM_\ell = \emptyset$.
In this case, Lemma~\ref{lemma:reliability} yields $\norm{u^\exact-u_\ell^\inexact}\HH=0$ and hence $u_\ell^\inexact=u^\exact$.
Moreover, $\MM_\ell = \emptyset$ implies $\TT_{\ell+1} = \TT_\ell$. 
Nested iteration guarantees $u_{\ell+1}^1 = \Phi_{\ell+1} u_{\ell+1}^0 =\Phi_{\ell+1} u_\ell^\inexact =\Phi_{\ell+1} u^\exact = u^\exact = u_\ell^\inexact = u_{\ell+1}^0$ and therefore $\norm{u_{\ell+1}^1 -  u_{\ell+1}^0}{\HH} =0$.
Together with $\TT_{\ell+1}=\TT_\ell$, this implies $u_{\ell+1} = u_\ell$ and $\eta_{\ell+1}(u_{\ell+1})=\eta_\ell(u_\ell)=0$.
Hence, for all $k \geq {\ell+1}$, step~{\rm(i)} terminates after one iteration with $u_k^\inexact = u^\exact$ and $\MM_k =\emptyset$. 
\end{proof}

For the rest of this section, we suppose that step~(i) of Algorithm~\ref{algorithm} terminates with $\eta_\ell(u_\ell)>0$ for all $\ell \in \N_0$.
In this case, we control the number of Picard iterates $\#\pic(\ell)$.

\begin{proposition}\label{proposition:nested_iteration}
Suppose~\eqref{axiom:monotone}--\eqref{axiom:lipschitz} for the nonlinear 
operator $A$ as well as stability~\eqref{axiom:stability} and reliability~\eqref{axiom:reliability} for the {\sl a~posteriori} 
error estimator.
Let $0<\theta\le1$ and $\lambda>0$ be the adaptivity parameters of Algorithm~\ref{algorithm}. 
Suppose that, for all $\ell\in\N_0$, step~{\rm(i)} of Algorithm~\ref{algorithm} terminates after finitely many steps and that $\eta_\ell(u_\ell^\inexact)>0$.
Then, there exists a constant $\Cpic\ge1$ such that for all $\ell\in\N$, 
the number of Picard iterates $\#\pic(\ell)$ satisfies
\begin{align}\label{eq:nested_iteration}
 \#\pic(\ell)\le \Cpic + \frac{1}{|\log q|} \, \log \big(  \max \big\{ 1 \,,\, \eta_{\ell-1}(u_{\ell-1}^\inexact) / \eta_\ell(u_\ell^\inexact) \big\}\big).
\end{align}
The constant $\Cpic$ depends only on~\eqref{axiom:monotone}--\eqref{axiom:lipschitz}, \eqref{axiom:stability}, and \eqref{axiom:reliability}.
\end{proposition}

The proof of Proposition~\ref{proposition:nested_iteration} employs the following auxiliary result.

\begin{lemma}\label{lemma:eta:picard}
Suppose~\eqref{axiom:monotone}--\eqref{axiom:lipschitz} for the 
operator $A$ as well as stability~\eqref{axiom:stability} and reliability~\eqref{axiom:reliability} for the 
error estimator.
Suppose that, for all $\ell\in\N_0$, step~{\rm(i)} of Algorithm~\ref{algorithm} terminates after finitely many steps.
Then, for $\ell\in\N$, the discrete Picard iterates satisfy
\begin{align}\label{eq:important}
 \norm{u_\ell^n-u_\ell^{n-1}}\HH
 \le q^{n-1}\,\frac{\alpha}{L}\,\Crel^\inexact\,\eta_{\ell-1}(u_{\ell-1}^\inexact)
 \quad\text{for all }n\in\N.
\end{align}
Moreover, 
it holds that
\begin{align}\label{eq:eta:picard}
 \eta_\ell(u_{\ell}^n) \le \frac{\alpha}{L}\,\frac{\Crel^\inexact}{\lambda} \,q^{n-1}\,\eta_{\ell-1}(u_{\ell-1}^\inexact)
 \quad\text{for all }n=1,\dots,\#\pic(\ell)-1.
\end{align}
\end{lemma}

\begin{proof}
Recall that $u_\ell^0 = u_{\ell-1}^\inexact$ and
\begin{align*}
\norm{u_{\ell}^1 - u_{\ell}^0}{\HH} \stackrel{\eqref{eq:remark_normbounded}}{\le}  
\frac{\alpha}{L} \,\norm{u^\exact-u_\ell^0}{\HH}
= \frac{\alpha}{L} \,\norm{u^\exact-u_{\ell-1}^\inexact}{\HH}.
\end{align*}
With the contraction~\eqref{eq:picard_contraction} for the Picard iterates and reliability~\eqref{eq:reliability}, this yields that
\begin{align*}
 \norm{u_\ell^n-u_\ell^{n-1}}\HH
 \reff{eq:picard_contraction}\le q^{n-1}\,\norm{u_\ell^1-u_\ell^0}\HH
 \le q^{n-1}\,\frac{\alpha}{L}\,\norm{u^\exact-u_{\ell-1}^\inexact}\HH
 \reff{eq:reliability}\le q^{n-1}\,\frac{\alpha}{L}\,\Crel^\inexact\,\eta_{\ell-1}(u_{\ell-1}^\inexact).
\end{align*}
This proves~\eqref{eq:important}. To see~\eqref{eq:eta:picard}, note that Algorithm~\ref{algorithm} ensures that $k:=\#\pic(\ell)$ is the minimal integer with $\norm{u_\ell^k-u_\ell^{k-1}}\HH \le \lambda\,\eta_\ell(u_\ell^k)$. For $n=1,\dots,k-1$, it thus holds that
\begin{align*}
 \lambda\,\eta_\ell(u_\ell^n) < \norm{u_\ell^n-u_\ell^{n-1}}\HH
 \reff{eq:important}\le \frac{\alpha}{L}\,\Crel^\inexact\,q^{n-1}\,\eta_{\ell-1}(u_{\ell-1}^\inexact).
\end{align*}
This concludes the proof.
\end{proof}

\begin{proof}[Proof of Proposition~\ref{proposition:nested_iteration}]
We prove the assertion in two steps.

{\bf Step~1.}\quad	Let $d_\ell:= \eta_{\ell-1}(u_{\ell-1}^\inexact) / \eta_{\ell}(u_{\ell}^\inexact)$. Choose $n \in \N$ minimal such that 
\begin{align}\label{eq:proposition_nbound}
0 < \frac{C d_\ell}{1-C' d_\ell\,q^n}\,q^{n-1} \le \lambda,
\quad\text{where}\quad
C:= \frac{\alpha}{L}\,\Crel^\inexact
\quad\text{and}\quad
C':= C\Cstab\,\frac{1}{1-q}.
\end{align}
In this step, we aim to show that $k:=\#\pic(\ell)\le n$. 
To this end, we argue by contradiction and assume $k>n$. First, recall that step~(i) of Algorithm~\ref{algorithm} guarantees that $k\in\N$ is the minimal index such that $\norm{u_\ell^k-u_\ell^{k-1}}\HH \le \lambda \eta_\ell(u_\ell^k)$. Second, note that~\eqref{eq:important} yields that
\begin{align*}
\norm{u_\ell^n-u_\ell^{n-1}}\HH
 \reff{eq:important}\le q^{n-1}\,\frac{\alpha}{L}\,\Crel^\inexact\,\eta_{\ell-1}(u_{\ell-1}^\inexact)
 = q^{n-1}\,\frac{\alpha}{L}\,\Crel^\inexact\, d_\ell \,\eta_{\ell}(u_\ell^\inexact)
 = C d_\ell q^{n-1}\,\eta_{\ell}(u_\ell^\inexact). 
\end{align*}
Since $k>n$, stability~\eqref{axiom:stability} and the geometric series prove that
\begin{align*}
 &\eta_{\ell}(u_\ell^\inexact) 
 = \eta_\ell(u_\ell^k)
 \reff{axiom:stability}\le \eta_\ell(u_\ell^n) + \Cstab\,\norm{u_\ell^k-u_\ell^n}\HH
 \le \eta_\ell(u_\ell^n) + \Cstab\,\sum_{j=n}^{k-1}\norm{u_\ell^{j+1}-u_\ell^{j}}\HH
 \\&\qquad
 \reff{eq:picard_contraction}\le \eta_\ell(u_\ell^n) + \Cstab\,\norm{u_\ell^{n}-u_\ell^{n-1}}\HH\,\sum_{j=n}^{k-1}q^{(j+1)-n}
 \le \eta_\ell(u_\ell^n) + \Cstab\,\frac{q}{1-q}\,\norm{u_\ell^n-u_\ell^{n-1}}\HH.
\end{align*}
Combining the last two estimates, we derive that
\begin{align*}
 \norm{u_\ell^n-u_\ell^{n-1}}\HH
 \le C d_\ell q^{n-1} \,\eta_\ell(u_\ell^n) +  
 C' d_\ell \,q^n\,\norm{u_\ell^n-u_\ell^{n-1}}\HH.
\end{align*}
Rearranging the terms in combination with~\eqref{eq:proposition_nbound}, we obtain that
\begin{align*}
\norm{u_\ell^n-u_\ell^{n-1}}\HH
 \le \frac{C d_\ell}{1-C' d_\ell \,q^n}\,q^{n-1} \,\eta_\ell(u_\ell^n)
 \le \lambda \eta_\ell(u_\ell^n).
\end{align*}
This contradicts the minimality of $k$ and concludes $k \leq n$.

{\bf Step~2.}\quad With $C,C'>0$ from~\eqref{eq:proposition_nbound}, define
\begin{align}
\begin{split}\label{eq:proposition_Ndefine}
 \Cpic &:= \frac{1}{|\log q|}\,\max\Big\{\Big|\log\frac{1}{2C'}\Big|\,,\,\Big|\log\frac{\lambda q}{2C}\Big|\Big\}+1,\\
 N&:= \Big\lceil\Cpic-1+\frac{1}{|\log q|}\,\log\big(\max\big\{1,d_\ell\big\}\big)\Big\rceil.
\end{split}
\end{align}
In this step, we will prove that equation~\eqref{eq:proposition_nbound} is satisfied with $N$ instead of $n$. According to step~1, this will lead to
\begin{align*}
 k = \#\pic(\ell)\le n\le N \le \Cpic + \frac{1}{|\log q|}\,\log\big(\max\big\{1,d_\ell\}\big)
\end{align*}
and hence conclude~\eqref{eq:nested_iteration}. To this end, first note that $$\log\big(\max\big\{1,d_\ell\}\big) = \big|\log\big(1 / \max\{1,d_\ell\}\big)\big|.$$ Therefore,
basic calculus in combination with $N\ge\frac{1}{|\log q|}\,\Big|\log\frac{1}{2C'} + \log\frac{1}{\max\{1,d_\ell\}}\Big|$ reveals that
\begin{align*}
	C'\max\{1,d_\ell\}  \,q^N  \leq C'\max\{1,d_\ell\} \, \frac{1}{2 C' \max \{1,d_\ell \} }  = \frac{1}{2}
\end{align*}
and proves, in particular, that
\begin{align*}
0 < \frac{C d_\ell}{1- C'd_\ell  \,q^N}\,q^{N-1}.
\end{align*}%
Together with $N \ge \frac{1}{|\log q|}\,\Big|\log\frac{\lambda q}{2C} + \log\big(1/\max\{1,d_\ell\}\big)\Big|$, we finally get
\begin{align*}
 0 < \frac{C d_\ell}{1- C'd_\ell  \,q^N}\,q^{N-1} \leq 
	\frac{C \max\{1,d_\ell\} }{1- C'\max\{1,d_\ell\}  \,q^N}\,q^{N-1}
   &\le \frac{2 C \max\{1,d_\ell\}}{q}\,q^{N}
   \\&\le \frac{2 C \max\{1,d_\ell\}}{q}\,\frac{\lambda q}{2 C \max\{1,d_\ell\}}
   = \lambda,
\end{align*}
i.e., we conclude that~\eqref{eq:proposition_nbound} holds for $n=N$.
\end{proof}

\begin{remark}\label{remarkXXX}
Linear convergence (see Theorem~\ref{theorem:linear_convergence} below) proves, in particular, that $\eta_{\ell}(u_{\ell}^\inexact) \le \Clin\qlin\,\eta_{\ell-1}(u_{\ell-1}^\inexact)$ for all $\ell\in\N$. In practice, we observe that $\eta_{\ell}(u_{\ell}^\inexact)$ is only improved by a fixed factor, i.e., there also holds the converse estimate $\eta_{\ell-1}(u_{\ell-1}^\inexact) \le \widetilde C\,\eta_{\ell}(u_{\ell}^\inexact)$ for all $\ell\in\N$. Even though, we cannot thoroughly prove this fact,
Proposition~\ref{proposition:nested_iteration} shows that ---up to such an assumption--- the number of Picard iterations in each step of Algorithm~\ref{algorithm} is, in fact, uniformly bounded with $\#\pic(\ell)\le \Cpic+\frac{1}{|\log q|}\,\log(\max\{1,\widetilde C\})$.
\end{remark}

\subsection{Estimator convergence}
In this section, we show that, if step~(i) of Algorithm~\ref{algorithm} terminates for all $\ell \in \N_0$, 
then Algorithm~\ref{algorithm} yields 
$\eta_\ell(u_\ell^\inexact)\to0$ as $\ell\to\infty$.

We first show that
the iterates $u_\plus^\inexact = u_\plus^n$ of Algorithm~\ref{algorithm} are close to the non-computable 
exact
Galerkin approximation $u_\plus^\exact\in\XX_\plus$ to~\eqref{eq:discrete}
and that the corresponding error estimators are equivalent.

\begin{lemma}\label{lemma1}
Suppose~\eqref{axiom:monotone}--\eqref{axiom:lipschitz} for the 
operator $A$ and 
stability~\eqref{axiom:stability} for the 
error estimator. With $C_\lambda:=\Cstab\,\frac{q}{1-q}$, the following holds for all $0<\lambda<C_\lambda^{-1}$: For all 
$u_\plus^0\in\XX_\plus$ 
and all Picard iterates $u_\plus^n = \Phi_\plus(u_\plus^{n-1})$, $n\in\N$, with $\norm{u_\plus^n-u_\plus^{n-1}}\HH \le \lambda\,\eta_\plus(u_\plus^n)$, it holds that
\begin{align}\label{eq1:lemma1}
 \norm{u_\plus^\exact-u_\plus^n}\HH \le \lambda\,\frac{q}{1-q}\,\min\Big\{\eta_\plus(u_\plus^n)\,,\,\frac{1}{1-\lambda C_\lambda}\,\eta_\plus(u_\plus^\exact)\Big\}.
\end{align}
Moreover, there holds equivalence
\begin{align}\label{eq2:lemma1}
 (1-\lambda C_\lambda)\,\eta_\plus(u_\plus^n) \le \eta_\plus(u_\plus^\exact) \le (1+\lambda C_\lambda)\,\eta_\plus(u_\plus^n).
\end{align}
\end{lemma}

\begin{proof}
For $\TT_\plus=\TT_\circ$, stability~\eqref{axiom:stability} guarantees $|\eta_\plus(u_\plus^\exact)-\eta_\plus(u_\plus^n)|\le \Cstab\,\norm{u_\plus^\exact-u_\plus^n}\HH$. Therefore, estimate~\eqref{eq:estimate:discrete} and the assumption on the Picard iterate $u_\plus^n$ imply that
\begin{align*}
 \norm{u_\plus^\exact-u_\plus^n}\HH
 \reff{eq:estimate:discrete}\le \frac{q}{1\!-\!q}\,\norm{u_\plus^n-u_\plus^{n-1}}\HH
 &\le\lambda\,\frac{q}{1\!-\!q}\,\eta_\plus(u_\plus^n)
 \reff{axiom:stability}\le\lambda\,\frac{q}{1\!-\!q}\,\big(\eta_\plus(u_\plus^\exact) + \Cstab\,\norm{u_\plus^\exact-u_\plus^n}\HH\big).
\end{align*}
Since $0<\lambda<C_\lambda^{-1}$ and hence $\lambda\Cstab\,\frac{q}{1-q} = \lambda C_\lambda<1$, this yields that
\begin{align*}
 \norm{u_\plus^\exact-u_\plus^n}\HH \le\frac{ \lambda\,\frac{q}{1-q} }{ 1 - \lambda\Cstab\,\frac{q}{1-q} }\,\eta_\plus(u_\plus^\exact)
 = \lambda\,\frac{q}{1-q}\,\frac{1}{1-\lambda C_\lambda}\,\eta_\plus(u_\plus^\exact).
\end{align*}
Altogether, this proves~\eqref{eq1:lemma1}. Moreover, we see
\begin{align*}
 \eta_\plus(u_\plus^\exact)
 \reff{axiom:stability}\le \eta_\plus(u_\plus^n) + \Cstab\,\norm{u_\plus^\exact-u_\plus^n}\HH
 \reff{eq1:lemma1}
 \le (1+\lambda C_\lambda)\,\eta_\plus(u_\plus^n)
\end{align*}
as well as 
\begin{align*}
 \eta_\plus(u_\plus^n)
 \reff{axiom:stability}\le \eta_\plus(u_\plus^\exact) + \Cstab\,\norm{u_\plus^\exact-u_\plus^n}\HH
 \reff{eq1:lemma1}
 \le \Big(1+ \,\frac{\lambda C_\lambda}{1-\lambda C_\lambda}\Big)\,\eta_\plus(u_\plus^\exact)
 = \frac{1}{1-\lambda C_\lambda}\,\eta_\plus(u_\plus^\exact).
\end{align*}
This concludes the proof.
\end{proof}

The following proposition gives a first convergence result for Algorithm~\ref{algorithm}. Unlike the stronger convergence result of Theorem~\ref{theorem:linear_convergence} below, 
plain convergence only relies on~\eqref{axiom:monotone}--\eqref{axiom:lipschitz}, but avoids the use of~\eqref{axiom:potential}.

\begin{proposition}\label{prop:convergence}
Suppose~\eqref{axiom:monotone}--\eqref{axiom:lipschitz} for the 
operator $A$ and~\eqref{axiom:stability}--\eqref{axiom:reduction} for the 
error estimator. With $C_\lambda$ from Lemma~\ref{lemma1}, let $0<\theta\le1$ and $0<\lambda< C_\lambda^{-1} \theta $.
Then, Algorithm~\ref{algorithm} guarantees the existence of constants $0<\qest<1$ and $\Cest>0$
which depend only on~\eqref{axiom:stability}--\eqref{axiom:reduction} as well as $\lambda$ and $\theta$, such that the following implication holds: If the repeat loop of step~{\rm(i)} of Algorithm~\ref{algorithm} terminates after finitely many steps for all $\ell\in\N_0$, then 
\begin{align}\label{eq:reduction}
 \eta_{\ell+1}(u_{\ell+1}^\exact)^2 \le \qest\,\eta_\ell(u_\ell^\exact)^2 + \Cest\,\norm{u_{\ell+1}^\exact-u_\ell^\exact}\HH^2
 \quad\text{for all }\ell\in\N_0,
\end{align}
where $u_\plus^\exact\in\XX_\plus$ in the (non-computable) Galerkin solution to~\eqref{eq:discrete}.
Moreover, there holds estimator convergence $\eta_\ell(u_\ell^\inexact)\to0$ as $\ell\to\infty$.
\end{proposition}

\begin{proof}
We prove the assertion in three steps.

{\bf Step~1.}\quad
Arguing as in the proof of Lemma~\ref{lemma1}, stability~\eqref{axiom:stability} proves that
\begin{align*}
 \eta_\ell(\MM_\ell,u_\ell^\inexact) 
 \reff{axiom:stability}\le \eta_\ell(\MM_\ell,u_\ell^\exact) + \Cstab\,\norm{u_\ell^\exact-u_\ell^\inexact}\HH
 \reff{eq1:lemma1}\le \eta_\ell(\MM_\ell,u_\ell^\exact) + \lambda C_\lambda \,\eta_\ell(u_\ell^\inexact),
\end{align*}
where we have used that $C_\lambda = \Cstab\,\frac{q}{1-q}$. Together with the D\"orfler marking strategy in step~(ii) of Algorithm~\ref{algorithm}, this proves that

\begin{align*}
 \theta'\,\eta_\ell(u_\ell^\exact) := 
 \frac{\theta- \lambda C_\lambda}{1+\lambda C_\lambda}\,\eta_\ell(u_\ell^\exact)
 \reff{eq2:lemma1}\le (\theta- \lambda C_\lambda) \,\eta_\ell(u_\ell^\inexact)
 \stackrel{\rm(ii)}\le \eta_\ell(\MM_\ell,u_\ell^\inexact) - \lambda C_\lambda \, \eta_\ell(u_\ell^\inexact)
 \le\eta_\ell(\MM_\ell,u_\ell^\exact).
\end{align*}
Note that $\lambda < C_\lambda^{-1} \theta$ implies $\theta'>0$.
Hence, this is  the D\"orfler marking for $u_\ell^\exact$ with parameter $0<\theta'<\theta$.
Therefore,~\cite[Lemma~4.7]{axioms} proves~\eqref{eq:reduction}.

{\bf Step~2.}\quad
Next, we adopt an argument from~\cite{bv,msv} to prove {\sl a priori} convergence of the sequence $(u_\ell^\exact)_{\ell \in \N_0}$:
Since the discrete subspaces are nested, $\XX_\infty:=\overline{\bigcup_{\ell=0}^\infty\XX_\ell}$ is a closed subspace of $\HH$ and hence a Hilbert space. Arguing as above (for the continuous and discrete problem), there exists a unique solution $u_\infty^\exact\in\XX_\infty$ of
\begin{align*}
 \dual{Au_\infty^\exact}{v_\infty} = \dual{F}{v_\infty} 
 \quad\text{for all }v_\infty\in\XX_\infty.
\end{align*}
Note that $\XX_\ell\subseteq\XX_\infty$ implies that $u_\ell^\exact$ is a Galerkin approximation to $u_\infty^\exact$. Hence, the C\'ea lemma (Lemma~\ref{lemma:cea}) is valid with $u^\exact\in\HH$ replaced by $u_\infty^\exact \in\XX_\infty$. Together with the definition of $\XX_\infty$, this proves that
\begin{align*}
 \norm{u_\infty^\exact-u_\ell^\exact}\HH \le \frac{L}{\alpha}\,\min_{w_\ell\in\XX_\ell}\norm{u_\infty^\exact-w_\ell}\HH
 \xrightarrow{\ell\to\infty}0.
\end{align*}
In particular, we infer that $\norm{u_{\ell+1}^\exact-u_\ell^\exact}\HH^2\to0$ as $\ell\to\infty$.

{\bf Step~3.}\quad
According to~\eqref{eq:reduction} and step~2, the sequence $\big(\eta_\ell(u_\ell^\exact)\big)_{\ell\in\N_0}$ is contractive up to a non-negative perturbation which tends to zero. Basic calculus (e.g.,~\cite[Lemma~2.3]{estconv}) proves $\eta_\ell(u_\ell^\exact)\to0$ as $\ell\to\infty$. Lemma~\ref{lemma1} guarantees the equivalence 
 $\eta_\ell(u_\ell^\exact)\simeq\eta_\ell(u_\ell^\inexact)$. This concludes the proof.
\end{proof}

\begin{remark}
As in Proposition~\ref{prop:convergence}, the linear convergence result of Theorem~\ref{theorem:linear_convergence} below allows arbitrary $0<\theta\le1$, but requires $0<\lambda<C_\lambda^{-1}\,\theta$ with $C_\lambda>0$ being the constant from Lemma~\ref{lemma1}. In many situations, the weaker constraint $0<\lambda<C_\lambda^{-1}$ which avoids any coupling of $\theta$ and $\lambda$, appears to be sufficient to guarantee \emph{plain} convergence. To see this, note that usually the error estimator is equivalent to error plus data oscillations
\begin{align*}
 \eta_\ell(u_\ell^\exact) \simeq \norm{u^\exact-u_\ell^\exact}\HH + \osc_\ell(u_\ell^\exact).
\end{align*}
If the ``discrete limit space'' $\XX_\infty := \overline{\bigcup_{\ell=0}^\infty\XX_\ell}$ satisfies $\XX_\infty = \HH$, possible  smoothness of $A$ guarantees $\norm{u^\exact-u_\ell^\exact}\HH + \osc_\ell(u_\ell^\exact)\to0$ as $\ell\to\infty$; see the argumentation in~\cite[Proof of Theorem~3]{afvm}. Moreover, $\XX_\infty = \HH$ follows either implicitly  if $u^\exact$ is ``nowhere discrete'', or can explicitly be ensured by the marking strategy without deteriorating optimal convergence rates; see~\cite[Section~3.2]{helmholtz}. Since $0<\lambda<C_\lambda^{-1}$, Lemma~\ref{lemma1} guarantees estimator equivalence $\eta_\ell(u_\ell^\inexact) \simeq \eta_\ell(u_\ell^\exact)$. Overall, such a situation leads to $\eta_\ell(u_\ell^\inexact)\to0$ as $\ell\to\infty$.
\end{remark}

\section{Linear convergence}
\label{section:linear}

Suppose that $A$ additionally satisfies \eqref{axiom:potential}.
For $v \in \HH$, we define the energy $Ev:=\Re(P-F)v$, where $P$ is the potential associated with $A$ from~\eqref{eq:gateaux} and $F\in\HH^*$ is the right-hand side of~\eqref{eq:strongform}.
The next lemma generalizes~\cite[Lemma~16]{dk08} and~\cite[Theorem 4.1]{gmz} and states equivalence of the energy difference and the difference in norm. 

\begin{lemma}\label{lemma:energy eq}
Suppose~\eqref{axiom:monotone}--\eqref{axiom:potential}. Let $\XX_\plus$ be a closed subspace of $\HH$ (which also allows $\XX_\plus=\HH$).
If $u_\plus^\exact \in \XX_\plus$ denotes the corresponding Galerkin approximation~\eqref{eq':strongform}, it holds that
\begin{align}\label{eq:energy eq}
\frac{\alpha}{2}\norm{v_\plus-u_\plus^\exact}{\HH}^2\le E(v_\plus)-E(u_\plus^\exact)\le \frac{L}{2}\norm{v_\plus-u_\plus^\exact}{\HH}^2
\quad\text{for all }v_\plus\in\XX_\plus.
\end{align}
\end{lemma}
\begin{proof}
Since $\HH$ is also a Hilbert space over $\R$, we interpret $E$ as an $\R$-functional.
Since $F$ is linear with G\^ateaux derivative $\dual{{\rm d}F(v)}{w}=\dual{F}{w}$ for all $v,w\in\HH$,
the energy $E$ is also G\^ateaux differentiable with $\dual{{\rm d}E(v)}{w}=\Re\dual{{\rm d}P(v)-F}{w} =\Re \dual{Av -F}{w} $.
Define $\psi(t):=E(u_\plus^\exact+t (v_\plus-u_\plus))$ for $t\in[0,1]$.
For $t\in[0,1]$, it holds that
\begin{align}\label{eq:psiprime}
\begin{split}
\psi'(t)&=\lim_{\substack{r\to 0\\r\in\R}}\frac{E\big(u_\plus^\exact+t(v_\plus-u_\plus^\exact)+r(v_\plus-u_\plus^\exact)\big)
-E\big(u_\plus^\exact+t(v_\plus-u_\plus^\exact)\big)}{r}\\
&=\dual{{\rm d}E\big(u_\plus^\exact+t(v_\plus-u_\plus^\exact)\big)}{v_\plus-u_\plus^\exact}\\
&=\Re\dual{A\big(u_\plus^\exact+t(v_\plus-u_\plus^\exact)\big)-F}{v_\plus-u_\plus^\exact}.
\end{split}
\end{align}
Hence, $\psi$ is differentiable.
For $s,t\in[0,1]$, Lipschitz continuity \eqref{axiom:lipschitz} of $A$ proves that
\begin{align*}
|\psi'(s)-\psi'(t)|&=\big|\Re\dual{A\big(u_\plus^\exact+s(v_\plus-u_\plus^\exact)\big)-A\big(u_\plus^\exact+t(v_\plus-u_\plus^\exact)\big)}{v_\plus-u_\plus^\exact}\big|\\
&\le L\norm{(s-t)(v_\plus-u_\plus^\exact)}{\HH}\norm{v_\plus-u_\plus^\exact}{\HH}=L\norm{v_\plus-u_\plus^\exact}{\HH}^2|s-t|,
\end{align*}
i.e., $\psi'$ is Lipschitz continuous with constant $L\norm{v_\plus-u_\plus^\exact}{\HH}^2$.
By Rademacher's theorem, $\psi'$ is almost everywhere differentiable and 
$|\psi''|\le L\norm{v_\plus-u_\plus^\exact}{\HH}^2$ almost everywhere. 
Moreover, the fundamental theorem of calculus applies and integration by parts yields that
\begin{align*}
E(v_\plus)-E(u_\plus^\exact)=\psi(1)-\psi(0)=\psi'(0)+\int_0^1\psi''(t)(1-t)\, dt.
\end{align*}
Since $\XX_\plus\subset\HH$ is a closed subspace, 
there also holds  ${\rm d}P_\plus=A_\plus$  with the restriction $P_\plus:=P|_{\XX_\plus}$.  
Hence, we may also define the restricted energy $E_\plus:=E|_{\XX_\plus}$.
With \eqref{eq:psiprime} and ${\rm d}E_\plus u_\plus^\exact= A_\plus u_\plus^\exact-F_\plus=0$, we see $\psi'(0)=0$.
Therefore,
\begin{align}\label{eq:intform E}
E(v_\plus)-E(u_\plus^\exact)=\int_0^1\psi''(t)(1-t)\, dt.
\end{align}
Since $|\psi''|\le L\norm{v_\plus-u_\plus^\exact}{\HH}^2$ almost everywhere, we get the upper bound in \eqref{eq:energy eq}.
To see the lower bound,  we compute for almost every $t\in[0,1]$
\begin{align*}
\psi''(t)&\reff{eq:psiprime}{=} \lim_{\substack{r\to 0\\r\in\R}}\frac{\Re\dual{A\big(u_\plus^\exact+(t+r)(v_\plus-u_\plus^\exact)\big)-F}{v_\plus-u_\plus^\exact}
-\Re\dual{A\big(u_\plus^\exact+t(v_\plus-u_\plus^\exact)\big)-F}{v_\plus-u_\plus^\exact}}{r}\\
&=\lim_{\substack{r\to 0\\r\in\R}}\frac{\Re\dual{A\big(u_\plus^\exact+(t+r)(v_\plus-u_\plus^\exact)\big)- A\big(u_\plus^\exact+t(v_\plus-u_\plus^\exact)\big)}{r(v_\plus-u_\plus^\exact)}}{r^2}\\
&\reff{axiom:monotone}{\ge} \lim_{\substack{r\to 0\\r\in\R}}\alpha\frac{\norm{r(v_\plus-u_\plus^\exact)}{\HH}^2}{r^2}=\alpha\norm{v_\plus-u_\plus^\exact}{\HH}^2.
\end{align*}
Together with \eqref{eq:intform E}, we conclude the proof.
\end{proof}

\begin{remark}
Lemma~\ref{lemma:energy eq} immediately implies that the Galerkin solution $u_\bullet^\exact\in\XX_\bullet$ 
to \eqref{eq:discrete}  minimizes the energy $E$ in $\XX_\bullet$, i.e., $E(u_\bullet^\exact)\le E(v_\bullet)$ for all $v_\bullet\in\XX_\bullet$.
On the other hand, if $w_\bullet\in\XX_\bullet$ is a minimizer of the energy in $\XX_\bullet$, we deduce  $E(w_\bullet)=E(u_\bullet^\exact)$.
Lemma~\ref{lemma:energy eq} thus implies $w_\bullet=u_\bullet^\exact$.
Therefore, solving the Galerkin formulation \eqref{eq:discrete} is equivalent to the minimization of the energy $E$ in $\XX_\bullet$. 
\qed
\end{remark}

Next, we prove a contraction property as in~\cite[Theorem~20]{dk08}, \cite[Theorem~4.7]{bdk}, and~\cite[Theorem~4.2]{gmz} and, in particular, obtain 
linear convergence of Algorithm~\ref{algorithm} in the sense of \cite{axioms}. 

\begin{theorem}\label{theorem:linear_convergence}
Suppose~\eqref{axiom:monotone}--\eqref{axiom:potential} for the 
operator $A$
and~\eqref{axiom:stability}--\eqref{axiom:reliability} for the 
error estimator. 
 Let $C_\lambda$ be the constant from Lemma~\ref{lemma1}. 
 Let $0 < \theta \leq 1$ and suppose $0 < \lambda < C_\lambda^{-1} \theta $.
Then, there exist constants $0<\qlin<1$ and $\gamma>0$ which depend only on~\eqref{axiom:monotone}--\eqref{axiom:lipschitz} and
\eqref{axiom:stability}--\eqref{axiom:reliability}  as well as on $\lambda$ and $\theta$, such that the following implication holds: 
If the repeat loop of step~{\rm(i)} of Algorithm~\ref{algorithm} terminates after finitely many steps for all $\ell\in\N_0$, then 
\begin{align}\label{eq:contraction}
 \Delta_{\ell+1} \le \qlin\,\Delta_\ell
 \text{ for all }\ell\in\N_0,
 \quad\text{where}\quad
 \Delta_\plus := E(u_\plus^\exact)-E(u^\exact) + \gamma\,\eta_\plus(u_\plus^\exact)^{2}.
\end{align}
Moreover, there exists a constant $\Clin>0$ such that
\begin{align}
\eta_{\ell+k}(u_{\ell+k}^{\inexact})^2\le \Clin \qlin^k \eta_\ell(u_\ell^{\inexact})^2 \text{ for all }k,\ell\in\N_0.
\end{align}
\end{theorem}

\begin{proof}
Because of nestedness $\XX_\ell \subseteq \XX_{\ell+1} \subset \HH$ for all $\ell \in \N_0$,
Lemma~\ref{lemma:energy eq} proves
\begin{align}\label{eq:energy eq uk}
\frac{\alpha}{2}\norm{u_k^\exact-u_\ell^\exact}{\HH}^2\le E(u_k^\exact)-E(u_\ell^\exact)\le \frac{L}{2}\norm{u_k^\exact-u_\ell^\exact}{\HH}^2
\quad\text{for all }k,\ell,\in\N_0\text{ with }k \leq \ell
\end{align}
and
\begin{align}\label{eq:energy eq u}
\frac{\alpha}{2}\norm{u_k^\exact-u^\exact}{\HH}^2\le E(u_k^\exact)-E(u^\exact)\le \frac{L}{2}\norm{u_k^\exact-u^\exact}{\HH}^2\quad&\text{for all }k\in\N_0.
\end{align}
We set $\gamma:=\alpha / (2\Cest)$.
Together with~\eqref{eq:energy eq uk}, estimator reduction \eqref{eq:reduction} gives
\begin{align*}
\Delta_{\ell+1} &=E(u_{\ell+1}^\exact)-E(u^\exact)+\gamma\eta_{\ell+1}(u_{\ell+1}^\exact)^2\\
&\le  \big(E(u_\ell^\exact)-E(u^\exact)\big)-\big(E(u_\ell^\exact)-E(u_{\ell+1}^\exact)\big)+\gamma\big(\qest\eta_\ell(u_\ell^\exact)^2+\Cest\norm{u_\ell^\exact-u_{\ell+1}^\exact}{\HH}^2\big)\\
&\le  E(u_\ell^\exact)-E(u^\exact)+\gamma\qest \, \eta_\ell(u_\ell^\exact)^2.
\end{align*}
Let $\varepsilon >0$. Combining this estimate with 
reliability~\eqref{axiom:reliability} and \eqref{eq:energy eq u}, we see that
\begin{align*}
\Delta_{\ell+1}
&\leq  E(u_\ell^\exact)-E(u^\exact)+\gamma (\qest+\varepsilon) \, \eta_\ell(u_\ell^\exact)^2 - \gamma \varepsilon \, \eta_\ell(u_\ell^\exact)^2 \\
&\leq  E(u_\ell^\exact)-E(u^\exact)+\gamma (\qest+\varepsilon) \, \eta_\ell(u_\ell^\exact)^2 - \gamma \frac{2 \varepsilon}{L \Crelexact} \big(E(u_\ell^\exact)-E(u^\exact)\big)\\
&= (1-\gamma \frac{2 \varepsilon}{L \Crelexact}) \big(E(u_\ell^\exact)-E(u^\exact)\big) + \gamma (\qest+\varepsilon) \, \eta_\ell(u_\ell^\exact)^2 \\
&\leq \max \Big\{(1-\gamma \frac{2 \varepsilon}{L \Crelexact}),  (\qest+\varepsilon) \Big \}\, \Delta_\ell.
\end{align*}
Defining $0 < \qlin:= \inf_{\varepsilon >0 } \max \big\{(1-\gamma \frac{2 \varepsilon}{L \Crelexact}), 
 (\qest+\varepsilon) \big \} < 1$, we prove \eqref{eq:contraction}. Moreover, induction on $k$ proves that
$ \Delta_{\ell+k} \le \qlin^k\,\Delta_\ell
 \quad\text{for all }k,\ell\in\N_0.$
In combination with \eqref{eq:energy eq u}, reliability~\eqref{axiom:reliability} 
and the estimator equivalence~\eqref{eq2:lemma1} of Lemma~\ref{lemma1} prove that, for all $k,\ell\in\N_0$,
\begin{align*}
 \eta_{\ell+k}(u_{\ell+k}^{\inexact})^2 \stackrel{\eqref{eq2:lemma1}}{\simeq} \eta_{\ell+k}(u_{\ell+k}^\exact)^2
 \simeq \Delta_{\ell+k} \le \qlin^k\,\Delta_\ell
 \simeq \qlin^k\,\eta_{\ell}(u_{\ell}^\exact)^2 \stackrel{\eqref{eq2:lemma1}}{\simeq} \qlin^k\, \eta_{\ell+k}(u_{\ell+k}^{\inexact})^2.
\end{align*}
This concludes the proof. 
\end{proof}

\section{Optimal convergence rates}
\label{section:optimal_rates}

\subsection{Fine properties of mesh-refinement}
The proof of optimal convergence rates requires the following additional properties of the mesh-refinement strategy.
\begin{enumerate}
\renewcommand{\theenumi}{R\arabic{enumi}}
\bf
\item\label{axiom:split}
\rm
\textbf{splitting property:} Each refined element is split in at least 2 and at most in 
$\Cson\ge2$ many sons, i.e., for all $\TT_\plus \in \T$ and 
all $\MM_\plus \subseteq \TT_\plus$, the refined triangulation
$\TT_\circ = \refine(\TT_\plus, \MM_\plus)$ satisfies
\begin{align}\label{eq:split}
	\# (\TT_\plus \setminus \TT_\circ) + \# \TT_\plus \leq \# \TT_\circ
	\leq \Cson \, \# (\TT_\plus \setminus \TT_\circ) + \# (\TT_\plus \cap \TT_\circ).
\end{align}

\bf
\item\label{axiom:overlay} 
\rm
\textbf{overlay estimate:} For all meshes $\TT \in \T$ and $\TT_\plus,\TT_\circ \in \refine(\TT)$
there exists a common refinement $\TT_\plus \oplus \TT_\circ \in \refine(\TT_\plus) \cap \refine(\TT_\circ) \subseteq \refine(\TT)$
which satisfies
\begin{align*}
	\# (\TT_\plus \oplus \TT_\circ) \leq \# \TT_\plus + \# \TT_\circ - \# \TT.
\end{align*}

\bf
\item\label{axiom:mesh_closure} 
\rm
\textbf{mesh-closure estimate:}
 There exists $\Cmesh>0$ such that the sequence $\TT_\ell$ with corresponding 
 sets of marked elements $\MM_\ell \subseteq \TT_\ell$ which is generated by Algorithm~\ref{algorithm}, satisfies
\begin{align*}
	\# \TT_\ell  - \# \TT_0 \leq \Cmesh \sum_{j=0}^{\ell -1} \# \MM_j.
\end{align*}

\end{enumerate}

For newest vertex bisection (NVB), the mesh-closure estimate \eqref{axiom:mesh_closure} has first been proved for $d=2$ in~\cite{bdd} and later for $d \geq 2$ in \cite{stevenson08}. 
While both works require an additional admissibility assumption on $\TT_0$, \cite{kpp} proved that this condition is unnecessary for $d=2$. 
The proof of the overlay estimate \eqref{axiom:overlay} is found in \cite{ckns,stevenson07}.
The lower bound of \eqref{axiom:split} is clearly satisfied for each feasible mesh-refinement strategy. 
For NVB the upper bound of \eqref{axiom:split} is easily verified for $d=2$ with $\Cson=4$,  
and the proof for general dimension $d \geq 2$ can be found in \cite{gss14}.

For red-refinement with first-order hanging nodes, the validity of~\eqref{axiom:split}--\eqref{axiom:mesh_closure} is 
shown in~\cite{bn}. For mesh-refinement strategies in isogeometric analysis, we refer to~\cite{mp} for T-splines and to~\cite{bgmp16,ghp17} for (truncated) hierarchical B-splines.

\begin{remark}
	Using~\eqref{axiom:split} and the definition of $\refine(\cdot)$, an induction argument proves that the lower estimate 
	$\# (\TT_\bullet \setminus \TT_\circ) + \# \TT_\bullet \leq \# \TT_\circ$ in~\eqref{eq:split}
	holds true for all $\TT_\bullet \in \T$ and arbitrary refinements $\TT_\circ \in \refine(\TT_\bullet)$.	
\end{remark}

\subsection{Approximation class}
\label{section:approximation_class}
For $N\in\N_0$, we define the set 
\begin{align}\label{def:TN}
\T_N := \set{\TT_\plus \in \refine(\TT_0)}{\#\TT_\plus - \# \TT_0 \leq N },
\end{align}
of all refinements of $\TT_0$ which have at most $N$ elements more than $\TT_0$.
For $s>0$, we define the approximation norm $\norm{\cdot}{\A_s}$ by

\begin{align}\label{eq:approximatoin_class}
	\norm{u^\exact}{\A_s} := \sup_{N \in \N_0} \Big( (N+1)^s \min_{\TT_\plus \in \T_N} \eta_\plus(u_\plus^\exact) \Big),
\end{align}
where $\eta_\plus(u_\plus^\exact)$ is the error estimator corresponding to the optimal triangulation $\TT_\plus \in \T_N$.
Note that $\norm{u^\exact}{\A_s} < \infty$ implies the existence of a not (necessarily nested)
sequence of triangulations, such that the error estimator $\eta_\plus(u_\plus^\exact)$ corresponding 
to the (non-computable) Galerkin approximation $u_\plus^\exact$ decays at least with algebraic rate 
$s>0$.

\subsection{Main result}
The following theorem is the main result of this work.
It proves that Algorithm~\ref{algorithm} does not only lead
to linear convergence, but also guarantees the best possible algebraic convergence rate
for the error estimator  $\eta_\ell(u_\ell^\inexact)$.

\begin{theorem}\label{theorem:optimal_rate}
Suppose~\eqref{axiom:monotone}--\eqref{axiom:potential} for the nonlinear operator $A$, \eqref{axiom:split}--\eqref{axiom:mesh_closure} for the mesh-refinement, and~\eqref{axiom:stability}--\eqref{axiom:discrete_reliability} for the {\sl a~posteriori} error estimator. 
Let $C_\lambda$ be the constant from Lemma~\ref{lemma1}.
Suppose that $0 < \theta \leq 1$ and $0<\lambda< C_\lambda^{-1} \theta$ satisfy
\begin{align}\label{eq:lambda3}
 \theta'' := \frac{\theta+\lambda C_\lambda}{1-\lambda C_\lambda} < \theta_{\rm opt}:= (1+\Cstab^2 (\Cdrel)^2)^{-1}
\end{align}		
(which is satisfied, e.g., for $0<\theta<\theta_{\rm opt}$ and sufficiently small $\lambda$).
Suppose that the repeat loop of step~{\rm(i)} of Algorithm~\ref{algorithm} terminates after finitely many steps for all $\ell\in\N_0$.
	 Then, for all $s>0$, there holds the equivalence
	 \begin{align}\label{eq:theorem_optimal_rate}
	 	\norm{u^\exact}{\A_s} < \infty \quad \Longleftrightarrow \quad 
		\exists \Copt>0\,\forall \ell\in\N_0\quad
		\eta_\ell(u_\ell^\inexact) \leq \Copt \big( \#\TT_\ell - \# \TT_0 +1  \big)^{-s}. 
	 \end{align}
	 Moreover, there holds $\Copt = \Copt'\norm{u^\exact}{\A_s}$, where $\Copt'>0$ 
	 depends only on $\TT_0$, $\theta$, $\lambda$, $s$, \eqref{axiom:stability}--\eqref{axiom:discrete_reliability}, \eqref{axiom:monotone}--\eqref{axiom:lipschitz},
	 and~\eqref{axiom:split}--\eqref{axiom:mesh_closure}.
\end{theorem}

The comparison lemma is found in \cite{axioms}.

\def\Ccomp{C_{\rm comp}}
\begin{lemma}[{\cite[Lemma~4.14]{axioms}}]
\label{lemma:doerfler2}
Suppose~\eqref{axiom:overlay},~\eqref{axiom:stability},~\eqref{axiom:reduction}, and~\eqref{axiom:discrete_reliability}. 
Let $0<\theta''<\theta_{\rm opt}$.
Then, there exists a constant $\Ccomp>0$ such that for all $s>0$ with $\norm{u^\exact}{\A_s} < \infty$ and all $\ell \in \N_0$, there exists $\RR_\ell \subseteq \TT_\ell $ which satisfies
\begin{align}\label{eq:doerfler2_Rbound}
\# \RR_\ell \leq \Ccomp \norm{u^\exact}{\A_s}^{1/s} \eta_\ell(u_\ell^\exact)^{-1/s},
\end{align}
as well as the D\"orfler marking criterion
\begin{align}\label{eq:doerfler2_doerfler}
\theta'' \eta_\ell(u_\ell^\exact) \leq \eta_\ell(\RR_\ell, u_\ell^\exact).
\end{align}
The constant $\Ccomp$ depends only on $\theta'',s$, and the constants of~\eqref{axiom:stability},~\eqref{axiom:reduction} and~\eqref{axiom:discrete_reliability}. \qed
\end{lemma}

The proof of Theorem~\ref{theorem:optimal_rate} follows ideas from \cite[Theorem~4.1]{axioms}.

\begin{proof}[Proof of Theorem~\ref{theorem:optimal_rate}]
We prove the assertion in three steps.

{\bf Step~1.}\quad The implication ``$\Longleftarrow$'' follows by definition of the approximation class,
the equivalence $\eta_\ell(u_\ell^\exact)\simeq\eta_\ell(u_\ell^\inexact)$ from Lemma~\ref{lemma1},
and the upper bound of \eqref{axiom:split} (cf.~\cite[Proposition~4.15]{axioms}). 
We thus focus on the converse, more important implication  ``$\Longrightarrow$''. 

{\bf Step~2.}\quad
Suppose $\norm{u^\exact}{\A_s} < \infty$. 
By Assumption~\eqref{eq:lambda3}, Lemma~\ref{lemma:doerfler2} provides a set $\RR_\ell\subseteq\TT_\ell$ with~\eqref{eq:doerfler2_Rbound}--\eqref{eq:doerfler2_doerfler}. 
Arguing as in the proof of Lemma~\ref{lemma1}, stability~\eqref{axiom:stability} proves that
\begin{align*}
 \eta_\ell(\RR_\ell,u_\ell^\exact) 
 \reff{axiom:stability}\le \eta_\ell(\RR_\ell,u_\ell^\inexact) + \Cstab\,\norm{u_\ell^\exact-u_\ell^\inexact}\HH
 \reff{eq1:lemma1}\le \eta_\ell(\RR_\ell,u_\ell^\inexact) + \lambda C_\lambda\,\eta_\ell(u_\ell^\inexact),
\end{align*}
where we have used that $C_\lambda = \Cstab\,\frac{q}{1-q}$. Together with $\theta'' \eta_\ell(u_\ell^\exact) \le \eta_\ell(\RR_\ell,u_\ell^\exact)$, this proves
\begin{align*}
 (1-\lambda C_\lambda)\theta''\,\eta_\ell(u_\ell^\inexact)
 \reff{eq2:lemma1}\le \theta''\,\eta_\ell(u_\ell^\exact)
 \le\eta_\ell(\RR_\ell,u_\ell^\exact)
 \le\eta_\ell(\RR_\ell,u_\ell^\inexact) + \lambda C_\lambda\,\eta_\ell(u_\ell^\inexact)
\end{align*}
and results in
\begin{align}\label{eq37*}
 \theta\,\eta_\ell(u_\ell^\inexact) \reff{eq:lambda3}= \Big((1-\lambda C_\lambda)\theta'' - \lambda C_\lambda\Big)\,\eta_\ell(u_\ell^\inexact) \le \eta_\ell(\RR_\ell,u_\ell^\inexact).
\end{align}
Hence, $\RR_\ell$ satisfies the D\"orfler marking for $u_\ell^\inexact$ with parameter $\theta$. By choice of $\MM_\ell$ in step~(ii) of Algorithm~\ref{algorithm}, we thus infer that
 \begin{align*}
 	\# \MM_\ell \stackrel{\eqref{eq37*}}{\leq} \Cmark \# \RR_\ell 
 	\stackrel{\eqref{eq:doerfler2_Rbound}}{\leq} \Cmark \Ccomp \norm{u^\exact}{\A_s}^{1/s} \eta_\ell(u_\ell^\exact)^{-1/s}
	\reff{eq2:lemma1}\simeq \norm{u^\exact}{\A_s}^{1/s}\,\eta_\ell(u_\ell^\inexact)^{-1/s}
	\quad\text{for all $\ell\in\N_0$}.
 \end{align*}
The mesh-closure estimate~\eqref{axiom:mesh_closure} guarantees that
\begin{align}\label{eq:theorem_optimal_rate_proof}
	\# \TT_\ell - \# \TT_0 +1 
	\lesssim \sum_{j=0}^{\ell-1} \# \MM_j 
	\lesssim \norm{u^\exact}{\A_s}^{1/s}\,\sum_{j=0}^{\ell-1} \eta_j(u_j^{\inexact})^{-1/s}  \quad \text{for all } \ell>0.
\end{align}

{\bf Step~3.}\quad  
The linear convergence of Theorem \ref{theorem:linear_convergence} implies $\eta_\ell(u_\ell^\inexact) \leq \Clin \qlin^{\ell - j} \eta_j(u_j^\inexact)$ for all $0 \leq j \leq \ell$.
In particular, this leads to
\begin{align*}
	\eta_j(u_j^{\inexact})^{-1/s} \leq \Clin^{1/s} \qlin^{(\ell -j )/s} \eta_\ell(u_\ell^{\inexact})^{-1/s}
	\quad\text{for all }0\le j\le \ell.
\end{align*}
By use of the geometric series with $0 < \qlin^{1/s} <1$,  we obtain that
\begin{align*}
	\sum_{j=0}^{\ell-1} \eta_j(u_j^{\inexact})^{-1/j} 
	\leq \Clin^{1/s} \eta_\ell(u_\ell^{\inexact})^{-1/s} \sum_{j=0}^{\ell-1} (\qlin^{1/s})^{\ell -j} 
	\leq \frac{\Clin^{1/s}}{1-\qlin^{1/s}} \, \eta_\ell(u_\ell^{\inexact})^{-1/s}.
\end{align*}
Combining the latter estimate with \eqref{eq:theorem_optimal_rate_proof}, we derive that
\begin{align*}
		\# \TT_\ell - \# \TT_0 +1 \lesssim \norm{u^\exact}{\A_s}^{1/s}\,\eta_\ell(u_\ell^{\inexact})^{-1/s}  \quad \text{for all } \ell>0.
\end{align*}
Since $\eta_0(u_0) \simeq \eta_0(u_0^\exact) \lesssim \norm{u^\exact}{\A_s}$, the latter inequality holds, in fact, for all $\ell \geq 0$.
Rearranging this estimate, we conclude the proof of \eqref{eq:theorem_optimal_rate}.
\end{proof}

\subsection{(Almost) Optimal computational work}
We show that Algorithm~\ref{algorithm} does not only lead to optimal algebraic convergence rates
for the error estimator $\eta_\ell(u_\ell^\inexact)$, but also guarantees that 
the overall cost of the algorithm is asymptotically (almost) optimal. 

\begin{itemize}
\item 
We suppose that the linear system involved in the computation of each step of the discrete Picard iteration 
$u_\ell^n := \Phi_\ell(u_\ell^{n-1})$, see Section~\ref{section:discreteproblem}, can be solved (e.g., by multigrid)
in linear complexity $\OO(\# \TT_\ell)$. Morever, we suppose that the evaluation of $\dual{A u_\ell^{n-1} -F}{v_\ell}$ and $\eta_\ell(T,v_\ell)$
for one fixed $v_\ell \in \XX_\ell$ and $T\in\TT_\ell$ is of order $\OO(1)$. Then, with $\#\pic(\ell)\ge1$ the number of Picard iterates 
in step~{\rm (i)} of Algorithm~\ref{algorithm}, we require $\OO(\# \pic(\ell) \, \# \TT_\ell)$ operations to compute the discrete 
solution $u_\ell \in \XX_\ell$. 
\item
We suppose that the set $\MM_\ell$ in step~{\rm (ii)} as well as the local
mesh-refinement $\TT_{\ell+1} := \refine(\TT_\ell,\MM_\ell)$ in step~{\rm (iii)} of Algorithm~\ref{algorithm} are
performed in linear complexity $\OO(\# \TT_\ell)$; see, e.g., \cite{stevenson07} with $\Cmark = 2$ for step~{\rm (ii)}.
\end{itemize}
Since one step of the adaptive algorithm depends on the full history of the adaptive meshes, the overall computational cost for 
the $\ell$-th step of Algorithm~\ref{algorithm} thus amounts to 
\begin{align*}
	\OO \Big( \sum_{j=0}^{\ell}  \# \pic(j)\,\# \TT_j \Big).
\end{align*}
\emph{Optimal convergence behavior} of Algorithm~\ref{algorithm} means that, given $\norm{u^\exact}{\A_s} < \infty$, the error estimator
$\eta_\ell(u_\ell)$ decays with rate $s>0$ with respect to the degrees of freedom $\OO(\# \TT_\ell)$; see Theorem~\ref{theorem:optimal_rate}.
\emph{Optimal computational complexity} would mean that, given $\norm{u^\exact}{\A_s} < \infty$, the error estimator $\eta_\ell(u_\ell)$ decays with rate $s>0$ with respect to the computational cost; 
see~\cite{feischlphd} for linear problems. Up to some small perturbation, the latter is stated in the following theorem.

\begin{theorem}\label{theorem:optimal_comp}
Suppose the assumptions of Theorem~\ref{theorem:optimal_rate} and $s>0$. Then, it holds that 
\begin{align*}
\norm{u^\exact}{\A_s} < \infty \, \Longrightarrow \, \forall \varepsilon>0 \, \exists \Cwork>0 \, \forall \ell \in \N_0 \,\,\,
\eta_\ell(u_\ell^\inexact) \leq \Cwork \Big(  \sum_{j=0}^{\ell}  \# \pic(j)\,\# \TT_j \Big)^{-(s - \varepsilon)} 
\end{align*}
There holds $\Cwork = \Cwork'\,\norm{u^\exact}{\A_s}$, where $\Cwork'>0$ depends only on $\theta$, $\lambda$, $s$, $\varepsilon$, \eqref{axiom:stability}--\eqref{axiom:discrete_reliability}, \eqref{axiom:monotone}--\eqref{axiom:lipschitz},
	 and~\eqref{axiom:split}--\eqref{axiom:mesh_closure}, as well as on $\TT_0$, $\eta_0(u_0)$, and $\#\pic(0)$. 
\end{theorem}

\begin{proof}
We prove the assertion in three steps.

{\bf Step~1.}\quad We show that there exist constants $C_0,C >0$ such that 
\begin{align}\label{eq:comp}
 \Big( \sum_{j=0}^{\ell}  \# \pic(j)\,\# \TT_j \Big)^s \Big( \log(C_0 /  \eta_\ell(u_\ell^\inexact)) \Big)^{-s} \eta_\ell(u_\ell^\inexact)  \leq C  \quad \text{for all } \ell \geq 0. 
\end{align}
Proposition~\ref{proposition:nested_iteration} gives a bound for the number of Picard iterations. 
For $j \geq 1$, it holds that
\begin{align}\label{eq:com:picard}
\begin{split}
	\# \pic(j) &\reff{eq:nested_iteration}\leq \Cpic + \frac{1}{|\log q|} \, \log \big(   \max \big\{ 1 \,,\, \eta_{j-1}(u_{j-1}^\inexact) / \eta_j(u_j^\inexact) \big\}  \big) \\
	&\leq \Cpic + \frac{1}{|\log q|} \big| \log\big(\eta_{j-1}(u_{j-1}^\inexact) / \eta_j(u_j^\inexact) \big) \big| \\
	&\leq \Cpic + \frac{1}{|\log q|}  \Big( \big| \log(\eta_{j-1}(u_{j-1}^\inexact)) \big| + \big| \log(\eta_j(u_j^\inexact))\big| \Big).
\end{split}
\end{align}
Quasi monotonicity~\cite[Lemma 3.5]{axioms} and Lemma~\ref{lemma1} imply 
$\eta_{\ell}(u_{\ell}^\inexact) \simeq \eta_\ell(u_\ell^\exact) \lesssim \eta_{k}(u_{k}^\exact) \simeq \eta_{k}(u_{k}^\inexact)$
and hence $\eta_{\ell}(u_{\ell}^\inexact) \leq \Cmon \eta_{k}(u_{k}^\inexact)$ for all $k \leq \ell$.
The constant $\Cmon>0$ depends only on~\eqref{axiom:stability}--\eqref{axiom:discrete_reliability}. 
With $C_0:=  e \, \Cmon \eta_0(u_0^\inexact)$,  there holds $\Cmon^{-1} C_0^{-1} \eta_{\ell} (u_{\ell}^\inexact) \leq C_0^{-1} \eta_{k}(u_{k}^\inexact) \leq e^{-1} <1 $ for all 
$k \leq \ell$. 
Hence, we obtain that
\begin{align*}
&\big| \log(\eta_{k}(u_{k}^\inexact)) \big| = \big| \log(C_0) +  \log(C_0^{-1} \eta_{k}(u_{k}^\inexact)) \big|
\leq  \big| \log(C_0)\big| + \big| \log(C_0^{-1} \Cmon^{-1} \eta_{\ell}(u_{\ell}^\inexact)) \big| \\
&\qquad
\leq \big| \log(C_0)\big| + \big| \log(\Cmon^{-1}) \big| + \big|  \log(C_0^{-1} \eta_{\ell}(u_{\ell}^\inexact)) \big|
=C' +  \log( C_0 / \eta_{\ell}(u_{\ell}^\inexact)),
\end{align*}
where the constant $C' := \big| \log(C_0)\big| + \big| \log(\Cmon^{-1}) \big|$ depends only on $\eta_0(u_0^\inexact)$ and $\Cmon$.
Combining this estimate for $1 \leq j \leq \ell$ with~\eqref{eq:com:picard}, we obtain that
\begin{align*}
	\# \pic(j) \leq C'' + 	\frac{2}{|\log q|}  \log(C_0 / \eta_{\ell}(u_{\ell}^\inexact)) \quad \text{with }  C'':= \Cpic + C'.	
\end{align*}
By definition of $C_0$, it holds that $C_0 /  \eta_{\ell}(u_{\ell}^\inexact) \geq e$ and hence $\log(C_0 / \eta_{\ell}(u_{\ell}^\inexact)) \geq \log(e) = 1$ for all $\ell \geq 0$.
This yields that
\begin{align*}
\sum_{j=1}^{\ell}  \# \pic(j)\,\# \TT_j &\leq \sum_{j=1}^{\ell} \# \TT_j \Big( C'' + \frac{2}{|\log q|}   \log(C_0 / \eta_{\ell}(u_{\ell}^\inexact))\Big) \\
&\leq  \Big(C'' + \frac{2}{|\log q|} \Big) \log(C_0 / \eta_{\ell}(u_{\ell}^\inexact)) \sum_{j=1}^{\ell} \# \TT_j.
\end{align*}
Using $\# \TT_j \leq \# \TT_0 (\# \TT_j - \# \TT_0 +1)$ (see, e.g., \cite[Lemma 22]{helmholtz}) and $C''' := \max\big\{C'' + {2}/|\log q|\,,\,\#\pic(0)\big\}$ as well as Theorem~\ref{theorem:optimal_rate}, we obtain that
\begin{align*}
	\sum_{j=0}^{\ell}  \# \pic(j)\,\# \TT_j 
	&\leq C''' \# \TT_0 \log(C_0 / \eta_{\ell}(u_{\ell}^\inexact)) \sum_{j=0}^{\ell} (\# \TT_j - \# \TT_0 +1) \\
	&\stackrel{\eqref{eq:theorem_optimal_rate}}{\leq} C''' \# \TT_0 \Copt^{-1/s} \log(C_0 / \eta_{\ell}(u_{\ell}^\inexact)) \sum_{j=0}^{\ell}  \eta_j(u_j^{\inexact})^{-1/s}
\end{align*} 
We argue as in the proof of Theorem~\ref{theorem:optimal_rate}.
The linear convergence from Theorem~\ref{theorem:linear_convergence} with $0 < \qlin^{1/s} < 1$ implies that
\begin{align}
\sum_{j=0}^{\ell}  \# \pic(j)\,\# \TT_j 
&\leq  C''' \, \# \TT_0 \, \Copt^{-1/s} \,\frac{\Clin^{1/s}}{1-\qlin^{1/s}} \,\, \log(C_0 / \eta_{\ell}(u_{\ell}^\inexact)) \, \eta_\ell(u_\ell^{\inexact})^{-1/s}.
\end{align}
Rearranging the terms we conclude~\eqref{eq:comp} with $C:=\Big( C''' \# \TT_0 \, \frac{\Clin^{1/s}}{1-\qlin^{1/s}} \Big)^s\,\Copt^{-1}$. \\

{\bf Step~2.}\quad Let $s,\delta>0$. Recall that $t^{\delta/s} \log(C_0 /t) \to 0$ as $t \to 0$. Hence, it follows
$\big( \log(C_0 / \eta_\ell(u_\ell\inexact))\big) \eta_\ell(u_\ell\inexact)^{\delta/s} \lesssim 1$ as $\ell \to \infty$.
This implies $\big( \log(C_0 / \eta_\ell(u_\ell\inexact) )\big)^{-s} \eta_\ell(u_\ell\inexact)^{-\delta} \gtrsim 1$ 
and results in
$\big( \log(C_0 / \eta_\ell(u_\ell\inexact) )\big)^{-s} \eta_\ell(u_\ell\inexact) 
	 \gtrsim \eta_\ell(u_\ell\inexact)^{1+\delta}$,
where the hidden constant depends only on $\delta$, $s$, $C_0$, and $\max_{\ell\in\N_0}\eta_\ell(u_\ell)\le\Cmon\,\eta_0(u_0)$.

{\bf Step~3.}\quad From step 1 and step~{\rm 2}, we infer that
\begin{align*}
	\Big( \sum_{j=0}^{\ell} \# \TT_j (1 + \# \pic(j)) \Big)^s \eta_\ell(u_\ell\inexact)^{1+\delta} \leq \Cwork,
\end{align*}
where $\Cwork\ge1$ depends only on $C$, $\delta$, $s$, $C_0$, $\Cmon$, and $\eta_0(u_0)$.
Choose $\delta>0$ such that $\frac{s}{1+\delta} = s- \varepsilon$. Then, we finally obtain that
\begin{align*}
	\Big( \sum_{j=0}^{\ell} \# \TT_j (1 + \# \pic(j)) \Big)^{s- \varepsilon} \eta_\ell(u_\ell\inexact) \leq \Cwork^{1/(1+\delta)} 
	= \Cwork^{(s-\varepsilon)/s} \leq \Cwork.
\end{align*}
Rearranging the terms concludes the proof.
\end{proof}

\section{Optimal convergence of full sequence}
\label{section:full}
In this section, we reformulate Algorithm~\ref{algorithm} in the sense that the discrete solutions $(u_\ell^\inexact)_{\ell\in\N_0}$ correspond to a subsequence $(\widetilde u_{\ell_k})_{k\in\N_0}$ obtained by the following algorithm which also accounts for the Picard iterates; see Remark~\ref{remark:interpretation} below for details. To distinguish between the quantities of Algorithm~\ref{algorithm} and Algorithm~\ref{algorithm:new}, we use the tilde for all quantities of Algorithm~\ref{algorithm:new}, e.g., $\widetilde\eta_\ell(\widetilde u_\ell)$ is the error estimator corresponding to some $\widetilde u_\ell\in\widetilde\XX_\ell$.

\bigskip

\begin{algorithm}\label{algorithm:new}
\textsc{Input:} Initial triangulation $\widetilde\TT_0:=\TT_0$, adaptivity parameters $0<\theta\le1$, $\lambda\ge0$, and $\Cmark\ge1$, initial guess $\widetilde u_{-1}:=0$.\\
\textsc{Adaptive loop:} For all $\ell=0,1,2,\dots$, iterate the following steps~{\rm(i)--(iv)}.
\begin{enumerate}
\item[\rm(i)] Compute discrete Picard iterate $\widetilde u_\ell = \widetilde\Phi_\ell \widetilde u_{\ell-1}\in\widetilde\XX_\ell$.
\item[\rm(ii)] Compute refinement indicators $\widetilde\eta_\ell(T,\widetilde u_\ell)$ for all $T\in\widetilde\TT_\ell$.
\item[\rm(iii)] If $\norm{\widetilde u_\ell-\widetilde u_{\ell-1}}\HH \le \lambda\,\widetilde\eta_\ell(\widetilde u_\ell)$, do the following steps~{\rm(a)--(b)}.
\begin{enumerate}
\item[\rm(a)] Determine a set $\widetilde\MM_\ell\subseteq\widetilde\TT_\ell$ of marked elements which has minimal cardinality up to the multiplicative constant $\Cmark$ and which satisfies the D\"orfler marking criterion $\theta\,\widetilde\eta_\ell(\widetilde u_\ell) \le \widetilde\eta_\ell(\widetilde\MM_\ell,\widetilde u_\ell)$.
\item[\rm(b)] Generate the new triangulation $\widetilde\TT_{\ell+1} := \refine(\widetilde\TT_\ell,\widetilde\MM_\ell)$ by refinement of (at least) all marked elements $T\in\widetilde\MM_\ell$.
\end{enumerate}
\item[\rm(iv)] Else, define $\widetilde\TT_{\ell+1} := \widetilde\TT_\ell$.
\end{enumerate}
\textsc{Output:} Sequence of discrete solutions $\widetilde u_\ell\in\widetilde\XX_\ell$ and corresponding estimators $\widetilde\eta_\ell(\widetilde u_\ell)$.\qed
\end{algorithm}

\begin{remark}\label{remark:interpretation}
With $\ell_{-1}:=-1$ and $\widetilde\TT_{-1}:=\TT_0$, define an index sequence $(\ell_k)_{k\in\N_0}$ inductively as follows:
\begin{itemize}
 \item $\ell_k > \ell_{k-1}$ is the smallest index such that $\widetilde\TT_{\ell_k+1}\neq\widetilde\TT_{\ell_k}$ 
 \item resp.\ $\ell_k=\infty$ if such an index does not exist.
\end{itemize}
In explicit terms, the indices $\ell_k$ are chosen such that the mesh is refined after the computation of $\widetilde u_{\ell_k}$. Hence,
Algorithm~\ref{algorithm} and Algorithm~\ref{algorithm:new} are related as follows:
\begin{itemize}
\item $\widetilde u_{\ell_k} = u_k^\inexact$ for all $k\in\N_0$ with $\ell_k<\infty$;
\item $\widetilde u_{\ell_{k-1}+m} = u_k^m$ for all $k\in\N_0$ with $\ell_{k-1}<\infty$ and all $m\in\N_0$ with $\ell_{k-1}+m \le \ell_k$;
\item $\widetilde\TT_\ell = \TT_k$ for all $\ell=\ell_{k-1}+1,\dots,\ell_k$ if $\ell_k<\infty$;
\item $\widetilde\TT_\ell = \TT_k$ for all $\ell\ge \ell_{k-1}+1$ if $\ell_{k-1}<\infty$ and $\ell_k=\infty$.
\end{itemize}
\end{remark}

The observations of Remark~\ref{remark:interpretation} allow to transfer the results for Algorithm~\ref{algorithm} to Algorithm~\ref{algorithm:new}.
As a consequence of our preceding analysis, we get the following theorem:

\begin{theorem}
Suppose~\eqref{axiom:monotone}--\eqref{axiom:potential} for the nonlinear operator 
and~\eqref{axiom:stability}--\eqref{axiom:reliability} for the {\sl a~posteriori} error estimator. Let $C_\lambda$ be the constant from Lemma~\ref{lemma1}. Let $0 < \theta \leq 1$ and suppose $0 < \lambda < C_\lambda^{-1} \theta $.
Then, the output of Algorithm~\ref{algorithm:new} satisfies
\begin{align}\label{eq1:theorem:new}
\lim_{\ell\to\infty}\norm{u^\exact-\widetilde u_\ell}\HH = 0 = \lim_{\ell\to\infty}\widetilde\eta_\ell(\widetilde u_\ell).
\end{align}
Additionally, suppose \eqref{axiom:split}--\eqref{axiom:mesh_closure} for the mesh-refinement, \eqref{axiom:discrete_reliability} for the {\sl a~posteriori} error estimator, and that
\begin{align}\label{eq':lambda3}
 \theta'' := \frac{\theta+\lambda C_\lambda}{1-\lambda C_\lambda} < \theta_{\rm opt}:= \big(1+\Cstab^2 (\Cdrel)^2\big)^{-1}.
\end{align}
\renewcommand{\thefootnote}{$\dagger$}
Then, for all $s>0$, it holds that \footnote{It is necessary to demand $\ell\ge\ell_0$. Indeed, the finitely many estimator values $\widetilde\eta_\ell(\widetilde u_\ell)$ for $\ell<\ell_0$ are obviously bounded by a constant.
However, this constant does not only depend on the  constants on which $\widetilde C_{\rm opt}$ resp.\ $\widetilde C_{\rm work}$ depend.}
	 \begin{align}\label{eq':theorem_optimal_rate}
	 	\norm{u^\exact}{\A_s} < \infty \quad \Longleftrightarrow \quad 
		\exists\widetilde C_{\rm opt}>0\,\forall \ell\ge\ell_0\quad\widetilde\eta_\ell(\widetilde u_\ell) \leq \widetilde C_{\rm opt} \big( \#\widetilde\TT_\ell - \#\TT_0 +1  \big)^{-s}.
	 \end{align}
	Moreover, if $\# \TT_\ell \to \infty$ as $\ell \to \infty$, then it holds that $^\dagger$
	 \begin{align}\label{eq':theorem_optimal_complexity}
	 \norm{u^\exact}{\A_s} < \infty \quad \Longrightarrow \quad \forall \varepsilon >0 \, \exists \tildeCwork>0 \, \forall \ell \ge \ell_0
	 \quad \widetilde \eta_\ell(\widetilde u_\ell) \leq  \tildeCwork \Big( \sum_{j=0}^{\ell} \# \widetilde\TT_j \Big)^{-(s-\varepsilon)}. 
	 \end{align}
	 	 Finally, it holds $\widetilde C_{\rm opt} = C\,C_{\rm opt}$ and $\tildeCwork = C\,\Cwork$, where $C_{\rm opt}$ is the constant from Theorem~\ref{theorem:optimal_rate}, $\Cwork>0$ is the constant from Theorem~\ref{theorem:optimal_comp}, and $C>0$ depends only on~\eqref{axiom:monotone}--\eqref{axiom:lipschitz}, \eqref{axiom:reliability}, \eqref{axiom:split}, $\TT_0$, and $s$. 
\end{theorem}

\begin{proof} We stick with the notation of Remark~\ref{remark:interpretation}.

{\bf Step~1.}\quad Suppose that there exists an index $\ell\in\N_0$ with $\norm{\widetilde u_\ell-\widetilde u_{\ell-1}}\HH \le \lambda\,\widetilde\eta_\ell(\widetilde u_\ell)$ and $\widetilde\MM_\ell = \emptyset$. Then, it follows $\widetilde\eta_\ell(\widetilde u_\ell)=0$. Proposition~\ref{prop:neww} concludes $u^\exact = \widetilde u_k$, $\widetilde\MM_k = \emptyset$, and $\widetilde\eta_k(\widetilde u_k)=0$ for all $k\ge\ell$. In particular, this proves~\eqref{eq':lambda3}--\eqref{eq':theorem_optimal_rate} with $\norm{u^\exact}{\A_s}<\infty$ for all $s>0$.

{\bf Step~2.} Suppose that there exists an index $k\in\N_0$ with $\ell_{k-1}<\infty$ and $\ell_k=\infty$. Then, $\widetilde\TT_\ell=\widetilde\TT_{\ell_{k-1}+1}$ for all $\ell>\ell_{k-1}$. According to step~1, we may further assume that $\norm{\widetilde u_\ell-\widetilde u_{\ell-1}}\HH > \lambda\,\widetilde\eta_\ell(\widetilde u_\ell)$ for all $\ell>\ell_{k-1}$. Then, this corresponds to the situation of Proposition~\ref{proposition:repeat} and hence results in convergence $\norm{u^\exact-\widetilde u_\ell}\HH + \widetilde\eta_\ell(\widetilde u_\ell)\to0$ as $\ell\to\infty$ and $\widetilde\eta_\ell(u^\exact)=0$ for all $\ell>\ell_{k-1}$. Again, this proves~\eqref{eq':lambda3}--\eqref{eq':theorem_optimal_rate} with $\norm{u^\exact}{\A_s}<\infty$ for all $s>0$.

{\bf Step~3.} According to step~2, we may suppose $\ell_k<\infty$ for all $k\in\N_0$. In this step, we will prove~\eqref{eq1:theorem:new}.
Given $\ell\in\N$, choose $k\in\N_0$ and $m\in\N_0$ with
$\ell = \ell_{k-1} + m < \ell_{k}<\infty$. Then, $\widetilde u_\ell = \widetilde u_{\ell_{k-1}+m} = u_{k}^m$ and hence
\begin{align*}
 \norm{u^\exact-\widetilde u_\ell}\HH = \norm{u^\exact-u_{k}^m}\HH
 &\le \norm{u^\exact-u_{k}^\inexact}\HH + \norm{u_{k}^\inexact-u_{k}^m}\HH
 \stackrel{\eqref{eq:reliability}}{\le} \Crel^\inexact\,\eta_k(u_{k}^\inexact) + \norm{u_{k}^\inexact-u_{k}^m}\HH.
\end{align*}
Moreover, let $u_{k}^\inexact = u_{k}^n$. Then, we obtain $m<n$ and
\begin{align*}
 \norm{u_{k}^\inexact-u_{k}^m}\HH
 \le \sum_{j=m}^{n-1} \norm{u_{k}^{j+1}-u_{k}^j}\HH
 \reff{eq:important}\le \frac{\alpha}{L}\,\Crel^\inexact\Big(\sum_{j=m}^{n-1}q^j\Big)\,\eta_{k-1}(u_{k-1}^\inexact)
 \le \frac{\alpha}{L}\,\Crel^\inexact\,\frac{q^m}{1-q}\,\eta_{k-1}(u_{k-1}^\inexact).
\end{align*}
Combining these estimates, we derive that
\begin{align*}
 \norm{u^\exact-\widetilde u_\ell}\HH
 \lesssim \eta_{k}(u_{k}^\inexact) + \eta_{k-1}(u_{k-1}^\inexact).
\end{align*}
For $k>0$, Lemma~\ref{lemma:eta:picard} proves that
\begin{align*}
 \widetilde\eta_\ell(\widetilde u_\ell)
 = \eta_{k}(u_{k}^m)
 \reff{eq:eta:picard}\lesssim \eta_{k-1}(u_{k-1}^\inexact). 
\end{align*}
For $k> 0$ and $m=0$, we even have equality 
\begin{align*}
\widetilde\eta_\ell(\widetilde u_\ell)
=\eta_{k-1}(u_{k-1}^\inexact).
\end{align*}
Altogether, Theorem~\ref{theorem:linear_convergence} proves convergence $\norm{u^\exact-\widetilde u_\ell}\HH+\widetilde\eta_\ell(\widetilde u_\ell)\to0$ as $\ell\to\infty$.

{\bf Step~4.}\quad According to step~2, we may suppose $\ell_k<\infty$ for all $k\in\N_0$. 
To show the implication ``$\Longleftarrow$'' of~\eqref{eq':theorem_optimal_rate}, note that the right-hand side of~\eqref{eq':theorem_optimal_rate} implies $\eta_{k}(u_k^\inexact) \lesssim \big( \#\TT_{k} - \# \TT_0 +1  \big)^{-s}$ for all $k\in\N_0$. Hence, the claim follows from Theorem~\ref{theorem:optimal_rate}.

{\bf Step~5.}\quad According to step~2, we may suppose $\ell_k<\infty$ for all $k\in\N_0$.
To show the implication ``$\Longrightarrow$'' of~\eqref{eq':theorem_optimal_rate}, let $\ell = \ell_{k-1}+m < \ell_{k}<\infty$ for some $k\in\N$ and $m\in\N_0$. Then, $\widetilde u_\ell = \widetilde u_{\ell_{k-1}+m} = u_{k}^m$ and $\widetilde\TT_\ell=\TT_{k}$. Moreover, elementary calculation (see, e.g.,~\cite[Lemma~22]{helmholtz}) proves that
\begin{align*}
 (\# \TT_0)^{-1} \# \widetilde \TT_\ell \leq \#\widetilde\TT_\ell - \#\TT_0 +1 \leq  \#\widetilde\TT_\ell
 \quad\text{for all }\ell\in\N_0.
\end{align*}
This and $\#\TT_{k-1} \simeq \#\TT_{k}$ (which follows from~\eqref{axiom:split} and $\TT_k = \refine(\TT_{{k-1}},\MM_{{k-1}})$) proves 
\begin{align*}
 \eta_{k-1}(u_{k-1}^\inexact)
 \reff{eq:theorem_optimal_rate}\lesssim (\#\TT_{k-1} - \#\TT_0 +1)^{-s}
 &\simeq (\#\TT_k - \#\TT_0 +1)^{-s}
 = (\#\widetilde\TT_{\ell} - \#\TT_0 +1)^{-s}.
\end{align*}
For $m>0$, estimate~\eqref{eq:eta:picard} proves
$\widetilde\eta_\ell(\widetilde u_\ell) =\eta_{k}(u_{k}^m)
\lesssim \eta_{k-1}(u_{k-1}^\inexact)$.
 For $m=0$, it holds $\widetilde\eta_\ell(\widetilde u_\ell)=\eta_{k-1}(u_{k-1}^\inexact)$.
Combining this with the latter estimate, we conclude the proof.

{\bf Step~6.}\quad The same argument as in step~{\rm 5} in combination with Theorem~\ref{theorem:optimal_comp} proves \eqref{eq':theorem_optimal_complexity}.

\end{proof}

\section{Numerical Experiments}
\label{section:numerics}

In this section, we present two numerical experiments in 2D to underpin our theoretical findings. 
In the experiments, we compare the performance of Algorithm~\ref{algorithm} for 
\begin{itemize}
	\item different values of $\lambda \in \{1, 0.1, 0.01,\ldots,10^{-6}\}$,
	\item different values of $\theta \in \{0.2, 0.4,\ldots,1\}$,
	\item nested iteration $u_\ell^0 := u_{\ell-1}^\inexact$ compared to a naive initial guess $u_\ell^0 := 0$.
\end{itemize}
As model problems serve nonlinear boundary value problems similar to those of~\cite{gmz11,gmz,multiscale,cw2015}.

\subsection{Model problem}
Let $\Omega \subset \R^d$   be a bounded Lipschitz domain with polyhedral boundary $\Gamma = \partial \Omega$, $d \in \{2,3\}$. 
Suppose that $\Gamma := \overline \Gamma_D \cup \overline \Gamma_N$ is split into relatively open and disjoint 
Dirichlet and Neumann boundaries $\Gamma_D,\Gamma_N \subseteq \Gamma$ with $|\Gamma_D| >0$.
For given $f \in L^2(\Omega)$, we consider problems of the following type:
\begin{eqnarray}\label{eq:example}
\begin{split}
- \diver( \GG(x,|\nabla u^\exact(x)|^2) \nabla u^\exact(x) ) &= f(x) \quad &\textrm{in }& \Omega, \\
u^\exact(x) &= 0  &\textrm{on }& \Gamma_D, \\
\mu(x,|\nabla u^\exact(x)|^2) \partial_{\bf n} u^\exact(x) &= g(x)  &\textrm{on }& \Gamma_N.
\end{split}
\end{eqnarray}
We suppose that the scalar nonlinearity $\GG: \Omega \times \R_{\geq 0} \rightarrow \R$ satisfies the following properties~\eqref{item:G_bounded}--\eqref{item:G_lipschitz1}, similarly considered in~\cite{gmz}.

\begin{enumerate}
 \renewcommand{\theenumi}{M\arabic{enumi}}
 \bf
 \item \label{item:G_bounded}
 \rm
 There exist constants $0<\gamma_1<\gamma_2<\infty$ such that
 \begin{align}\label{eq:G_bounded}
 \gamma_1 \le \GG(x,t) \le \gamma_2 \quad \textrm{for all } x \in \Omega \text{ and all } t \geq 0.
 \end{align}
 \bf
 \item \label{item:G_differentiable}
 \rm
There holds $\GG(x,\cdot) \in C^1(\R_{\geq 0} ,\R)$ for all $x \in \Omega$, and there exist $0 < \widetilde{\gamma}_1<\widetilde{\gamma}_2<\infty$ such that
 \begin{align}\label{eq:G_differentiable}
 \widetilde{\gamma}_1 \le \GG(x,t) +2 t  \frac{\d{}}{\d{t}} \GG(x,t) \le \widetilde{\gamma}_2 \quad \textrm{for all } x \in \Omega \text{ and all } t \geq 0.
 \end{align}
  \bf
 \item \label{item:G_lipschitz}
 \rm
 Lipschitz-continuity of $\GG(x,t)$ in $x$, i.e., there exists $L_{\GG}>0$ such that
 \begin{align}\label{eq:G_lipschitz}
 | \GG(x,t) - \GG(y,t) | \le L_{\GG} | x - y | \quad \textrm{for all } x,y \in \Omega \text{ and all } t \geq 0.
 \end{align}
 \bf
 \item \label{item:G_lipschitz1}
 \rm
Lipschitz-continuity of $t \frac{\d{}}{\d{t}} \GG(x,t)$ in $x$, i.e., there exists $\widetilde{L}_{\GG}>0$ such that
 \begin{align}\label{eq:G_lipschitz1}
 | t \frac{\d{}}{\d{t}} \GG(x,t) - t \frac{\d{}}{\d{t}} \GG(y,t) | \le \widetilde{L}_{\GG} | x - y | \quad \textrm{for all } x,y \in \Omega
 \text{ and all } t \geq 0.
 \end{align}
\end{enumerate}

\subsection{Weak formulation}
The weak formulation of~\eqref{eq:example} reads as follows : Find $u \in H_D^{1}(\Omega):= \{w \in H^1: \, w=0 \text{ on } \Gamma_D \, \text{in the sense of traces}  \}$ such that
\begin{align}\label{eq:weak_example}
\IntSet{\Omega} \GG(x,|\nabla  u^\exact(x)|^2) \,\nabla u^\exact\cdot \nabla v \d{x} = \IntSet{\Omega} f v \d{x} 
+ \IntSet{\Gamma_N} g v \d{s} \quad \textrm{for all } v \in 
H^1_D(\Omega)
\end{align}
With respect to the abstract framework, it holds $\HH := H^{1}_D(\Omega)$ 
with $\| v \|_{\HH} := \norm{\nabla v}{L^2(\Omega)}$.
With $\dual{\cdot}{\cdot}$ being the extended $L^2(\Omega)$ scalar product,
we obtain~\eqref{eq':strongform} with operators 
\begin{subequations}
\begin{align}
\dual{Aw}{v} &:= \IntSet{\Omega} \GG(x,|\nabla w(x)|^2) \, \nabla w(x)\cdot\nabla v(x) \d{x},  \label{eq:AG} \\
F(v) &:= \IntSet{\Omega} f v \d{x} +  \IntSet{\Omega} g v \d{s}. \label{eq:F}
\end{align}
\end{subequations}

Next, we show that $A$ satisfies \eqref{axiom:monotone}--\eqref{axiom:potential}. For the sake of completeness, we first recall an auxiliary lemma, which is just a simplified version of \cite[Lemma 2.1]{lb96}.

\begin{lemma}\label{lemma:kappafun}
Let $C_1>0$ as well as $0 < C_2 \le C_3 < \infty$ and $\kkappa(x,\cdot) \in C^1(\R_{\ge0},\R_{\ge0})$ with $\kappa(x,t) \leq C_1$ for all $x \in \Omega$ and $t \geq 0$ satisfy
\begin{align}
C_2 \le  \frac{\d{}}{\d{t}} (t \kkappa(x,t) ) \le \frac{\d{}}{\d{t}} (t \kkappa(x,t) )  \le C_3
\quad \textrm{for all } x \in \Omega \text{ and all } t \geq 0.
\end{align}
Then, it holds that
\begin{align}\label{eq:coercive}
\big(\kkappa(x,|y|) y - \kkappa(x,|z|) z \big) \cdot \big( y- z  \big) \ge C_1 | y - z|^2
\quad \textrm{for all } x \in \Omega \text{ and } y,z \in \R^d,
\end{align}
as well as 
\begin{align}\label{eq:continuous}
\big| \kkappa(x,|y|) y - \kkappa(x,|z|) z \big| \le C_2 | y - z |
\quad \textrm{for all } x \in \Omega \text{ and } y,z\in \R^d.
\qquad\qed
\end{align}
\end{lemma}

\begin{proposition}\label{prop:operatorproperties}
Suppose that $\GG: \Omega \times \R_{\geq 0} \rightarrow \R$ satisfies \eqref{item:G_bounded}--\eqref{item:G_differentiable}. Then, the corresponding operator $A$ satisfies \eqref{axiom:monotone}--\eqref{axiom:potential} with constants $\alpha := \widetilde{\gamma}_1$ and 
$L :=\widetilde{\gamma}_2$. 
\end{proposition}

\begin{proof}
 We prove the assertion in two steps.

{\bf Step~1.}\quad To show~\eqref{axiom:monotone}--\eqref{axiom:lipschitz}, let $\kkappa(x,t) := \GG(x,t^2)$.
Assumptions~\eqref{item:G_bounded}--\eqref{item:G_differentiable} and $\frac{\d{}}{\d{t}}(t \kkappa(x,t)) = \GG(x,t^2) + 2t^2 \partial_2 \GG(x,t^2)$ allow to apply Lemma~\ref{lemma:kappafun}. 
For all $v,w \in H^1_D(\Omega)$, we obtain 
\begin{align}\label{eq:example_prop_step}
\begin{split}
\alpha | \nabla v - \nabla w |^2 
\stackrel{\eqref{eq:coercive}}{\le}& \big( \GG(\cdot,|\nabla v|^2) \nabla v - \GG(\cdot,|\nabla w|^2 \big) \nabla w) \cdot \big(\nabla v - \nabla w \big)\quad\text{a.e. in }\Omega,
\end{split} 
\end{align}
as well as
\begin{align}\label{eq:example_prop_step2}
\big| \GG( \cdot ,|\nabla v|^2) \nabla v - \GG(\cdot,|\nabla w|^2) \nabla w) \big|^2 \stackrel{\eqref{eq:continuous}}{\le}
 L^2 \big| \nabla v - \nabla w \big|^2\quad\text{a.e. in }\Omega.
\end{align}
Integration over $\Omega$ proves strong monotonicity~\eqref{axiom:monotone} and Lipschitz continuity~\eqref{axiom:lipschitz}.

{\bf Step~2.}\quad We next show~\eqref{axiom:potential}.
Analogously to \cite{Hasanov2010}, we define
\begin{align}\label{eq:defpot}
P: H^{1}_D(\Omega) \rightarrow \R_{\geq 0} : \quad w \mapsto \frac12 \IntSet{\Omega} \Int{0}{|\nabla w|^2} \GG(x,\zzeta) \d{\zzeta} \d{x}.
\end{align}
Note that boundedness~\eqref{item:G_bounded} implies well posedness of $P$. Next, we show that $A$ is the Gateaux-derivative $dP$ of $P$. To that end, let $r > 0$ and $v,w \in H^{1}_D(\Omega)$. Define
\begin{align*}
H(r) := P(w + rv) \stackrel{\eqref{eq:defpot}}{=} \frac12 \IntSet{\Omega} \Int{0}{|\nabla w + r\nabla v|^2} \GG(x,\zzeta) \d{\zzeta} \d{x}.
\end{align*}
With the Leibniz rule, we get
\begin{align*}
H^{\prime}(r) &= \frac12 \IntSet{\Omega} \GG(x,| \nabla w+r \nabla v|^2) \frac{\d{}}{\d r} \big( |\nabla w + r \nabla v|^2 \big) \d{x} \\
&= \IntSet{\Omega} \GG(x,| \nabla w+r \nabla v|^2)\, (\nabla w + r \nabla v)\cdot \nabla v \d{x}.
\end{align*}
This concludes $\dual{dP(w)}{v} = H^{\prime}(0) = \IntSet{\Omega} \GG(x,|\nabla w|^2) \, \nabla w \cdot \nabla v \d{x} 
\stackrel{\eqref{eq:AG}}{=} \dual{Aw}{v}$.
\end{proof}

\subsection{Discretization and \textsl{a~posteriori} error estimator} \label{example:estimator}
Let $\TT_\plus$ be a conforming triangulation of $\Omega$. For $T \in \TT_\plus$, define $h_T := |T|^{1/d} \simeq \diam(T)$.
Consider
\begin{align}\label{eq:dpr:XX}
\XX_{\plus} := \{ v: \Omega \to \R: \, v|_{T} \in \mathcal{P}^1(T) \textrm{ for all } T \in \TT_{\plus} \} \,\cap \, H^1_D(\Omega).
\end{align}
For ease of notation, set $\GG_v(x):= \GG(x,|\nabla v(x)|^2)$. 
Let $[\, \cdot \,] \big|_{\partial T \cap \Omega}$ denote the jump of discrete functions across the element interfaces.
As in \cite[Section 3.2]{gmz}, we define for all  $T \in \TT_{\plus}$ and all $v_\plus \in \XX_{\plus}$, the corresponding residual refinement indicators
\begin{align}\label{eq:dpr:eta} 
\begin{split}\eta_{\plus}(T,v_{\plus})^2 &:= h_T^2 \|f+ \diver( \GG_{v_{\plus}} \nabla v_{\plus} ) \|_{L^2(T)}^2 + 
h_T \| [ \GG_{v_{\plus}} \partial_{\rm n} v_{\plus} ] \|_{L^2(\partial T\cap\Omega)}^2 
\\&\quad
+ h_T \norm{g - \GG_{v_\plus} \partial_{\rm n} v_\plus}{L^2(\partial T \cap \Gamma_N)}^2.
\end{split}
\end{align}

The well-posedness of the error estimator requires that the nonlinearity $\mu(x,t)$ is Lipschitz continuous in $x$, i.e.~\eqref{item:G_lipschitz}. Then, reliability~\eqref{axiom:reliability} and discrete reliability~\eqref{axiom:discrete_reliability} are proved as in the linear case; see, e.g.,~\cite{ckns} for the linear case or \cite[Theorem~3.3]{gmz} and \cite[Theorem 3.4]{gmz}, respectively, for strongly monotone nonlinearities.

The verification of stability~\eqref{axiom:stability} and reduction~\eqref{axiom:reduction} requires the validity of a certain inverse estimate. For scalar nonlinearities and under the assumptions~\eqref{item:G_bounded}--\eqref{item:G_lipschitz1}, the latter is proved in~{\cite[Lemma 3.7]{gmz}}. Using this inverse estimate, the proof of~\eqref{axiom:stability} and~\eqref{axiom:reduction} 
follows as for the linear case; see, e.g.,~\cite{ckns} for the linear case or \cite[Section~3.3]{gmz} for scalar nonlinearities. We note that the necessary inverse estimate is, in particular, open for non-scalar nonlinearities. In any case, the arising constants in~\eqref{axiom:stability}--\eqref{axiom:discrete_reliability} depend also on the uniform shape regularity of the triangulations generated by newest vertex bisection.

\begin{figure}
  \centering
  \includegraphics[width=0.5\textwidth]{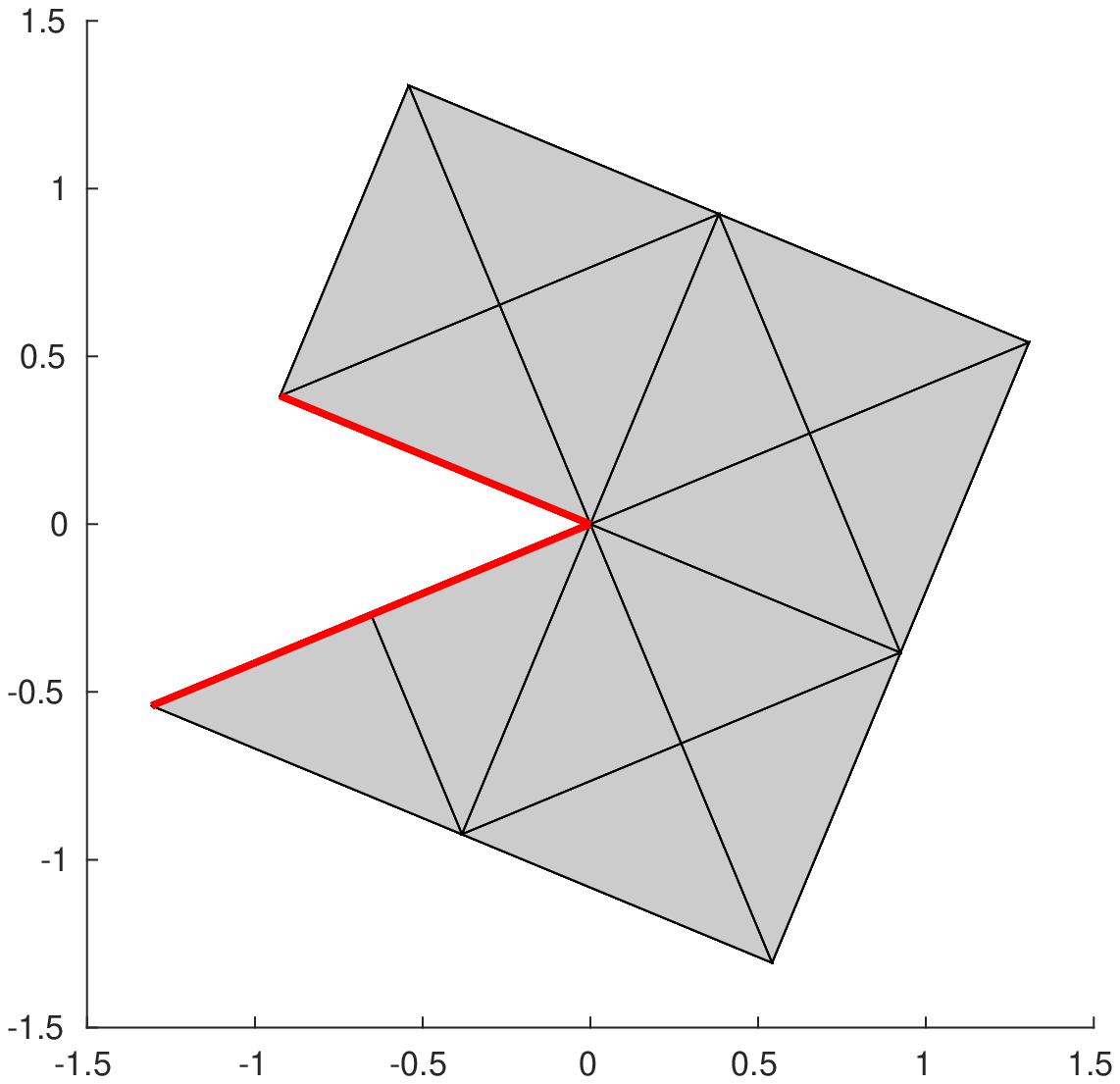}%
  \includegraphics[width=0.5\textwidth]{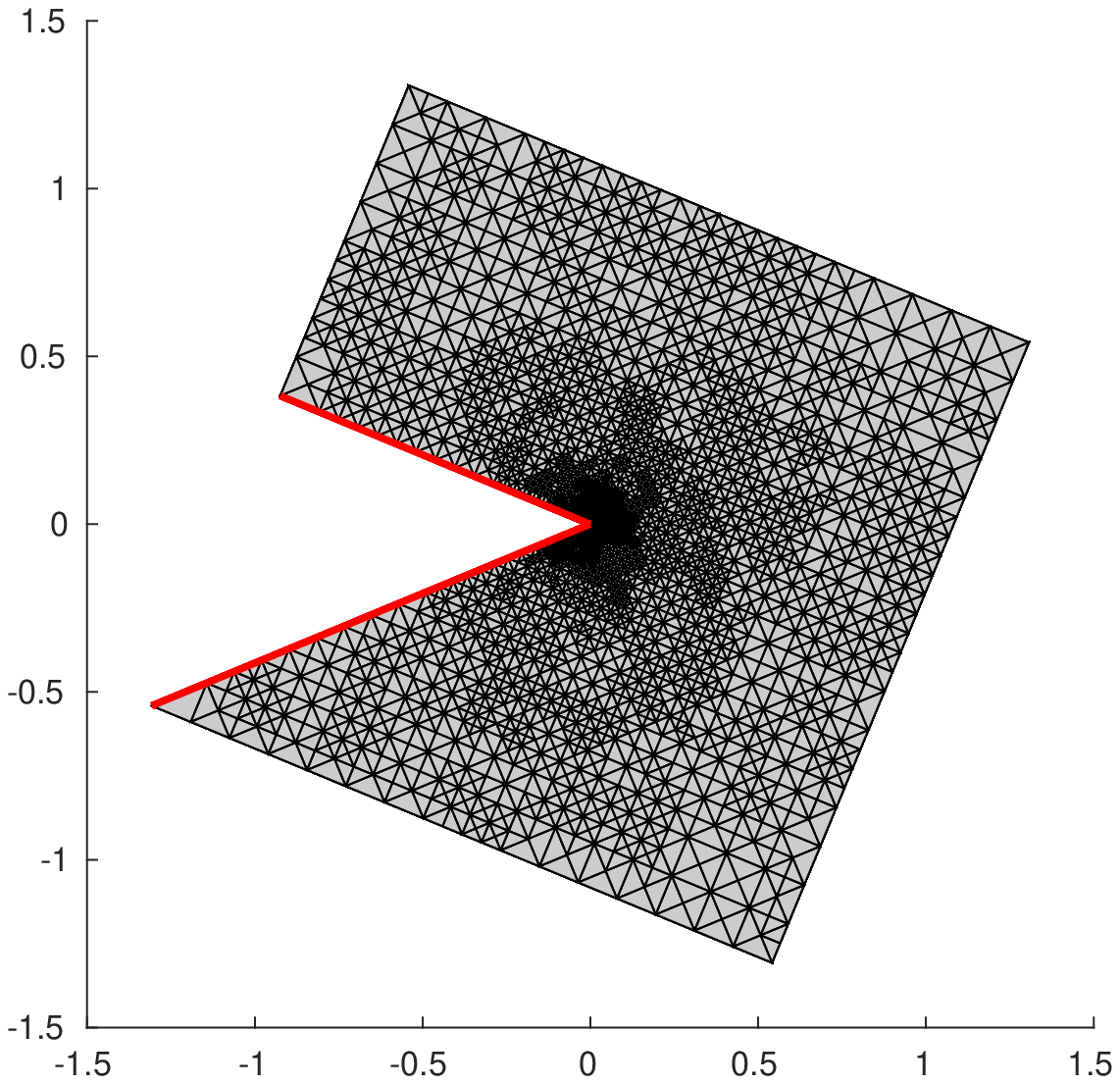}
  \caption{Experiment with known solution from Section~\ref{section:example1}: $Z$-shaped domain $\Omega \subset \R^2$ and the initial mesh $\TT_0$ (left) and with NVB adaptively generated mesh $\TT_{19}$ with $5854$ elements (right). $\Gamma_D$ is visualized in red. 
   \vspace*{2mm}}
  \label{fig:ex1:mesh}
\end{figure}
\begin{figure}
	\centering
	\psfrag{nE}[c][c]{\tiny number of elements $N$}
	\psfrag{estimator}[c][c]{\tiny error resp.\ estimator}
	\psfrag{theta =0.2}{\tiny $\theta = 0.2$}
	\psfrag{theta =0.4}{\tiny $\theta = 0.4$}
	\psfrag{theta =0.6}{\tiny $\theta = 0.6$}
	\psfrag{theta =0.8}{\tiny $\theta = 0.8$}
	\psfrag{theta =1}{\tiny $\theta = 1.0$ (uniform)}	
	\psfrag{O12}[c][r]{\tiny $\OO(N^{-1/2})$}
	\psfrag{Oalpha}[c][c]{\tiny $\OO(N^{-\beta/2})$}
	\includegraphics[width=0.48\textwidth]{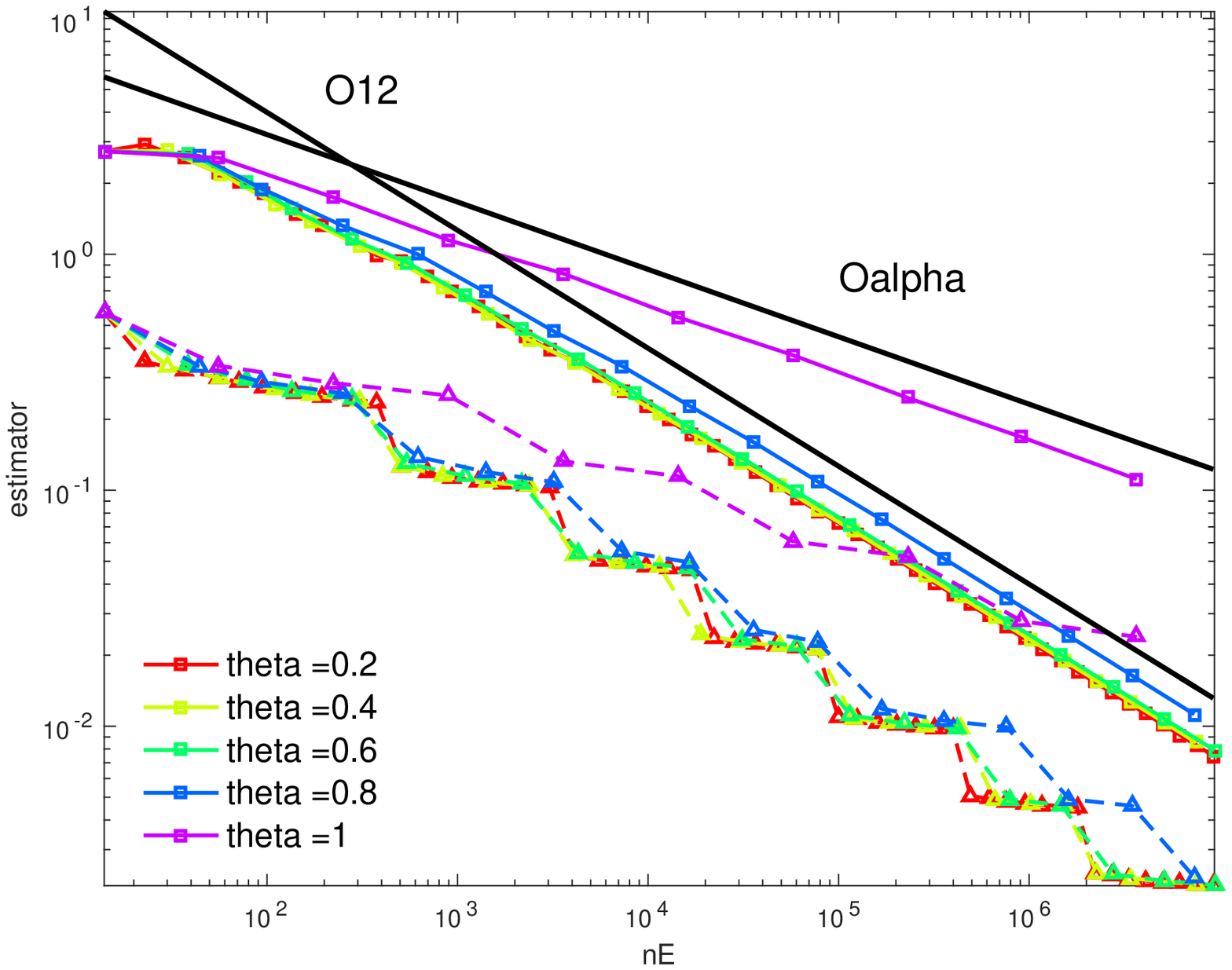}
	\hfill
	\includegraphics[width=0.48\textwidth]{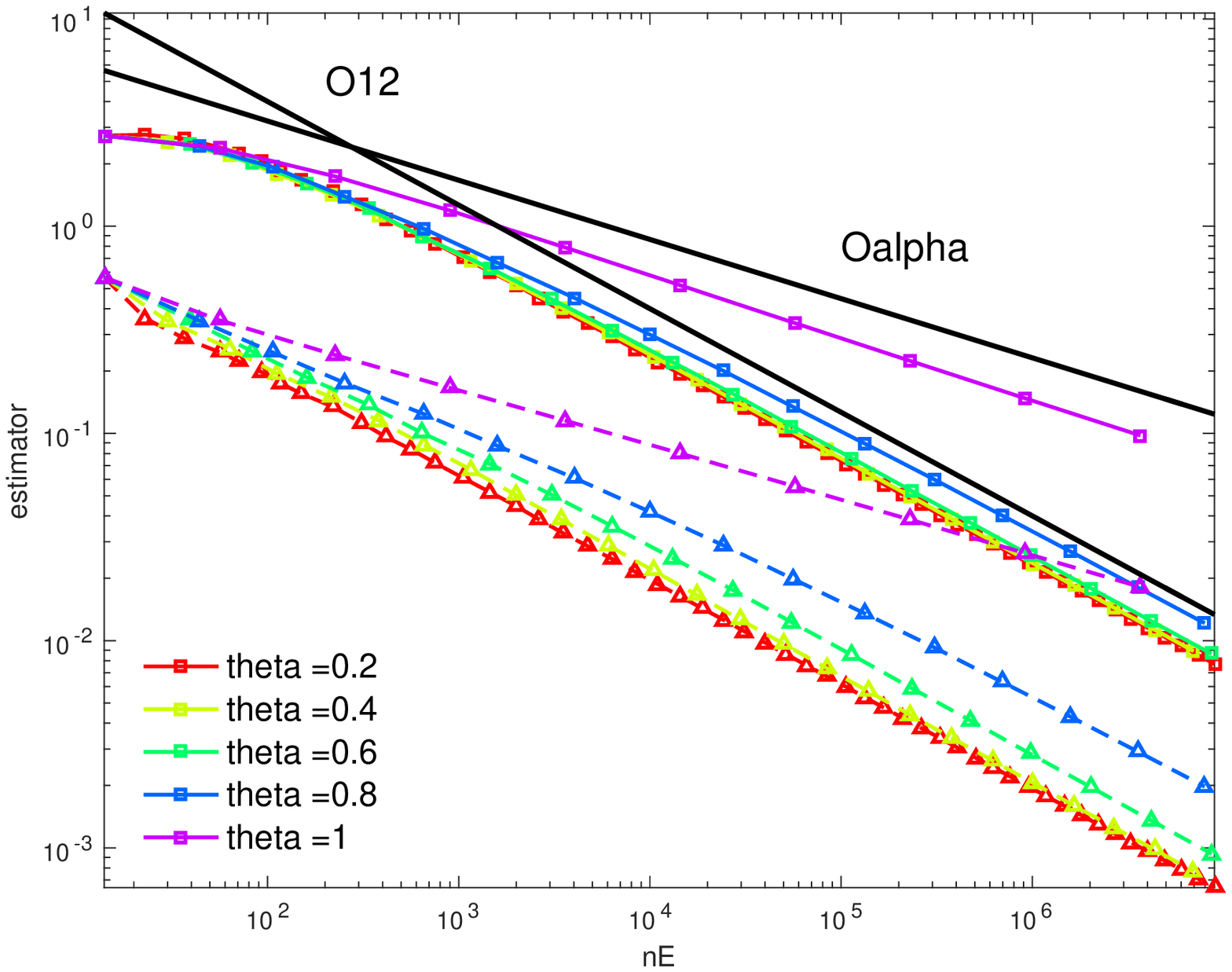}		
	\caption{Experiment with known solution from Section~\ref{section:example1}: Convergence of $\eta_\ell(u_\ell^\inexact)$ (solid lines) and $\norm{\nabla u^\exact - \nabla u^\inexact_\ell}{L^2(\Omega)}$ (dashed lines) 
		for $\lambda = 0.1$ and different values of $\theta\in\{0.2,\dots, 1\}$, where we compare naive initial guesses $u_\ell^0 := 0$ (left) and nested iteration $u_\ell^0 := u_{\ell-1}^\inexact$ (right).\vspace*{1mm}}
	\label{fig:ex1:compare_theta_1}	  
\end{figure}
\begin{figure}
	\centering
	\psfrag{nE}[c][c]{\tiny number of elements $N$}
	\psfrag{estimator}[c][c]{\tiny error resp.\ estimator}
	\psfrag{theta =0.2}{\tiny $\theta = 0.2$}
	\psfrag{theta =0.4}{\tiny $\theta = 0.4$}
	\psfrag{theta =0.6}{\tiny $\theta = 0.6$}
	\psfrag{theta =0.8}{\tiny $\theta = 0.8$}
	\psfrag{theta =1}{\tiny $\theta = 1.0$ (uniform)}	
	\psfrag{O12}[c][r]{\tiny $\OO(N^{-1/2})$}
	\psfrag{Oalpha}[c][c]{\tiny $\OO(N^{-\beta/2})$}		
	\includegraphics[width=0.48\textwidth]{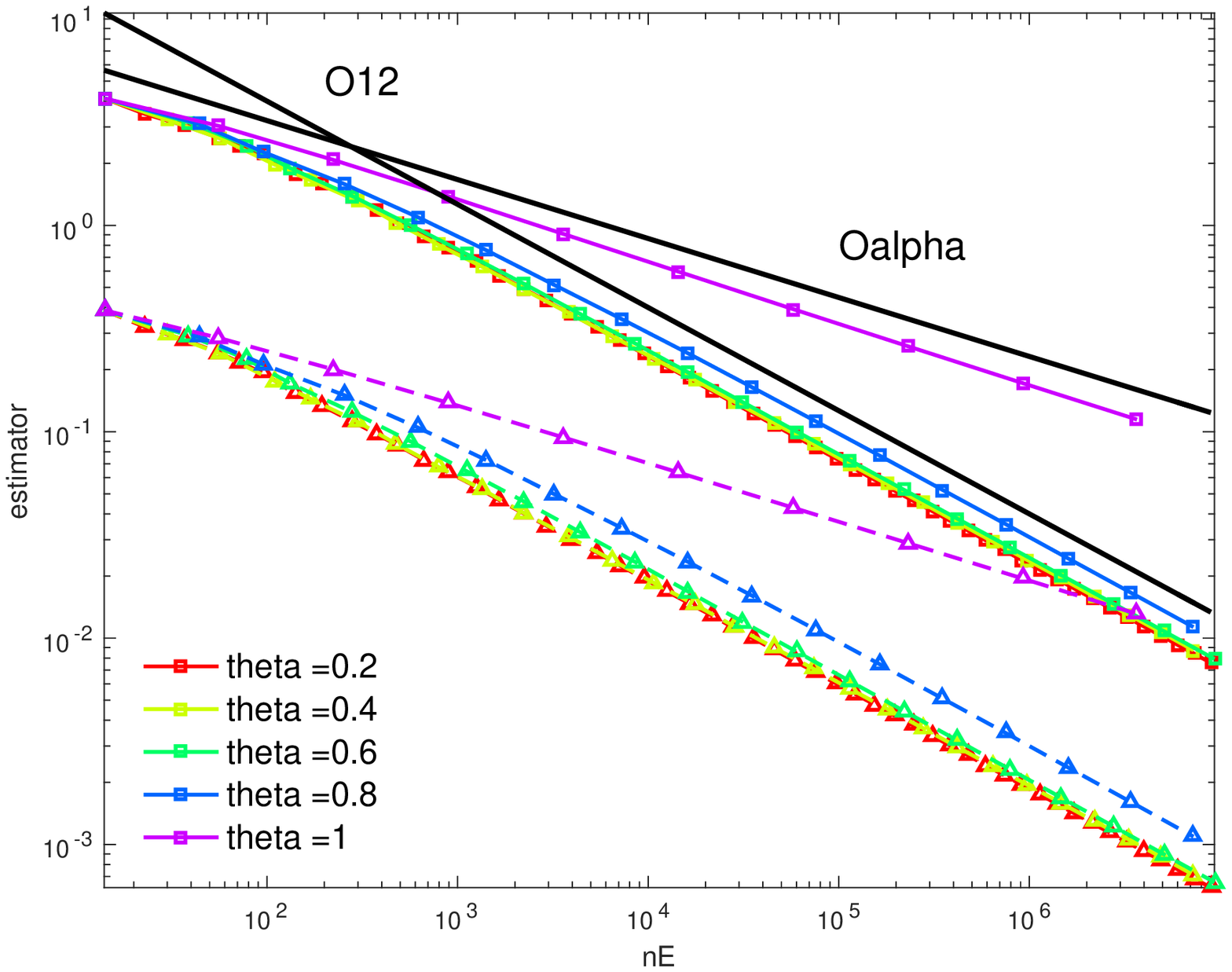}
	\hfill
	\includegraphics[width=0.48\textwidth]{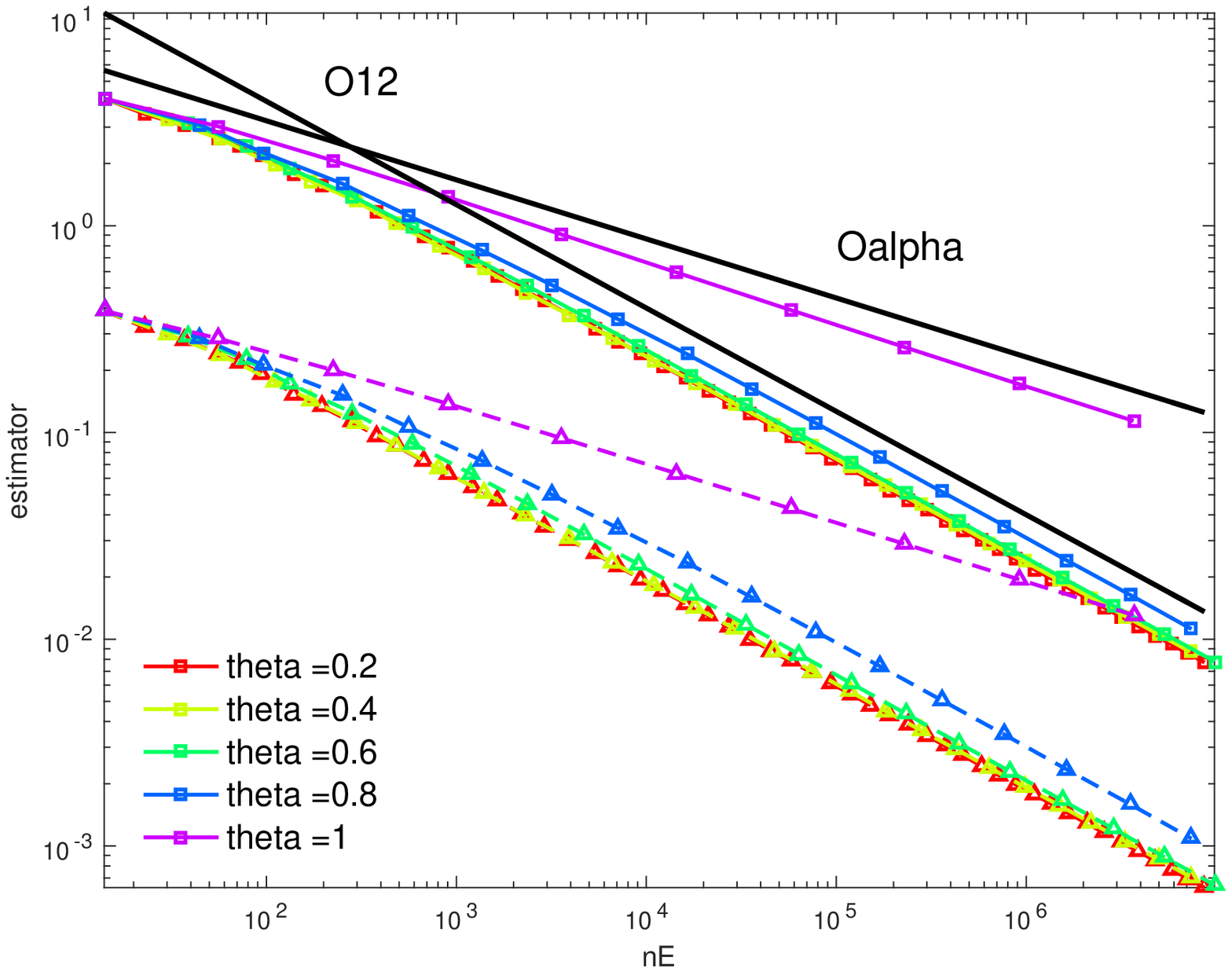}	
	\caption{Experiment with known solution from Section~\ref{section:example1}: Convergence of $\eta_\ell(u_\ell^\inexact)$ (solid lines) and $\norm{\nabla u^\exact - \nabla u^\inexact_\ell}{L^2(\Omega)}$ (dashed lines)
		for $\lambda = 10^{-5}$ and different values of $\theta\in\{0.2,\dots, 1 \}$, where we compare naive initial guesses $u_\ell^0 := 0$ (left) and nested iteration $u_\ell^0 := u_{\ell-1}^\inexact$ (right).\vspace*{1mm}}
  \label{fig:ex1:compare_theta_2}
\end{figure}
\begin{figure}
	\centering
	\psfrag{nE}[c][c]{\tiny number of elements $N$}
	\psfrag{estimator}[c][c]{\tiny error resp.\ estimator}
	\psfrag{lambda =1}{\tiny $\lambda = 1$}
	\psfrag{lambda =0.1}{\tiny $\lambda = 10^{-1}$}
	\psfrag{lambda =0.01}{\tiny $\lambda = 10^{-2}$}
	\psfrag{lambda =0.001}{\tiny $\lambda = 10^{-3}$}
	\psfrag{lambda =0.0001}{\tiny $\lambda = 10^{-4}$}
	\psfrag{lambda =1e-05}{\tiny $\lambda = 10^{-5}$}
	\psfrag{lambda =1e-06}{\tiny $\lambda = 10^{-6}$}	
	\psfrag{O12}[c][r]{\tiny $\OO(N^{-1/2})$}
	\includegraphics[width=0.48\textwidth]{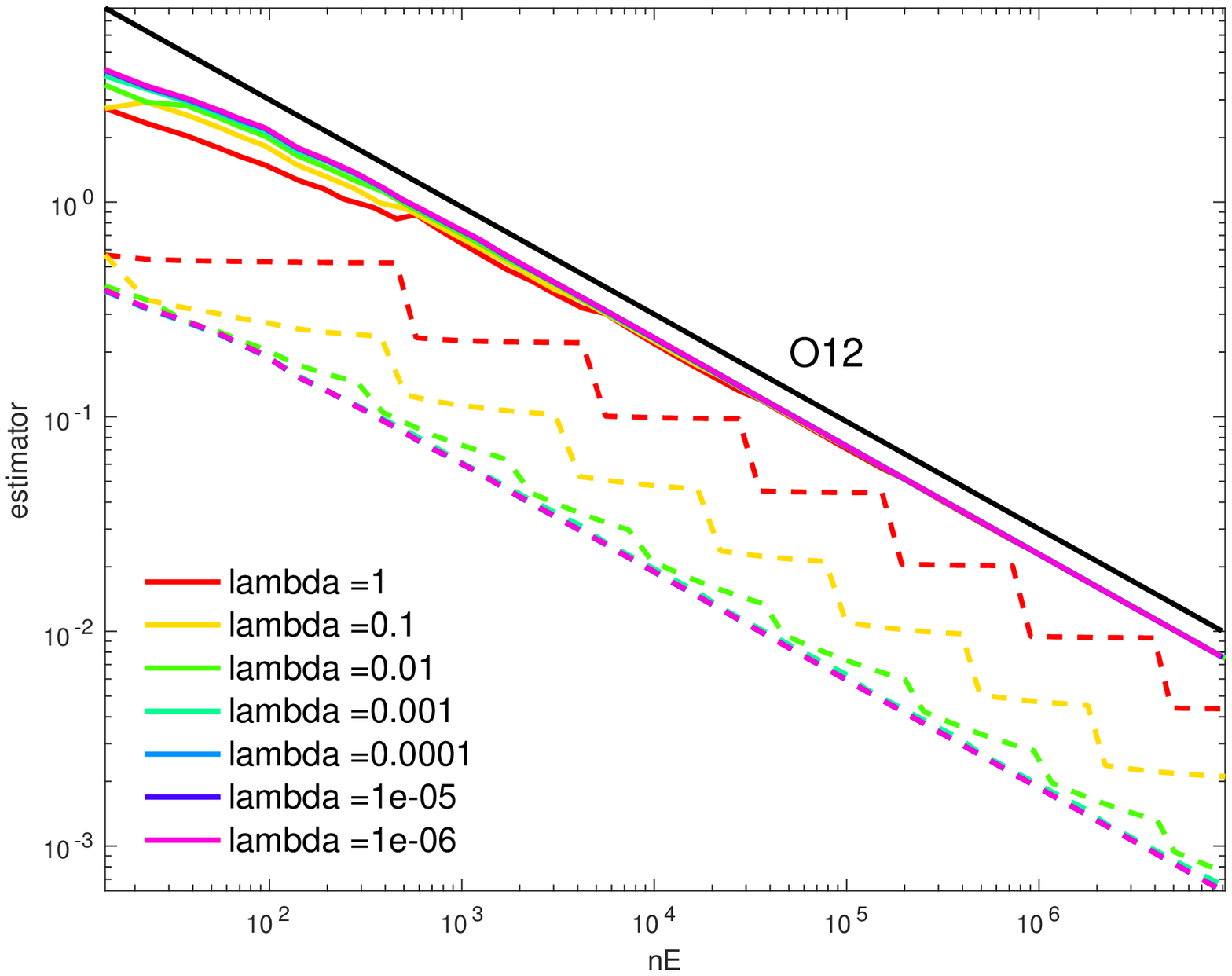}
	\hfill
	\includegraphics[width=0.48\textwidth]{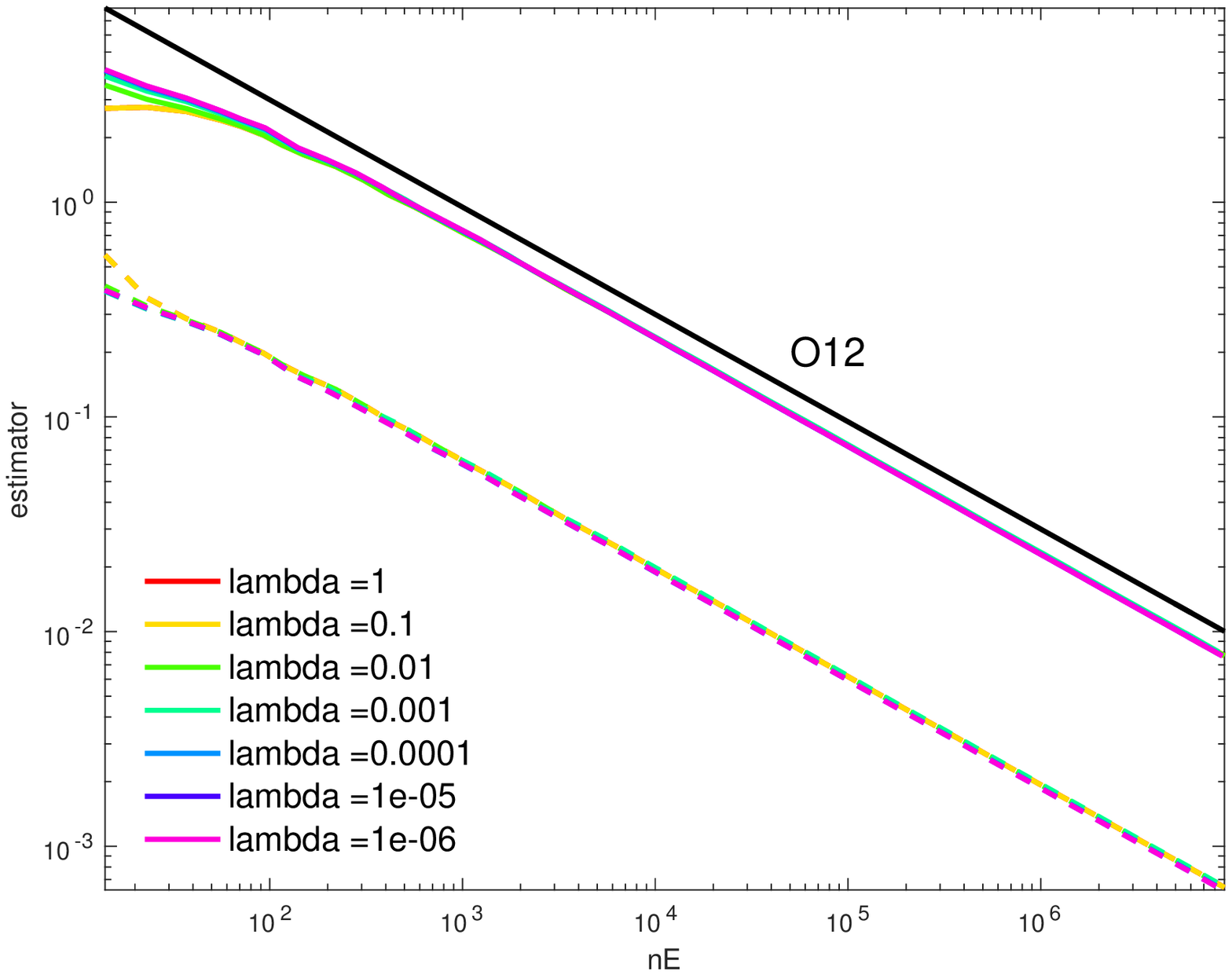}
	\caption{Experiment with known solution from Section~\ref{section:example1}: Convergence of $\eta_\ell(u_\ell^\inexact)$ (solid lines) and $\norm{\nabla u^\exact - \nabla u_\ell^\inexact}{L^2(\Omega)}$ (dashed lines)
		for $\theta = 0.2$ and different values of $\lambda\in\{1,0.1, 0.01,\dots, 10^{-6} \}$, where we compare naive initial guesses $u_\ell^0 := 0$ (left) and nested iteration $u_\ell^0 := u_{\ell-1}^\inexact$ (right).\vspace*{1mm}}
	\label{fig:ex1:compare_lambda_1}
\end{figure}

\begin{figure}
	\centering
	\psfrag{nE}[c][c]{\tiny number of elements $N$}
	\psfrag{FP-steps}[c][c]{\tiny  number of Picard iterations}
	\psfrag{lambda=0.1}{\tiny $\lambda = 10^{-1}$}
	\psfrag{lambda=0.01}{\tiny $\lambda = 10^{-2}$}
	\psfrag{lambda=0.001}{\tiny $\lambda = 10^{-3}$}
	\psfrag{lambda=0.0001}{\tiny $\lambda = 10^{-4}$}
	\psfrag{lambda=1e-05}{\tiny $\lambda = 10^{-5}$}
	\psfrag{lambda=1e-06}{\tiny $\lambda = 10^{-6}$}		
	\psfrag{Olog}[c][c]{\tiny $\OO(\log(N))$}
	\includegraphics[width=0.48\textwidth]{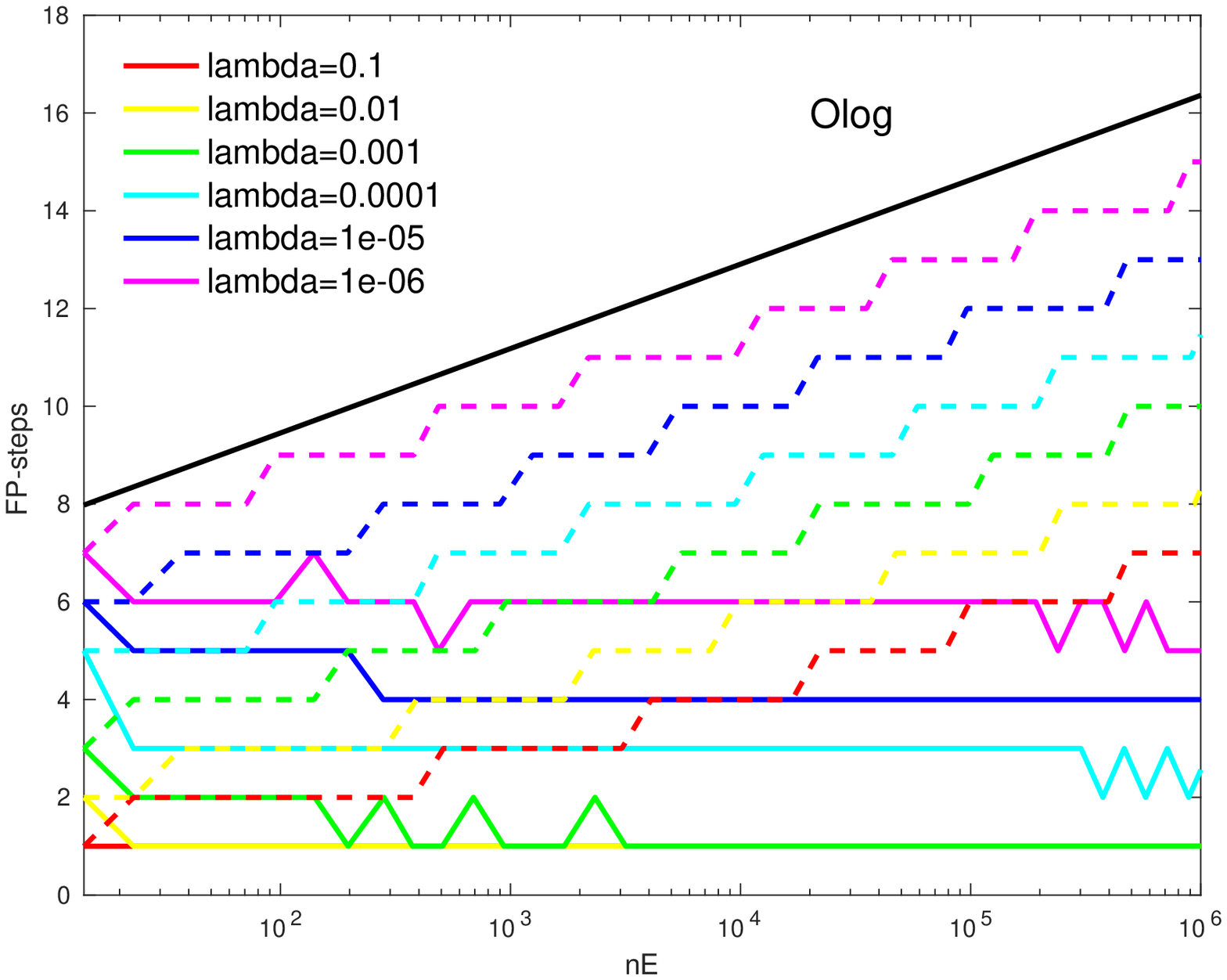}
	\hfill
	\includegraphics[width=0.48\textwidth]{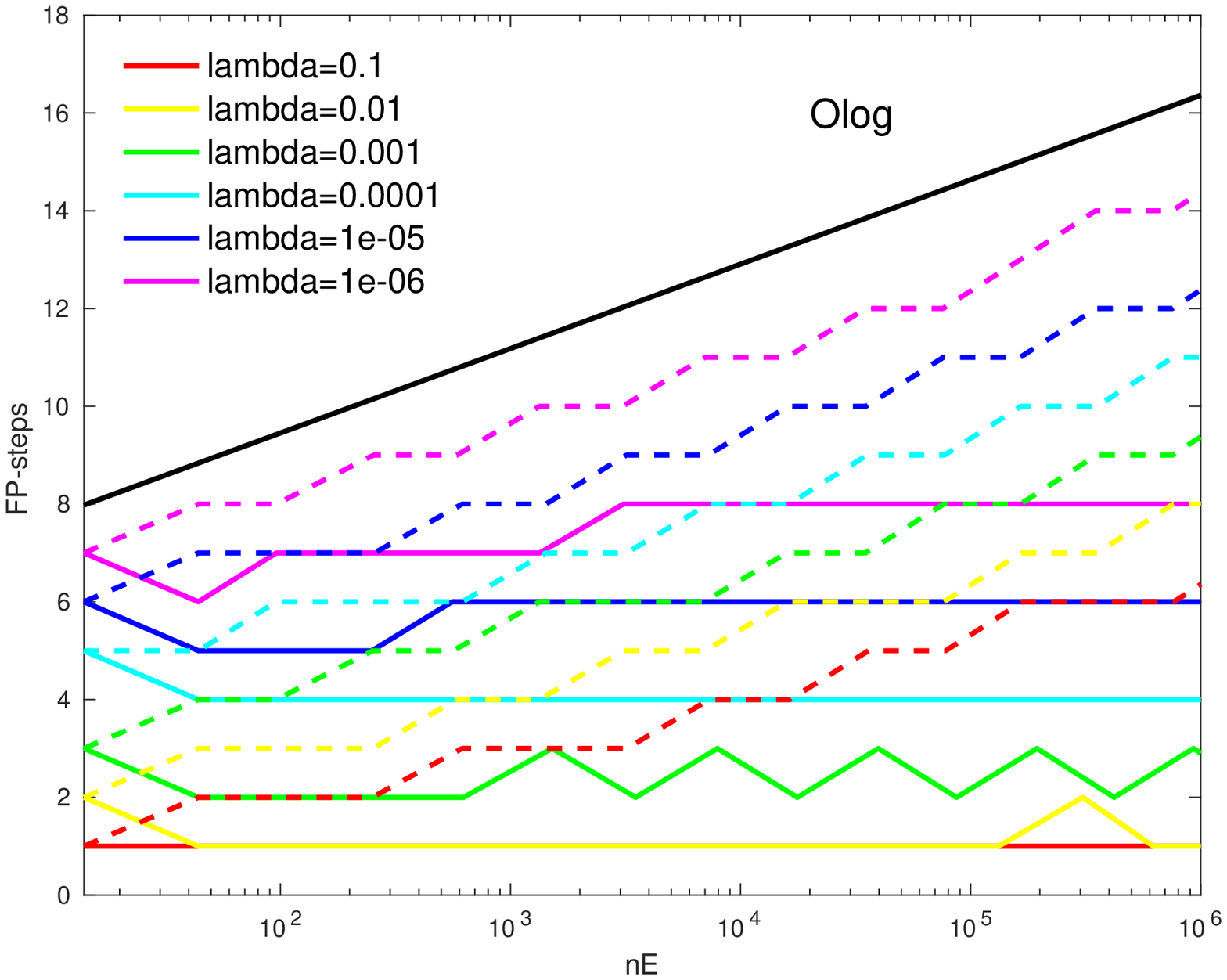}
	\caption{Experiment with known solution from Section~\ref{section:example1}: Number of Picard iterations used in each step of Algorithm~\ref{algorithm} for different values of $\lambda\in\{0.1,\dots, 10^{-6} \}$ and  $\theta = 0.2$ (left) as well as $\theta = 0.8$ (right). As expected, we observe logarithmic growth for naive initial guesses $u_\ell^0 := 0$ (dashed lines); see Remark~\ref{remark1}. On the other hand, nested iteration  $u_\ell^0 := u_{\ell-1}^\inexact$ (solid lines) leads to a bounded number of Picard iterations; see Remark~\ref{remarkXXX}.\vspace*{1mm}}
	\label{fig:ex1:compare_fpsteps}
	\vspace*{2mm}
\end{figure}
\begin{figure}
	\centering
	\psfrag{nE}[c][c]{\tiny number of elements $N$}
	\psfrag{estimator}[c][c]{\tiny   estimator}
	\psfrag{theta =0.2}{\tiny $\theta = 0.2$}
	\psfrag{theta =0.4}{\tiny $\theta = 0.4$}
	\psfrag{theta =0.6}{\tiny $\theta = 0.6$}
	\psfrag{theta =0.8}{\tiny $\theta = 0.8$}
	\psfrag{theta =1}{\tiny $\theta = 1.0$ (uniform)}	
	\psfrag{O12}[c][c]{\tiny $\OO(N^{-1/2})$}
	\psfrag{Oalpha}[l][l]{\tiny $\OO(N^{-\beta/2})$}
	\includegraphics[width=0.48\textwidth]{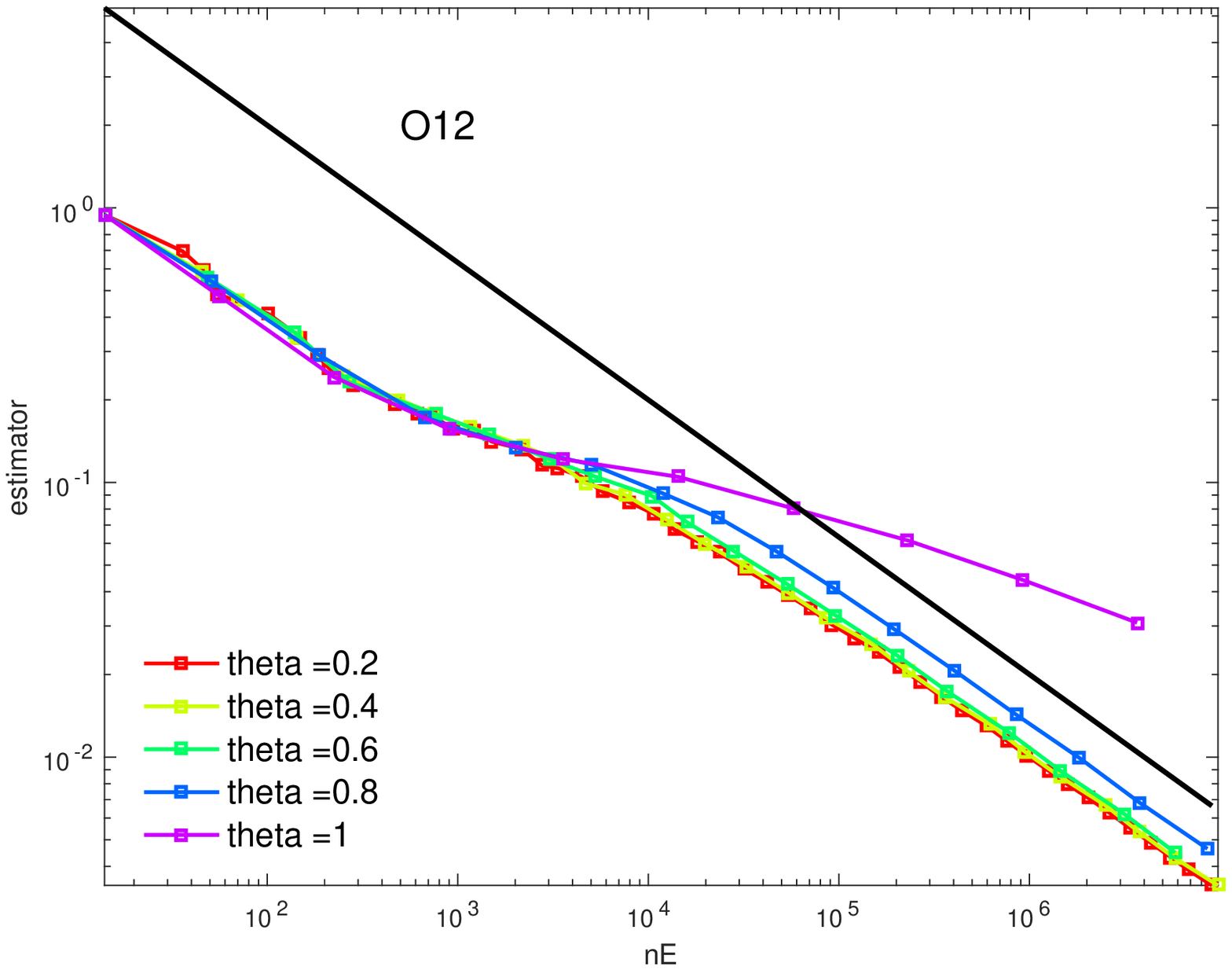}
	\hfill
	\includegraphics[width=0.48\textwidth]{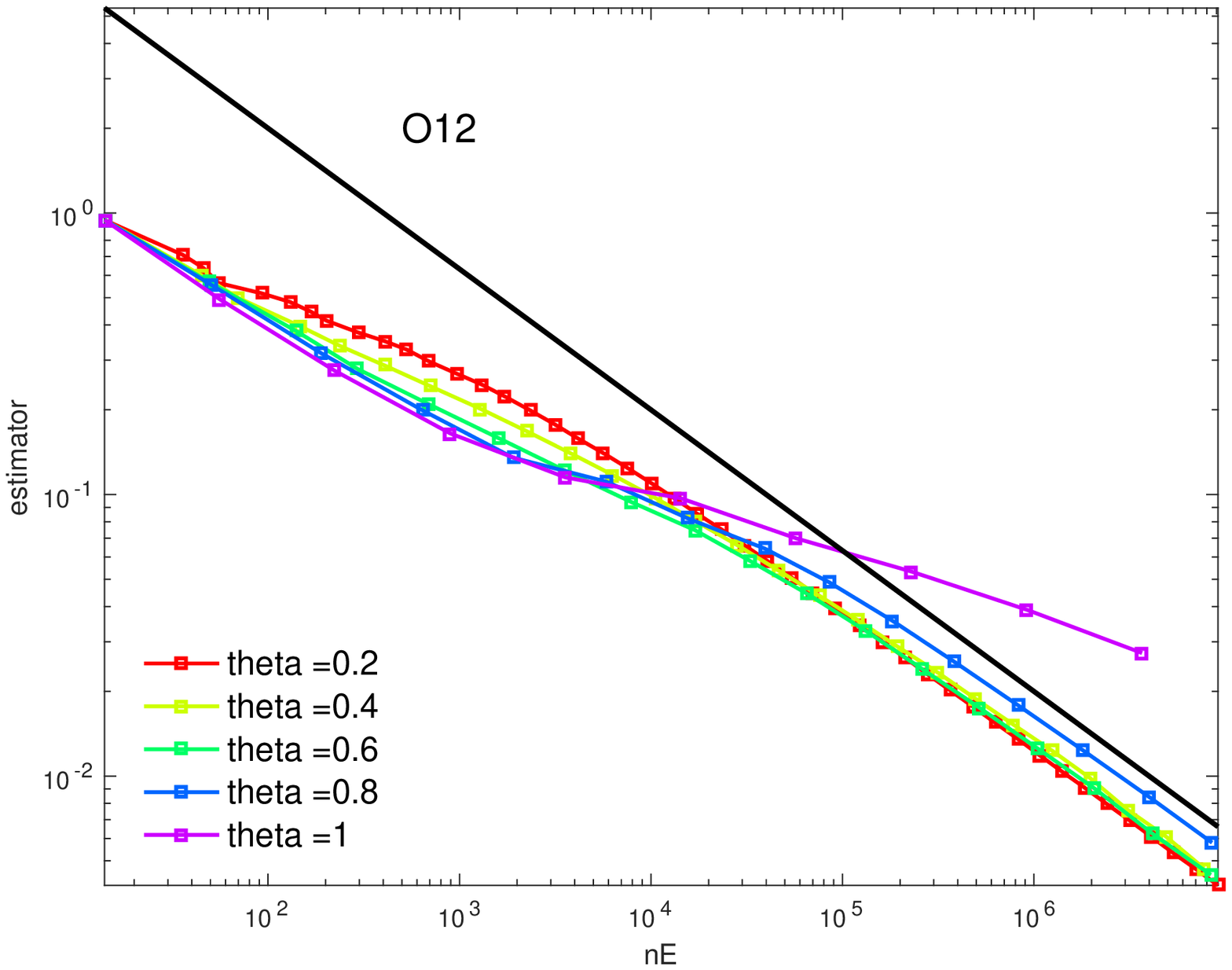}
	\caption{Experiment with unknown solution from Section~\ref{section:example2}: Convergence of $\eta_\ell(u_\ell^\inexact)$ for $\lambda = 0.1$ and $\theta\in\{0.2,\dots, 1\}$, where we compare naive initial guesses $u_\ell^0 := 0$ (left) and nested iteration $u_\ell^0 := u_{\ell-1}^\inexact$ (right).\vspace*{2mm}}
	\label{fig:ex2:compare_theta_1}	  
\end{figure}
\begin{figure}
	\centering
	\psfrag{nE}[c][c]{\tiny number of elements $N$}
	\psfrag{estimator}[c][c]{\tiny   estimator}
	\psfrag{theta =0.2}{\tiny $\theta = 0.2$}
	\psfrag{theta =0.4}{\tiny $\theta = 0.4$}
	\psfrag{theta =0.6}{\tiny $\theta = 0.6$}
	\psfrag{theta =0.8}{\tiny $\theta = 0.8$}
	\psfrag{theta =1}{\tiny $\theta = 1.0$ (uniform)}	
	\psfrag{O12}[c][c]{\tiny $\OO(N^{-1/2})$}
	\psfrag{Oalpha}[l][l]{\tiny $\OO(N^{-\beta/2})$}		
	\includegraphics[width=0.48\textwidth]{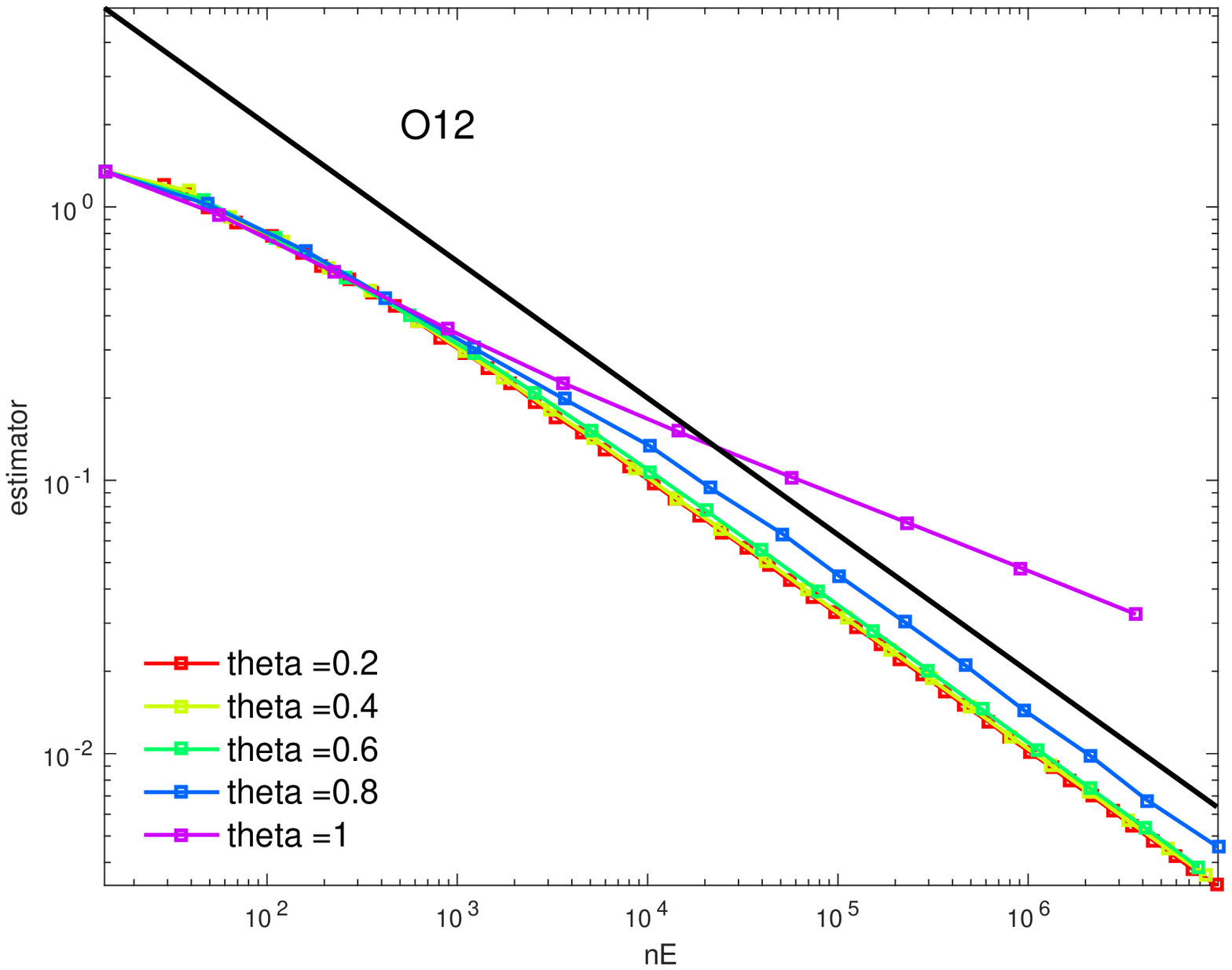}
	\hfill
	\includegraphics[width=0.48\textwidth]{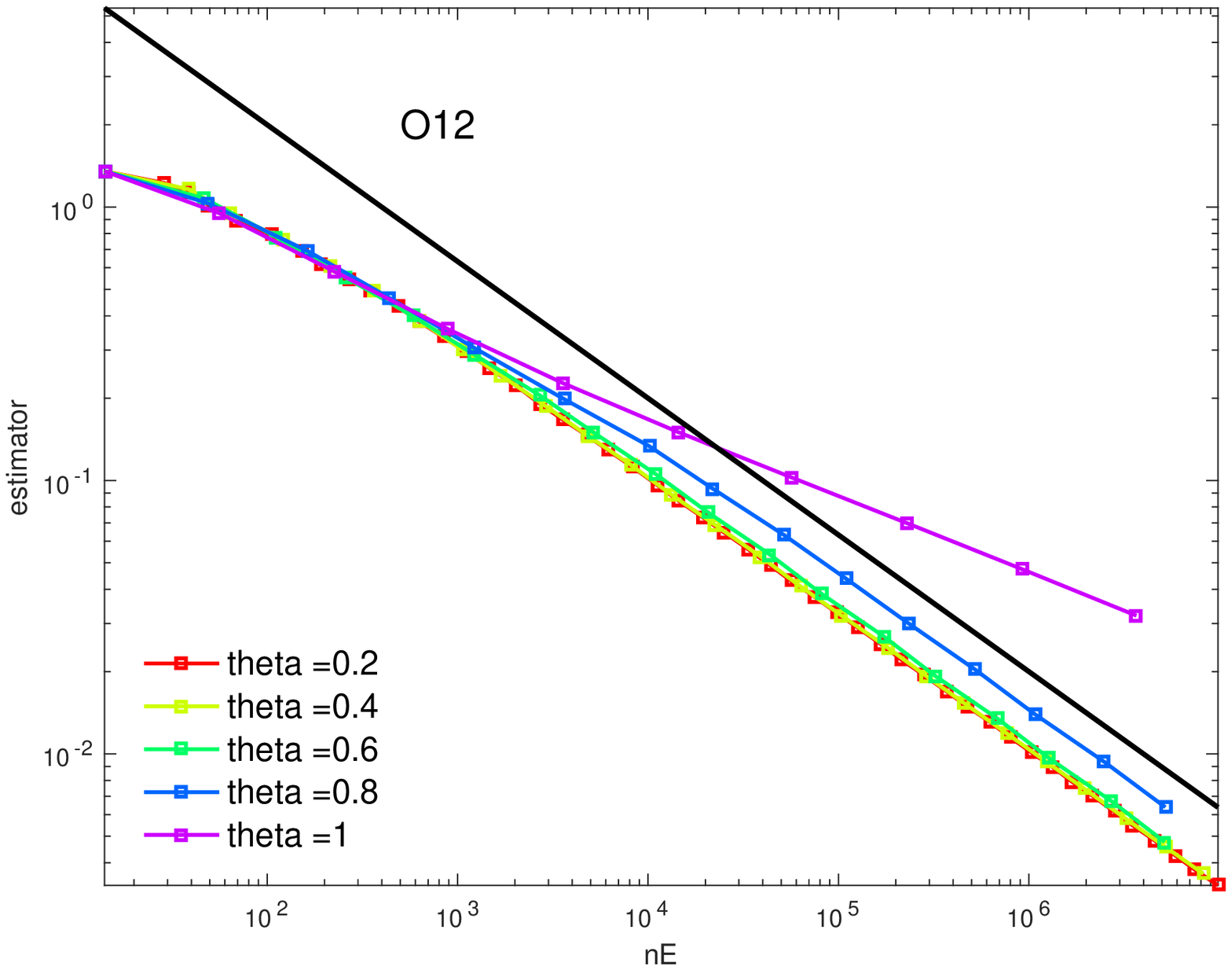}
	\caption{Experiment with unknown solution from Section~\ref{section:example2}: Convergence of $\eta_\ell(u_\ell^\inexact)$ for $\lambda = 10^{-5}$ and $\theta\in\{0.2,\dots, 1\}$, where we compare naive initial guesses $u_\ell^0 := 0$ (left) and nested iteration $u_\ell^0 := u_{\ell-1}^\inexact$ (right).\vspace*{2mm}}
  \label{fig:ex2:compare_theta_2}
\end{figure}
\begin{figure}
	\centering
	\psfrag{nE}[c][c]{\tiny number of elements $N$}
	\psfrag{estimator}[c][c]{\tiny   estimator}
	\psfrag{lambda =1}{\tiny $\lambda = 1$}
	\psfrag{lambda =0.1}{\tiny $\lambda = 10^{-1}$}
	\psfrag{lambda =0.01}{\tiny $\lambda = 10^{-2}$}
	\psfrag{lambda =0.001}{\tiny $\lambda = 10^{-3}$}
	\psfrag{lambda =0.0001}{\tiny $\lambda = 10^{-4}$}
	\psfrag{lambda =1e-05}{\tiny $\lambda = 10^{-5}$}
	\psfrag{lambda =1e-06}{\tiny $\lambda = 10^{-6}$}
	\psfrag{O12}[c][c]{\tiny $\OO(N^{-1/2})$}
	\includegraphics[width=0.48\textwidth]{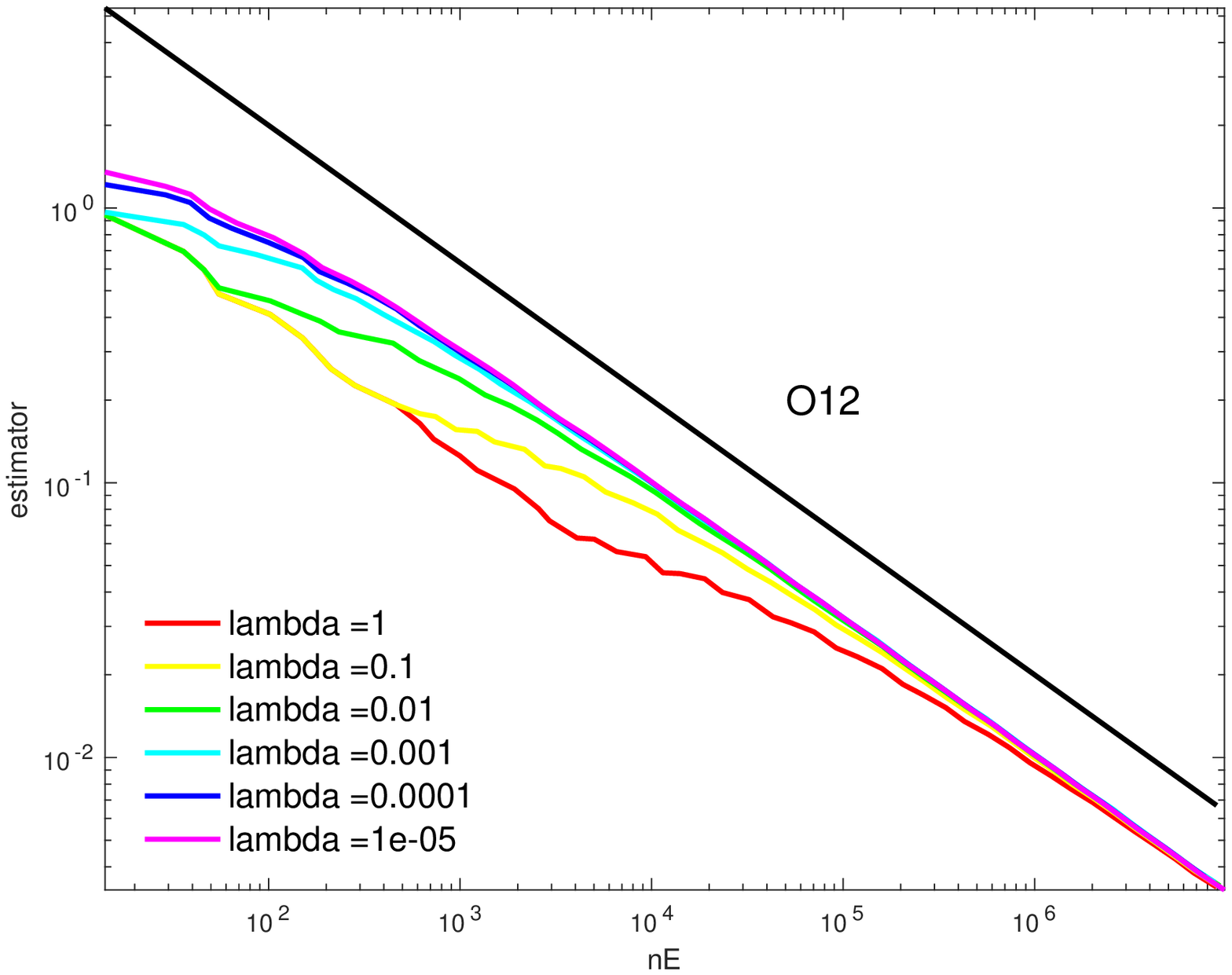}
	\hfill
	\includegraphics[width=0.48\textwidth]{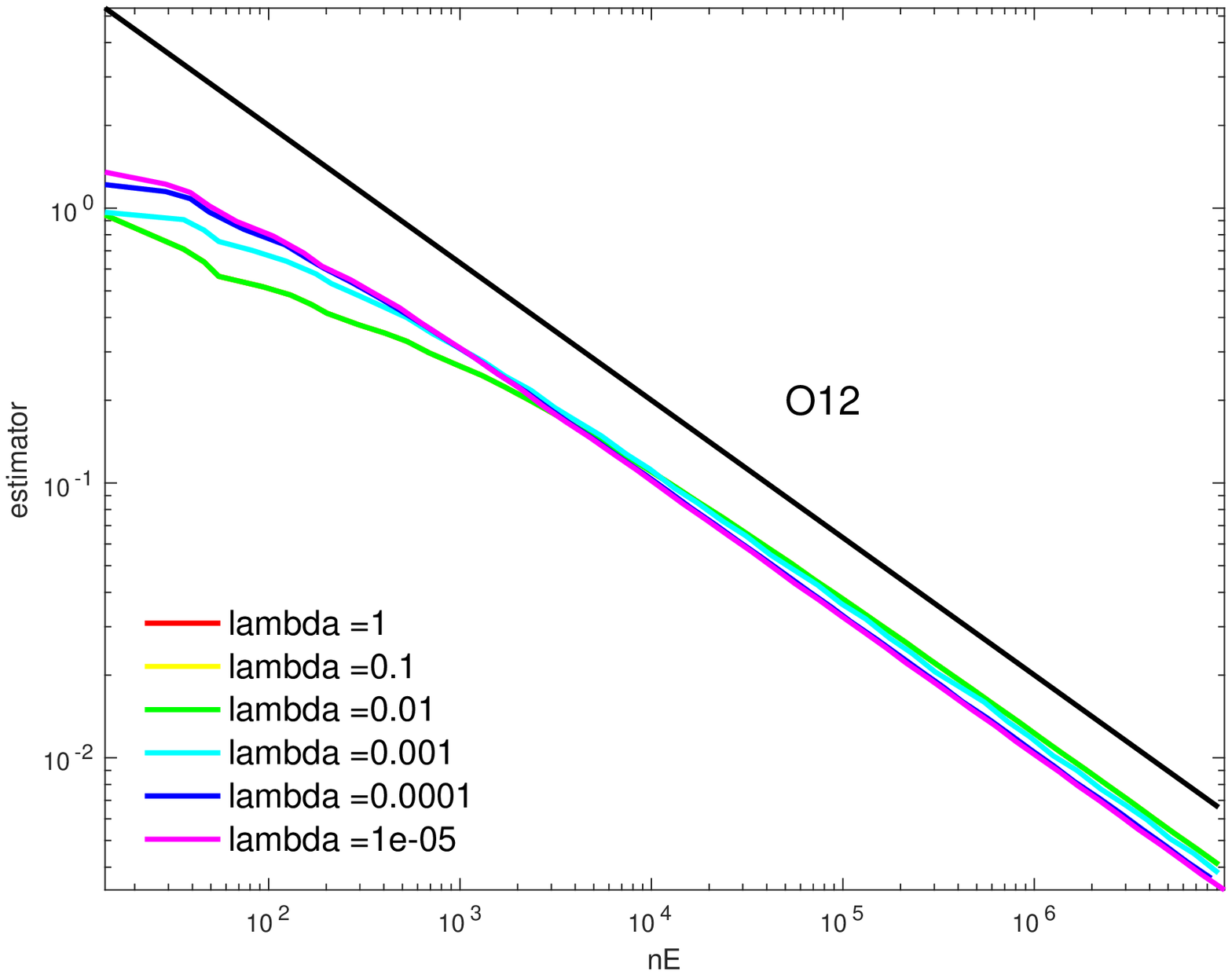}
	\caption{Experiment with unknown solution from Section~\ref{section:example2}:  Convergence of $\eta_\ell(u_\ell^\inexact)$ 
		for $\theta = 0.2$ and  $\lambda\in\{1,0.1, 0.01,\dots, 10^{-5} \}$, where we compare naive initial guesses $u_\ell^0 := 0$ (left) and nested iteration $u_\ell^0 := u_{\ell-1}^\inexact$ (right).\vspace*{2mm}}
	\label{fig:ex2:compare_lambda_1}
\end{figure}

\begin{figure}
	\centering
	\psfrag{nE}[c][c]{\tiny number of elements $N$}
	\psfrag{FP-steps}[c][c]{\tiny  number of Picard iterations}
	\psfrag{lambda=1}{\tiny $\lambda = 1$}
	\psfrag{lambda=0.1}{\tiny $\lambda = 10^{-1}$}
	\psfrag{lambda=0.01}{\tiny $\lambda = 10^{-2}$}
	\psfrag{lambda=0.001}{\tiny $\lambda = 10^{-3}$}
	\psfrag{lambda=0.0001}{\tiny $\lambda = 10^{-4}$}
	\psfrag{lambda=1e-05}{\tiny $\lambda = 10^{-5}$}
	\psfrag{lambda=1e-06}{\tiny $\lambda = 10^{-6}$}
	\psfrag{Olog}[c][c]{\tiny $\OO(\log(N))$}
	\includegraphics[width=0.46\textwidth]{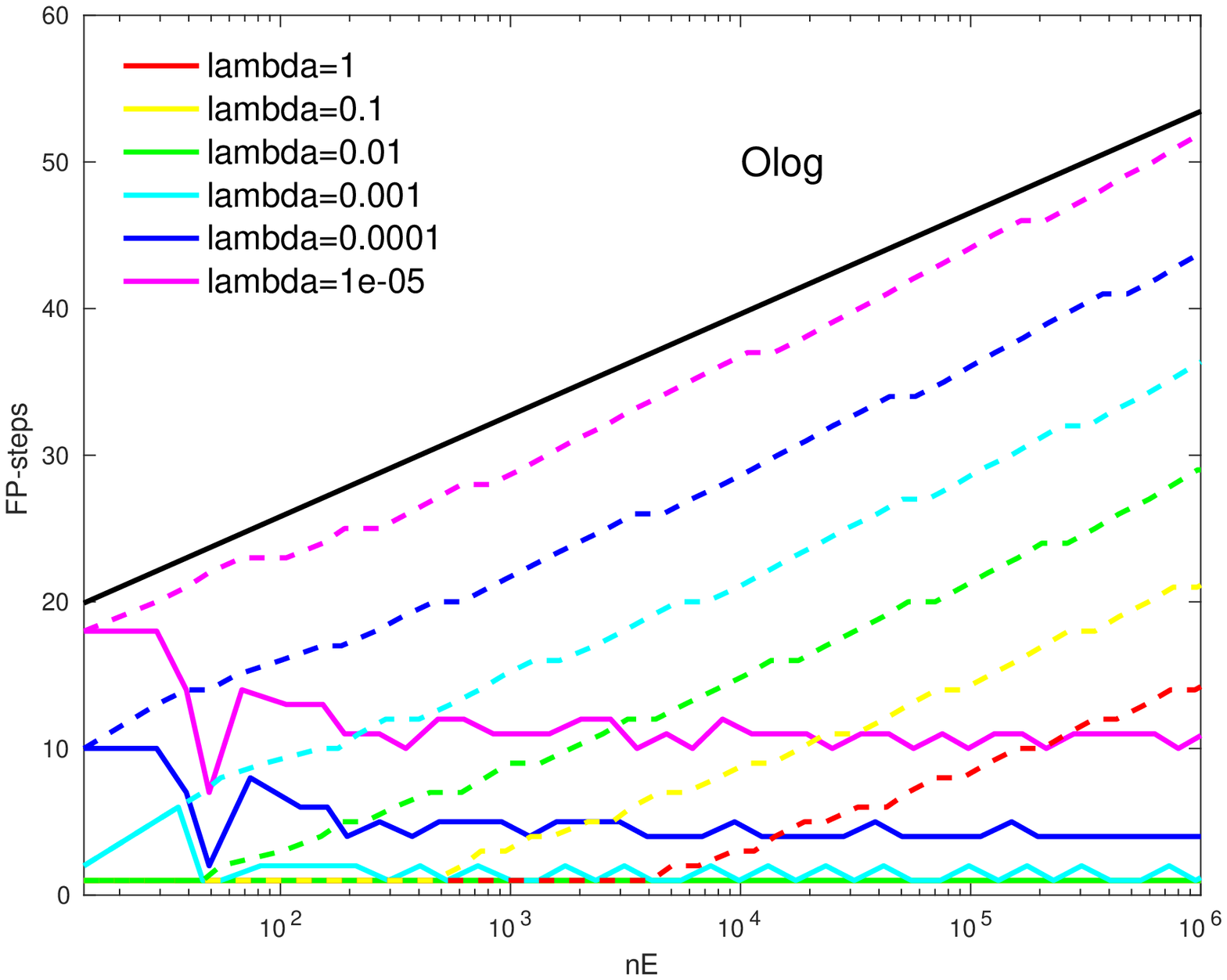}
	\hfill
	\includegraphics[width=0.46\textwidth]{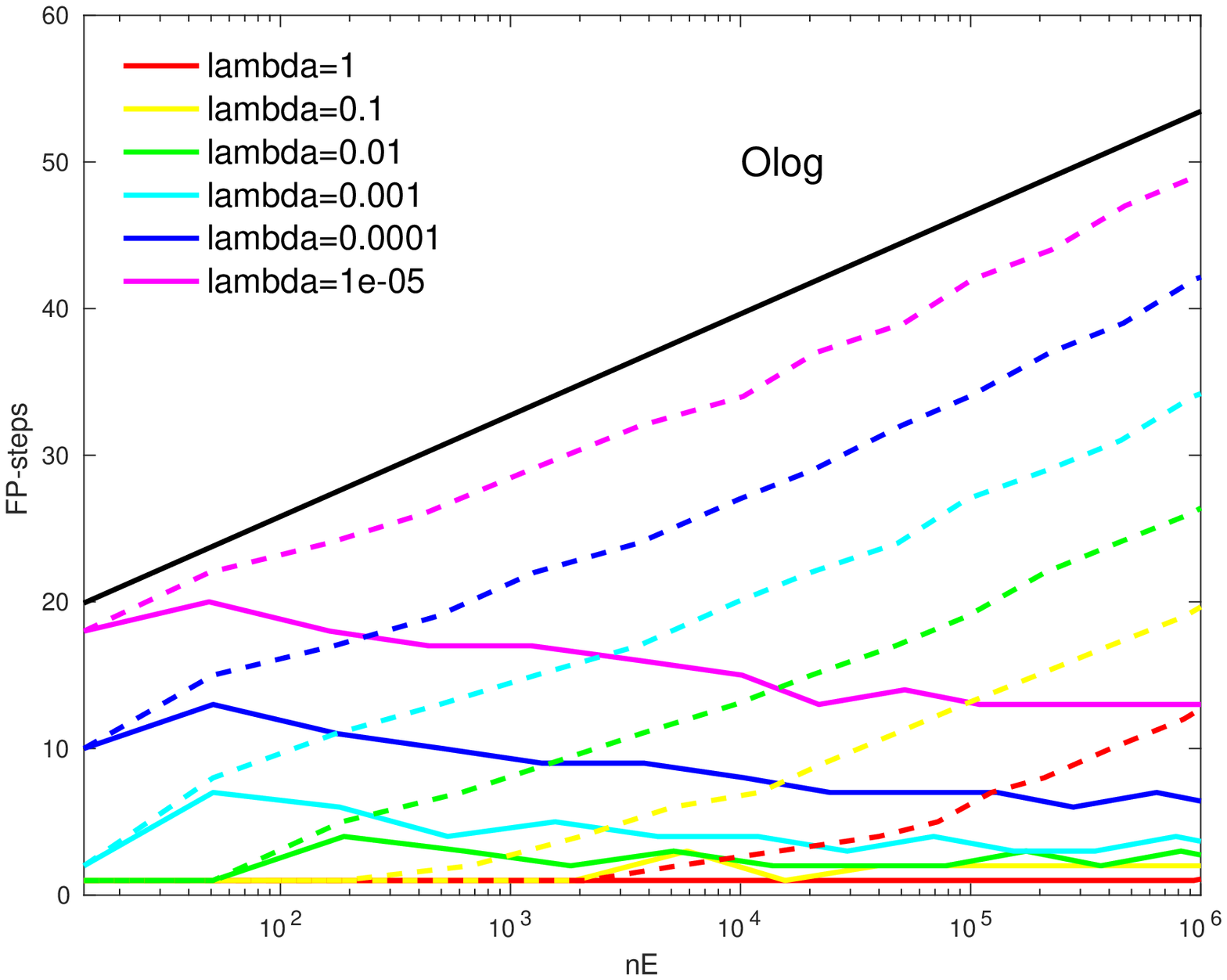}
	\caption{Experiment with unknown solution from Section~\ref{section:example2}: Number of Picard iterations used in each step of Algorithm~\ref{algorithm} for different values of $\lambda\in\{1, 0.1,\dots, 10^{-5} \}$ and  $\theta = 0.2$ (left) as well as $\theta = 0.8$ (right). As expected, we observe logarithmic growth for naive initial guesses $u_\ell^0 := 0$ (dashed lines) as well as a bounded number of Picard iterations for nested iteration  $u_\ell^0 := u_{\ell-1}^\inexact$ (solid lines).\vspace*{1mm}}
	\label{fig:ex2:compare_fpsteps}
\end{figure}


\subsection{Experiment with known solution}%
\label{section:example1}
We consider the Z-shaped domain $\Omega \subset \R^2$ from Figure~\ref{fig:ex1:mesh} (left) with mixed Dirichlet-Neumann boundary and the nonlinear problem~\eqref{eq:example}, where $\mu(x,|\nabla u^\exact(x)|^2) := 2+ \frac{1}{\sqrt{1+|\nabla u^\exact(x)|^2 }}$. 
This choice of $\mu$ leads to $\alpha=2$ and $L=3$ in~{(O1)--(O2)}.
We prescribe the solution $u^\exact$ in polar coordinates by
\begin{align}
	u^\exact(x,y) = r^{\beta} \cos\Big(\beta \, \phi \Big),
\end{align} 
with $\beta = 4/7$ and compute $f$ and $g$ in~\eqref{eq:example} accordingly. We note that $u^\exact$ has a generic singularity at the reentrant corner $(x,y)=(0, 0)$. 

Our empirical observations are the following: Due to the singular behavior of $u^\star$, uniform refinement leads to a reduced convergence rate $\OO(N^{-\beta/2})$ for both, the energy error $\norm{\nabla u^\exact - \nabla u_\ell^\inexact}{L^2(\Omega)}$ as well as the error estimator $\eta_\ell (u_\ell^\inexact)$. 
On the other hand, the adaptive refinement of Algorithm~\ref{algorithm} regains the optimal convergence rate $\OO(N^{-1/2})$, independently of the actual choice of $\theta \in \{0.2, 0.4, 0.6, 0.8\} $ and $\lambda \in \{1, 0.1, 0.01, \ldots, 10^{-6}\}$; see Figure~\ref{fig:ex1:compare_theta_1}--\ref{fig:ex1:compare_lambda_1}. 
Throughout, we compare the performance of the iterative solver for naive initial guess $u_\ell^0:=0$ as well as nested iteration $u_\ell^0:=u_{\ell-1}$:
The larger $\lambda$, the stronger is the influence on the convergence behavior; see Figure~\ref{fig:ex1:compare_lambda_1}. Moreover, Figure~\ref{fig:ex1:compare_fpsteps} shows the number of Picard iterations for $\theta\in\{0.2, 0.8\}$ and $\lambda \in \{0.1, 0.01, \ldots, 10^{-6}\}$.
As expected from Remark~\ref{remark1}, for the naive initial guess $u_\ell^0:=0$, the number of Picard iterations grows logarithmically with the number of elements $\#\TT_\ell$, while we observe a bounded number of Picard iterations for nested iteration $u_\ell^0:=u_{\ell-1}$; cf.~Proposition~\ref{proposition:nested_iteration} resp.\ Remark~\ref{remarkXXX}. 

\subsection{Experiment with unknown solution}
\label{section:example2}
We consider the Z-shaped domain $\Omega \subset \R^2$ from Figure~\ref{fig:ex1:mesh} (left) and the nonlinear Dirichlet problem~\eqref{eq:example} with $\Gamma=\Gamma_D$ and constant right-hand side $f\equiv1$, where $\mu(x, |\nabla u^\exact|^{2}) = 1+ \arctan(|\nabla u^\exact|^2)$. According to~\cite[Example 1]{cw2015}, there holds (O1)--(O2) with $\alpha = 1$ and $L = 1+\sqrt{3}/2 + \pi/3$.

Since the exact solution is unknown, our empirical observations are concerned with the error estimator only; see Figure~\ref{fig:ex2:compare_theta_1}--\ref{fig:ex2:compare_fpsteps}. Uniform mesh-refinement leads to a suboptimal rate of convergence, while the use of Algorithm~\ref{algorithm} regains the optimal rate of convergence. The latter appears to be robust with respect to $\theta \in \{0.2,\dots, 0.8\}$ as well as $\lambda \in \{1, 0.1,\ldots ,10^{-5}\}$. While naive initial guesses $u_\ell^0:=0$ for the iterative solver lead to a logarithmic growth of the number of Picard iterations, the proposed use of nested iteration $u_\ell^0 := u_{\ell-1}$ again leads to bounded iteration numbers.



\section{Conclusions}

In this paper, we have analyzed an adaptive FEM proposed in~\cite{cw2015} for the numerical solution of quasi-linear PDEs with strongly monotone operators. Conceptually based on the residual error estimator and the D\"orfler marking strategy, the algorithm steers the adaptive mesh-refinement as well as the iterative solution of the arising nonlinear systems by means of a simple fixed point iteration which is adaptively stopped if the fixed point iterates are sufficiently accurate. In the spirit of~\cite{axioms}, the numerical analysis is given in an abstract Hilbert space setting which (at least) covers conforming first-order FEM and certain scalar nonlinearities~\cite{gmz}.

We prove that our adaptive algorithm guarantees convergence of the FE solutions to the unique solution of the PDE at optimal algebraic rate for the error estimator (which, in usual applications, is equivalent to energy error plus data oscillations~\cite{gmz}). Employing nested iterations, we prove that the number of fixed point iterations per mesh is bounded logarithmically with respect to the improvement of the error estimator. As a consequence, we thus prove that the adaptive algorithm is not only convergent at optimal rate with respect to the degrees of freedom, but also at (almost) optimal rate with respect to the computational work.

Numerical experiments for quasi-linear PDEs in $H^1_0(\Omega)$ confirm our theory and underline the performance of the proposed adaptive strategy, where each step of the considered fixed point iteration requires only the numerical solution of one linear Poisson problem. In the present work, the iterative and inexact solution of these linear problems is not considered, but it can be included into the analysis~\cite{haberlphd}.

Overall, the present work appears to be the first which guarantees optimal convergence for an adaptive FEM with iterative solver for nonlinear PDEs. Open questions for future research include the following: Is it possible to generalize the numerical analysis to other type of error estimators (e.g., estimators based on equilibrated fluxes~\cite{ev13}) as well as higher-order and/or non-conforming FEM (see, e.g.,~\cite{cw2015} for numerical experiments)? Is it possible to treat non-scalar nonlinearities?

\medskip 
\noindent

\thanks{{\bf Acknowledgements.} 
The authors acknowledge support of the the Austria Science Fund (FWF) through the research project \emph{Optimal adaptivity for BEM and FEM-BEM coupling} under grant P27005 (AH, DP), and the research project \emph{Optimal isogeometric boundary element methods} under grant P29096 (DP, GG). In addition, DP and GG are supported through the FWF doctoral school \emph{Nonlinear PDEs} funded under grant W1245. Moreover, BS and DP acknowledge support of the Vienna Science and Technology Fund (WWTF) through the research project \emph{Thermally controlled magnetization dynamics} under grant MA14-44. }


\bibliographystyle{alpha}
\bibliography{literature}

\end{document}